\documentclass[12pt]{article}
\usepackage{amsmath,amsthm,amssymb,amsfonts,bm}
\usepackage{graphicx,adjustbox}

\usepackage{enumitem}
\usepackage[round]{natbib}
\usepackage{mathtools}
\usepackage{tikz}
\usepackage{multirow}
\usepackage{tabularx}
\usepackage{booktabs}
\usepackage{float}
\usepackage{caption} 
\usepackage{adjustbox}
\usepackage{subcaption}
\usepackage{mathrsfs}
\usepackage{bbm}
\usepackage{chapterbib}
\usepackage{authblk}
\usepackage[title]{appendix}
\usepackage{hyperref}

\makeatletter
\newcommand*{\rom}[1]{\expandafter\@slowromancap\romannumeral #1@}

\newcolumntype{Y}{>{\centering\arraybackslash}m{.5in}}
\newcolumntype{A}{>{\centering\arraybackslash}m{.8in}}
\newcolumntype{Z}{>{\centering\arraybackslash}X}

\renewcommand{\le}{\leqslant}
\renewcommand{\ge}{\geqslant}
\DeclareMathOperator*{\argmin}{arg\,min}
\DeclareMathOperator*{\argmax}{arg\,max}
\DeclareMathOperator{\diag}{diag}
\DeclareMathOperator{\op}{op}

\newcommand{\E}{\mathbb{E}}
\newcommand{\F}{\mathcal{F}}
\newcommand{\G}{\mathbb{G}}
\newcommand{\R}{\mathbb{R}}
\renewcommand{\P}{\mathbb{P}}
\newcommand{\vx}{\bm{x}}
\newcommand{\vX}{\bm{X}}
\newcommand{\cX}{\mathcal{X}}
\newcommand{\cF}{\mathcal{F}}
\newcommand{\cS}{\mathcal{S}}
\newcommand{\cG}{\mathcal{G}}
\newcommand{\cD}{\mathcal{D}}
\newcommand{\bS}{\mathbb{S}}

\theoremstyle{plain}% default
\newtheorem{theorem}{Theorem}
\newtheorem{lemma}{Lemma}
\newtheorem{corollary}{Corollary}
\newtheorem{assumption}{Assumption}
\newtheorem{proposition}{Proposition}

\theoremstyle{definition}

\newtheorem{example}{Example}
\newtheorem{remark}{Remark}

\newcommand{\blind}{0}

\begin{document}

\def\spacingset#1{\renewcommand{\baselinestretch}%
{#1}\small\normalsize} \spacingset{1}

%%%%%%%%%%%%%%%%%%%%%%%%%%%%%%%%%%%%%%%%%%%%%%%%%%%%%%%%%%%%%%%%%%%%%%%%%%%%%%

\if0\blind
{
  \title{\bf Bootstrap inference for quantile-based modal regression}
  \author{
  	Tao Zhang \thanks{Tao Zhang is partially supported by NSF grant DMS-1952306.}\\
  	Department of Statistics and Data Science, Cornell University
  	\and
    Kengo Kato \thanks{
    	Kengo Kato is partially supported by NSF grant DMS-1952306 and DMS-2014636.}\\
    Department of Statistics and Data Science, Cornell University
    \and
    David Ruppert\\
    Department of Statistics and Data Science, Cornell University\\
    School of Operations Research and Information Engineering, Cornell University}
  \maketitle
} \fi

\if1\blind
{
  \bigskip
  \bigskip
  \bigskip
  \begin{center}
    {\LARGE\bf Bootstrap inference for quantile-based modal regression}
\end{center}
  \medskip
} \fi

\bigskip
\baselineskip=2\baselineskip
\begin{abstract}
	\baselineskip=2\baselineskip
		In this paper, we develop uniform inference methods for the conditional mode based on quantile regression. Specifically, we propose to estimate the conditional mode by minimizing the derivative of the estimated conditional quantile function defined by smoothing the linear quantile regression estimator, and 
		%develop a novel bootstrap method, which we call the pivotal bootstrap, 
		develop two bootstrap methods, a novel pivotal bootstrap and the nonparametric bootstrap,
		 for our conditional mode estimator. 
	Building on high-dimensional Gaussian approximation techniques, we establish the validity of  simultaneous confidence rectangles constructed from the two bootstrap methods for the conditional mode. %\textcolor{blue}{We additionally prove the validity of simultaneous confidence rectangles based on conventional nonparametric bootstrap.}
	We also extend the preceding analysis to the case where the dimension of the covariate vector is increasing with the sample size. 
	Finally, we conduct simulation experiments and a real data analysis using U.S. wage data  to demonstrate the finite sample performance of our inference method. The supplemental materials include the wage dataset, R codes and an appendix containing proofs of the main results, additional simulation results, discussion of model misspecification and quantile crossing, and additional details of the numerical implementation.
\end{abstract}

\noindent%
\vfill
\newpage

\setcounter{section}{0} %***
\setcounter{equation}{0} %-1

\spacingset{1.5} % DON'T change the spacing!
\section{Introduction}
\subsection{Overview}
Modal regression is a principal statistical methodology to estimate and make inference on the conditional mode. 
Modes provide useful distributional information missed by the mean when the (conditional) distribution is skewed \citep{chen2016nonparametric} and are known to be robust  under measurement errors (\citealt{bound1991extent,hu2008instrumental}).  The global mode  offers intuitive interpretability by being understood as  ``the most likely"  or ``the most common" (\citealt{heckman2001molecular,hedges2003comparison}). As such, modal regression has wide applications in various areas including astronomy \citep{bamford2008revealing}, medical research \citep{wang2017regularized}, econometrics \citep{kemp2012regression}, etc. We refer the reader to \cite{chacon2018} and  \cite{chen2018modal} for recent reviews on modal regression; see also a literature review below.

In this paper, we consider estimating the conditional mode by ``inverting" a quantile
regression model, which builds on the observation that the derivative of the conditional quantile
function coincides with the reciprocal of the conditional density so that the conditional
mode can be obtained by minimizing the derivative of the conditional quantile function.
Specifically, we estimate the conditional mode by minimizing the derivative of the kernel smoothed
Koenker-Bassett estimator of the conditional quantile function \citep{koenker1978regression} with a sufficiently smooth kernel. We develop asymptotic theory for the proposed estimator $\hat{m}(\vx)$ of the conditional
mode $m(\vx)$. In particular, we consider
simultaneous confidence intervals for the conditional mode at multiple design points,  $m(\vx_{1}),\dots,m(\vx_{L})$, where
$L$ is allowed to grow with the sample size $n$, i.e., $L = L_{n} \to \infty$. To this end, we first show that 
$\hat{m}(\vx) - m(\vx)$ can be approximated by the linear term $(nh^{3/2})^{-1}\sum_{i=1}^{n} \psi_{\vx}(U_{i},\vX_{i})$ uniformly over a range of
design points $\vx$, where $h=h_{n} \to 0$ is a sequence of bandwidths,  $\psi_{\vx}$ is the influence function
(that depends on $n$) at design point $\vx$, $\vX_1,\dots,\vX_{n}$ are independent covariate vectors, and $U_1,\dots,U_{n}$ are mutually independent
uniform random variables on $(0,1)$ independent of the covariate vectors. Building on high
dimensional Gaussian approximation techniques developed in \cite{chernozhukov2014gaussian,chernozhukov2017central}, we show that $\sqrt{nh^{3}}(\hat{m}(\vx_{\ell}) - m(\vx_{\ell}))_{\ell=1}^{L}$ can be approximated by an $L$-dimensional
Gaussian vector uniformly over the hyperrectangles in $\R^{L}$, i.e., all sets $A$ of the form:  
$A=\{w \in \mathbb{R}^L\ :\ a_j\le w_j\le b_j\text{ for all $j=1,\dots,L$}\}$ for some $-\infty\le a_j\le b_j\le \infty$, $j=1,\dots, L$,
even when $L \gg n$.

As the limiting Gaussian distribution is infeasible in practice, we consider two bootstrap methods, the nonparametric bootstrap and a novel pivotal bootstrap, to conduct valid inference. We first discuss the motivation of the new pivotal bootstrap.
The leading stochastic term in the prescribed expansion is conditionally ``pivotal" in
the sense that conditionally on $\vX_{1},\dots,\vX_{n}$, the distribution of the process
\[
\vx \mapsto (nh^{3/2})^{-1} \sum_{i=1}^{n} \psi_{\vx}(U_{i},\vX_{i})
\]
is completely known up to some nuisance parameters. This suggests a version of bootstrap for the proposed estimator by sampling
uniform random variables $U_{i}$ independent of the data. In practice, the influence function $\psi_{\vx}$ depends on nuisance parameters and we replace them by consistent estimates. We call the resulting
bootstrap ``pivotal bootstrap" and prove that the pivotal bootstrap can consistently
estimate the sampling distribution of
$\sqrt{nh^{3}}(\hat{m}(\vx_{\ell}) - m(\vx_{\ell}))_{\ell=1}^{L}$ uniformly over the rectangles in $\R^L$
even when $L \gg n$. In fact, our inference framework is more general and covers simultaneous inference for linear combinations of the vector $(m(\vx_{\ell}))_{\ell=1}^{L}$, which can be used to construct simultaneous confidence intervals for partial effects and test significance of certain covariates on the conditional mode. We also establish a similar consistency result for the nonparametric bootstrap. Finally, we extend the previous analysis to the case where the dimension
of the covariate vector increases with the sample size.

We conduct simulation experiments on various mode inference problems and a real data analysis to demonstrate the finite sample performance of the bootstrap methods. 
Our simulation experiments show that the pivotal bootstrap yields accurate pointwise and simultaneous confidence intervals for the conditional modes. %In addition to those simulation experiments,
Additionally, we apply our inference method to analyze a real U.S. wage dataset.  Analysis of wage data is important in econometric and social science (\citealt{autor2008trends,western2011unions,buchinsky1994changes}). Wage data are often positively skewed and ``the most common wage"  as a representative of the majority of the population is usually of more interest. Common questions in the analysis of wage data include: 
i) What is the most likely wage for given covariates? How to construct pointwise and simultaneous confidence intervals for the estimated wages? 
ii) Is there an effect of a specific covariate on the most likely wage given the same other covariates?  
We address those empirical questions using the inference method developed in the present paper. 

From a technical perspective, the asymptotic analysis in this paper is highly nontrivial. Our program of the technical analysis proceeds as 1) first establishing a uniform asymptotic representation and 2) high-dimensional Gaussian approximation to our estimate, and 3) then proving the validity of the pivotal and nonparametric bootstraps building on 1) and 2). Each of these steps relies  on modern empirical process theory and high-dimensional Gaussian approximation techniques recently developed by \cite{chernozhukov2014gaussian, chernozhukov2017central}. In particular, the pivotal bootstrap differs from the nonparametric or multiplier bootstraps that have been analyzed in the literature in the high-dimensional setup \citep{belloni2019conditional,chernozhukov2016empirical,deng2017beyond,chen2019jackknife}, and proving the validity of the pivotal bootstrap requires a substantial work. Further, we employ a new multiplier inequality for the empirical process in \cite{han2019convergence} to establish the validity of the nonparametric bootstrap. 
%Our approach can be applied to more general multiplier bootstrap methods and therefore, may be of independent interest.}

In summary, the present paper contributes to the literature on modal regression in twofold. First, we propose a new quantile-based conditional mode estimate that enjoys both desirable computational and statistical guarantees.  The proposed estimator only requires solving a linear quantile regression problem and a one-dimensional optimization both of which can be solved efficiently.
Second, we establish the theoretical validity of two bootstrap methods for a broad spectrum of inference tasks in a unified way. In particular, we propose a new resampling method (pivotal bootstrap) that builds on an insight into the specific structure of our estimate.
%we propose a new resampling method (pivotal bootstrap) that builds on an insight into the specific structure of our estimate, and establish theoretical validity of the pivotal bootstrap for a broad spectrum of inference tasks in a unified way.

\subsection{Literature review}
Starting from the pioneering work of  \cite{sager1982maximum},
there is now a large literature on modal regression.
There are two major approaches to estimating the conditional mode comparable to our method; one is linear modal regression  where the conditional mode is assumed to be linear in covariates \citep{lee1989mode,lee1993,kemp2012regression,yao2014new}, and the other is nonparametric estimation \citep{yao2012,chen2016nonparametric,yao2016nonparametric,feng2020statistical}; see also \cite{lee1998semiparametric,manski1991regression,einbeck2006modelling,sasaki2016,ho2017,khardani2017,krief2017semi} for alternative methods including semiparametric and Bayesian estimation. \cite{lee1989mode,lee1993} assume symmetry of the error distribution to derive limit theorems for their proposed estimators, but the symmetry assumption implies that the conditional mean, median, and mode coincide, thereby significantly reducing the complexity of estimating the conditional mode. \cite{kemp2012regression} and \cite{yao2014new} consider an alternative estimator defined by minimizing a kernel-based loss function for linear modal regression and develop limit distribution theory for the estimator without assuming symmetry of the error distribution. 
However, the optimization problem of \cite{kemp2012regression} and \cite{yao2014new} is (multidimensional and) nonconvex, and while they propose EM-type algorithms to compute their estimators, ``there is no guarantee that the algorithm will converge to the global optimal solution'' \citep[][p. 659]{yao2014new}. Compared with the method of \cite{kemp2012regression} and \cite{yao2014new}, 
%both methods (ours and theirs) 
all three methods (including ours) enjoy the same rate of convergence, while our method is computationally attractive since linear quantile regression can be formulated as a linear programming problem \citep{koenker2005quantile}, and minimizing the estimated derivative of the conditional quantile function is a one-dimensional optimization problem both of which can be solved accurately and efficiently. %and carried out by a grid search. 

\cite{yao2012} consider local linear estimation of the conditional mode but their Condition (A6) is essentially the symmetry assumption on the error distribution, which makes their problem statistically equivalent to conditional mean estimation.
\cite{chen2016nonparametric} study nonparametric estimation of the conditional mode based on kernel density estimation (KDE), and develop nonparametric bootstrap inference for their KDE-based estimate. The nonparametric estimation is able to avoid model misspecification. \cite{chen2016nonparametric} also allow for multiple local modes, while we assume the existence of the unique global mode at each design point of interest. Thus, the setup of \cite{chen2016nonparametric} is more general than ours. However, the convergence rate of the KDE-based estimate of \cite{chen2016nonparametric} is slow even when the dimension of the covariate vector is moderately large (``curse of dimensionality"). Specifically, the convergence rate of the \cite{chen2016nonparametric} estimate is at best $n^{-2/(p+7)}$ where $p$ is the number of continuous covariates under the assumption of four times differentiability of the conditional density, while our estimate can achieve the  $n^{-2/7}$ rate (up to logarithmic factors when evaluated under the uniform norm) assuming three times differentiability of the conditional density (albeit assuming a linear quantile regression model). Finally, \cite{chen2016nonparametric} also consider the application of the nonparametric bootstrap to inference on the conditional mode. However, our estimator is substantially different from their estimator and requires different analysis to establish the validity of the nonparametric bootstrap.

The present paper builds on (but substantially differs from) the recent work of \cite{ohta2018quantile}, which proposes a different quantile-based estimate of the conditional mode and develops pointwise limit distribution theory for their estimator. Contrary to ours, \cite{ohta2018quantile} directly use the linear quantile regression estimate and minimize  its difference quotient  (as the linear quantile regression estimate is not smooth in the quantile index), which makes a substantial difference between their asymptotic analysis and ours. 
Indeed, \cite{ohta2018quantile} show that the rate of convergence of their estimate is at best $n^{-1/4}$ that is slower than our $n^{-2/7}$ rate, and find that the pointwise limit distribution  is a scale transformation of nonstandard Chernoff's distribution. The nonstandard limit distribution poses a substantial challenge in inference using their estimate and \cite{ohta2018quantile} only consider pointwise inference using a general purpose subsampling method \citep{politis1999}. We overcome this limitation by employing kernel smoothing, and further, develop a model-based bootstrap method (pivotal bootstrap) that enables us to deal with much broader inference tasks including simultaneous confidence intervals and significance testing.

This paper also builds on and contributes to the quantile regression literature. Quantile regression provides a comparatively full picture of how the covariates impact the conditional distribution of a response variable and has wide applications \citep{koenker2017quantile}.  In particular, the pivotal bootstrap of the present paper is related to  \cite{parzen1994, chernozhukov2009, he2017resampling, belloni2019conditional} who study resampling-based inference methods that build on (conditionally) pivotal influence functions in the quantile regression setup. Their scopes and methods are, however, substantially different from ours. To the best of our knowledge, exploiting pivotal influence functions to make inference for modal regression is new. 

%From a technical perspective, the asymptotic analysis in this paper is highly nontrivial. Our program of the technical analysis proceeds as 1) first establishing a uniform asymptotic representation and 2) high-dimensional Gaussian approximation to our estimate, and 3) then proving the validity of the pivotal bootstrap building on 1) and 2). Each of these steps relies  on modern empirical process theory and high-dimensional Gaussian approximation techniques recently developed by \citep{chernozhukov2014gaussian, chernozhukov2017central}. In particular, the pivotal bootstrap differs from the nonparametric or multiplier bootstraps that have been analyzed in the literature in the high-dimensional setup, and proving the validity of the pivotal bootstrap requires a substantial work.

%consistency, are built on fairly new probabilistic tools and require substantial work. We emphasize that the above theoretical results are not direct applications of the existing results and many significant modifications are made to adapt to the current modal inference setting. For example, to show our Gaussian approximation result, we develop new techniques based on Gaussian anti-concentration to deal the residue term of the asymptotic representation while applying the high dimensional central limit theorem \citep{chernozhukov2017central} to the leading term. 

\subsection{Organization}
The rest of the paper is organized as follows. In Section \ref{sec: setup}, we introduce the setup and define the proposed quantile-based modal estimator. In Section \ref{sec: main}, we present the main theoretical results for the proposed estimator. We first derive a uniform asymptotic linear representation for the proposed estimator. Then we present our general inference framework based on the pivotal and nonparametric bootstraps together with their theoretical guarantees.  In Section \ref{sec: numerical}, we present the simulation results and a real data example. 
In Section \ref{sec: extension}, we extend the preceding analysis to the increasing dimension case. 
Finally, we summarize the paper in Section \ref{sec: summary}. 
The proofs of main results and additional discussion are relegated to the Appendix which is included in the supplemental materials.

\section{Mode estimation via smoothed quantile regression}
\label{sec: setup}
We begin with the setup and define our estimator. 
We are interested in making inference on the conditional mode of a scalar response variable $Y\in \R$ given a $d$-dimensional covariate vector $\vX \in \R^d$. 
We will initially assume that the dimension $d$ is fixed in Section \ref{sec: main}, but consider the extension to the case with $d = d_n \to \infty$ in Section \ref{sec: extension}. In what follows, we assume that there exists a conditional density of $Y$ given $\vX$, $f(y \mid \vx)$, which is (at least) continuous in $y$ for each design point $\vx$. 
We are interested in making inference on the conditional mode over a compact subset $\cX_{0}$ of the support of $\vX$. We assume that  for each $\vx \in \cX_{0}$, there exists a unique global mode $m(\vx)$, i.e.,  $m(\vx)$ is the unique maximizer of the function $y \mapsto f(y \mid \vx)$,
\begin{equation}
m(\vx)=\argmax_{y\in \R}f(y \mid \vx).
\label{eq: mode}
\end{equation}

Our  strategy to estimate the conditional mode is based on ``inverting" a quantile regression model.  For $\tau\in (0,1)$, let $Q_{\vx}(\tau)$ denote the conditional $\tau$-quantile of $Y$ given $\vX$. Observe that the derivative of the conditional quantile function with respect to the quantile index $\tau$ coincides with the reciprocal of the conditional density at $Q_{\vx}(\tau)$, i.e., 
%By definition, the following equation holds
%\setlength{\abovedisplayskip}{7pt}%
%\setlength{\belowdisplayskip}{7pt}%
%\setlength{\abovedisplayshortskip}{0pt}%
%\setlength{\belowdisplayshortskip}{0pt}%
%\begin{equation}\label{qr}
%F(Q_{\vx}(\tau)\mid \vx)=\tau.
%\end{equation}
\begin{equation}\label{qr2}
s_{\vx}(\tau):=Q_{\vx}'(\tau) := \frac{\partial Q_{\vx}(\tau)}{\partial \tau} = \frac{1}{f(Q_{\vx}(\tau) \mid \vx)}. 
\end{equation}
This suggests that the conditional mode $m(\vx)$ can be obtained by minimizing the ``sparsity" function $s_{\vx}(\tau) := Q_{\vx}'(\tau)$. 
Specifically, let $\tau_{\vx}$ denote the minimizer of $s_{\vx}(\cdot)$, i.e., 
\[
\tau_{\vx} = \argmin_{\tau \in (0,1)} s_{\vx}(\tau).
\]
Then, we arrive at the expression $m(\vx) = Q_{\vx}(\tau_{\vx})$.
Hence, estimation of $m(\vx)$ reduces to  estimation of $Q_{\vx}(\cdot)$ and $\tau_{\vx}$.  

To estimate the conditional quantile function, 
we assume a linear quantile model, i.e., 
\[
Q_{\vx}(\tau)=\vx^T\beta(\tau), \quad \tau \in (0,1). 
\]
Suppose that we are given i.i.d.\ observations $(Y_{1},\vX_1),\dots,(Y_{n},\vX_n)$ of $(Y,\vX)$. We estimate the slope vector $\beta (\tau)$ by the standard quantile regression estimator \citep{koenker1978regression},
\begin{equation}\label{qr3}
\hat{\beta}(\tau)=\argmin_{\beta \in \R^d}\sum_{i=1}^{n}\rho_{\tau}(Y_i-\vX_i^T\beta),
\end{equation}
where $\rho_{\tau}(u)=u\left\{\tau-I(u\le 0) \right\}$ is the check function. However, the plug-in estimator $\check{Q}_{\vx}(\tau):=\vx^T\hat{\beta}(\tau)$ for the conditional quantile function is not smooth in $\tau$. 
To overcome this difficulty, we propose to smooth the naive estimator $\check{Q}_{\vx}(\tau)$ by a kernel function, and estimate $\tau_{\vx}$ by minimizing the derivative of the smoothed quantile estimator. To this end, let $K:\R \to \R$ be a kernel function (a function that integrates to $1$) that is smooth and supported in $[-1,1]$ (see Assumption 1 (vii) in the following for more details).  
For a given sequence of bandwidth parameters $h=h_{n} \to 0$, we modify the naive estimator $\check{Q}_{\vx}(\tau)$ by 
\[
\hat{Q}_{\vx}(\tau):=\int_{\tau-h}^{\tau+h} \check{Q}_{\vx}(t)K_h(\tau-t)dt, \  \tau \in [\epsilon,1-\epsilon],
\]
where $K_{h}(\cdot) := h^{-1}K(\cdot/h)$ and $\epsilon \in (0,1/2)$ is some small user-chosen parameter. The restriction of the range of $\tau$ is to avoid the boundary problem.  Since $K$ is supported in $[-1,1]$, the integral $\int_{\tau-h}^{\tau+h}$ above can be formally replaced by $\int_{\R}$ with the convention that $\check{Q}_{\vx}(t) = 0$ for $t \notin (0,1)$. 

Then, we can estimate $s_{\vx}(\tau)$ by differentiating $\hat{Q}_{\vx}(\tau)$, $\hat{s}_{\vx}(\tau) := \hat{Q}_{\vx}'(\tau)$, 
and estimate $\tau_{\vx}$ by minimizing $\hat{s}_{\vx}(\tau)$,
\[
\hat{\tau}_{\vx}:= \argmin_{\tau\in [\epsilon,1-\epsilon]} \hat{s}_{\vx}(\tau). 
\]
By the smoothness of $K(\cdot)$,  the map $\tau \mapsto \hat{s}_{\vx}(\tau)$ is smooth, so $\hat{\tau}_{\vx}$ is guaranteed to exist by compactness of $[\epsilon,1-\epsilon]$. 
Finally, we propose to estimate the conditional mode $m(\vx)$ by a plug-in method:
\[
\hat{m}(\vx):=\hat{Q}_{\vx}\left(\hat{\tau}_{\vx}\right).
\]
Some remarks on the proposed estimator are in order.

\begin{remark}[Linear quantile regression]
	The linear quantile regression model is common in the quantile regression literature and can cover many data generating processes (see Remark 1 in \citealt{ohta2018quantile}). Importantly,  the linear quantile regression problem can be solved efficiently since the optimization problem (\ref{qr3}) can be formulated as a (parametric) linear programming problem whose solution path can be computed efficiently even for large-scale datasets \citep{koenker2005quantile}. 
	Having said that, the linear specification of the conditional quantile function is not essential and the theoretical results developed in the following Section \ref{sec: main} and Section \ref{sec: extension}  can be extended to nonlinear quantile regression models. 
\end{remark}

\begin{remark}[Comparison with other estimators]

	%The proposed estimator enjoys better computational antheoretical properties over existing alternative methods.
	%The proposed  estimator has advantages in both computational and theoretical properties compared with existing alternative methods.
	Compared with linear modal regression, our setting allows for
	%more general modeling of the conditional mode by admitting 
	nonlinear conditional mode functions even though the conditional quantile function is assumed linear in $\vx$ (see Remark 1 in \cite{ohta2018quantile}). In fact, under linear quantile assumption, $m(\vx)=\vx^T\beta(\tau_{\vx})$ and $\beta(\tau_{\vx})$ is allowed to be a (possibly nonlinear) function of $\vx$. 
	%If $\tau_{\vx}$ is independent of $\vx$,  
	In addition,  computation of linear modal regression  involves  non-convex optimization \citep{yao2014new,cheng1995mean,einbeck2006modelling}, while the proposed method only relies on linear quantile regression that can be formulated as a linear programming problem, and an one-dimensional optimization. \cite{chen2016nonparametric} show the convergence rate $O_P(h^2+n^{-1/2}h^{-(p+3)/2})$ for the KDE-based mode estimator, where $h$ is the KDE bandwidth parameter and $p$ is the number of continuous covariates. This implies slow convergence for even moderate dimensions which is the price of a more nonparametric approach. In contrast, we show that the convergence rate of our estimator is $O_P(h^2+n^{-1/2}h^{-3/2})$ for any fixed dimension $d$ and thus our estimator is free from the ``curse of dimensionality".
	%which is independent of the dimension of $\vX$.
	
\end{remark}

\section{Main  results}
\label{sec: main}
\subsection{Notation and conditions}
We use  $U(0,1)$ and $N(\mu,\Sigma)$ to denote  the uniform distribution on  $(0,1)$ and the normal distribution with mean $\mu$ and  covariance matrix $\Sigma$, respectively. We use $\| \cdot \|$, $\| \cdot \|_1$, $\| \cdot \|_{\infty}$ to denote the Euclidean, $\ell^1$, and $\ell^\infty$-norms, respectively. For a smooth function $f(x)$, we write $f^{(r)}(x)=\partial^r f(x)/\partial x^r$ for any integer $r\ge 0$ with $f^{(0)}=f$. 
For vectors $a = (a_1,\dots,a_L)^{T}, b = (b_1,\dots,b_L)^{T} \in \R^{L}$, we write $a \le b$ if $a_{\ell} \le b_{\ell}$ for all $1 \le \ell \le L$. 
%In the following, $c_{\cdot}$ and $C_{\cdot}$ are used to denote universal constants. 

Let $\cX \subset \R^{d}$ denote the support of $\vX$ and 
let $\cX_0\subset \cX$ be the set over which we make inference on the conditional mode.  In this section the dimension $d$ of $\vX$ is assumed to be fixed. Recall the baseline assumption in the last section that we are given i.i.d.\ observations $(Y_{1},\vX_1),\dots,(Y_{n},\vX_n)$ of $(Y,\vX)$ where the conditional distribution of $Y$ given $\vX$ has a unique mode and satisfies the linear quantile regression model. We make the following additional assumption.

\begin{assumption}  \label{asmp: assumption1}
%	In addition to the model assumptions above, we assume the following conditions. 
		(i) The set $\cX_0$ is compact in $\R^{d}$;
		(ii) For any $\vx\in \cX_0$, $\tau_{\vx}\in (\epsilon,1-\epsilon)$;
		(iii) The covariate vector $\vX$ has finite $q$-th moment,  $\E[\| \vX\|^{q}]<\infty$, for some $q \in [4,\infty)$, and the Gram matrix  $\E[\vX\vX^T]$ is positive definite;
		(iv) The conditional density $f(y \mid \vx)$ is three times continuously differentiable with respect to $y$ for each $\vx\in \cX$. Let $f^{(j)}(y \mid \vx)=\partial^jf(y \mid \vx)/\partial y^j$ for $j=0,1,2,3$. There exits a constant $C_1$ such that $|f^{(j)}(y \mid \vx)|\le C_1$ for all $j=0,1,2,3$ and  $(y,\vx)\in \R\times \cX$;
		(v) There exists a positive constant $c_{1}$ (that may depend on $\epsilon$) such that $f(y \mid \vx) \ge c_{1}$ for all $y \in [Q_{\vx}(\epsilon/2), Q_{\vx}(1 - \epsilon/2)]$ and $\vx \in \cX$;
		(vi) There exists a positive constant $c_{2}$ such that $-f^{(2)}(m(\vx) \mid \vx) \ge c_{2}$ for all $\vx \in \cX_{0}$;
		(vii) The kernel function $K$ is three times differentiable, symmetric, and supported in $[-1,1]$;
		(viii) The bandwidth $h=h_{n} \to 0$ satisfies that $nh^{5}/\log n \to  \infty$.
\end{assumption}

Condition (i) is innocuous (recall that $\cX_{0}$ is not the support of $\vX$). 
Condition (ii) excludes the extreme quantile case where $\tau_{\vx} \to 0$ or $1$ for some sequence of $\vx$. 
Condition (iii) is a  moment condition on the covariate vector $\vX$. 
Conditions (iv) and (v) are standard smoothness conditions on the conditional density  $f(\cdot\mid \vx)$ in the quantile regression literature \citep{koenker2005quantile}. Similar conditions  appear in \cite{chen2016nonparametric} and \cite{ohta2018quantile}. Smoothness of $f(\cdot\mid \vx)$ implies smoothness of conditional quantile function $Q_{\vx}(\tau)$. Indeed, under Conditions (iv) and (v), $Q_{\vx}(\tau)$ is four-times continuously differentiable. 
Condition (vi) ensures that the conditional mode $m(\vx)$ as a solution to the optimization problem (\ref{eq: mode}) is nondegenerate.
Condition (vi) also ensures that the map $\vx \mapsto s_{\vx}''(\tau_{\vx})$ is bounded away from zero on $\cX_{0}$, as 
\[
s_{\vx}''(\tau) = Q_{\vx}^{(3)}(\tau) = \frac{3f^{(1)}(Q_{\vx}(\tau) \mid \vx) - f (Q_{\vx}(\tau) \mid \vx) f^{(2)}(Q_{\vx}(\tau) \mid \vx)}{f(Q_{\vx}(\tau) \mid \vx)^{5}}
\]
and $f^{(1)}(Q_{\vx}(\tau_{\vx}) \mid \vx) = f^{(1)}(m(\vx) \mid \vx) = 0$. It is important to note that we only require Condition (vi) to hold for $\vx \in \mathcal{X}_0$, the set of design points we make inference on. A similar condition to Condition (vi) also appears in \cite{chen2016nonparametric}. 
Conditions (vii) and (viii) are concerned with the kernel function $K$ and the bandwidth $h_n$. 
We will use the biweight kernel $K(t) = \frac{15}{16} (1-t^{2})^2 I(|t| < 1)$ in our numerical studies. 
Condition (viii) ensures  $\hat{Q}_{\vx}^{(3)}(\tau)$ to be (uniformly) consistent; see Lemma \ref{UnifQr} in Appendix. 

\subsection{Uniform asymptotic linear representation}\label{sec:UAL}
In this section, we derive a uniform asymptotic linear representation for our estimator $\hat{m}(\vx)$, which will be a building block for the pivotal bootstrap. Define
\[
J(\tau):=\E[f\left(\vX^T\beta(\tau)\mid \vX\right)\vX\vX^T].
\]
By Assumption \ref{asmp: assumption1} (iii) and (v), the minimum eigenvalue of the matrix $J(\tau)$ is bounded away from zero for $\tau \in [\epsilon,1-\epsilon]$. 
Further, for $(u,\vx') \in (0,1) \times \R^{d}$, define
\[
\psi_{\vx}(u,\vx'):=-\frac{s_{\vx}(\tau_{\vx})}{s''_{\vx}(\tau_{\vx})\sqrt{h}}K'\left(\frac{\tau_{\vx}-u}{h}\right) \vx^TJ(\tau_{\vx})^{-1}\vx',
\]
which will serve as an influence function for our estimator $\hat{m}(\vx)$. Let $\kappa = \int t^2 K(t) dt$. 
\begin{proposition}[Uniform asymptotic linear representation]
	\label{prop: UAL}
	Under Assumption \ref{asmp: assumption1}, the following asymptotic linear representation holds uniformly in $\vx\in \cX_0$:	
	\[
	\begin{split}
	&\hat{m}(\vx)-m(\vx) + \frac{s_{\vx}(\tau_{\vx})s_{\vx}^{(3)}(\tau_{\vx})}{2s''_{\vx}(\tau_{\vx})} \kappa h^2 + o_P(h^2) \\ &\quad =\frac{1}{nh^{3/2}}\sum_{i=1}^n\psi_{\vx}(U_i,\vX_i)+O_P(n^{-1/2}h^{-1}+n^{-1}h^{-4}\log n),
	\end{split}
	\]
	where $U_{1},\dots,U_{n} \sim U(0,1)$ i.i.d.\ independent of $\vX_{1},\dots,\vX_{n}$. 
	In addition, we have 
	\begin{align*}
	&\sup_{\vx\in\cX_0}\left|\frac{1}{nh^{3/2}}\sum_{i=1}^n\psi_{\vx}(U_i,\vX_i)\right|=O_P(n^{-1/2}h^{-3/2}\sqrt{\log n}).
	\end{align*}
\end{proposition}

The influence function $\psi_{\vx}(U_i,\vX_i)$ has mean zero when $h\le \min\{\tau_{\vx} ,1-\tau_{\vx} \}$ which holds for sufficiently large $n$, since
\begin{equation}
\int_{0}^{1} K'\left (\frac{\tau_{\vx} - u}{h}  \right ) du = h \int_{(\tau_{\vx}-1)/h}^{\tau_{\vx}/h} K'(u) du = h \int_{\R} K'(u) du = 0
\label{eq: influence}
\end{equation}
and by independence between $U_i$ and $\vX_i$. 
Proposition \ref{prop: UAL} in particular implies pointwise asymptotic normality of the proposed estimator. 

\begin{corollary}[Pointwise asymptotic normality]\label{lem: pointwise}
	Suppose that Assumption \ref{asmp: assumption1} holds.  Then,  for any fixed $\vx\in \cX_0$, we have
	\[
	\sqrt{nh^3} \left [ \hat{m}(\vx)-m(\vx) + \frac{s_{\vx}(\tau_{\vx})s_{\vx}^{(3)}(\tau_{\vx})}{2s''_{\vx}(\tau_{\vx})} \kappa h^2 + o_P(h^2)\right ] \stackrel{d}{\to} N(0,V_{\vx}),
	\]
	where $V_{\vx}=s_{\vx}(\tau_{\vx})^{2}\E[(\vx^TJ(\tau_{\vx})^{-1}\vX)^2] \kappa_1/s_{\vx}''(\tau_{\vx})^{2}$ and $\kappa_1 = \int K'(t)^2 dt$.  
\end{corollary}
Proposition \ref{prop: UAL} shows that the uniform convergence rate of the proposed estimator is 
$O_P (n^{-1/2}h^{-3/2}\sqrt{\log n} + h^{2})$,
which is dimension-free (i.e., independent of $d$). If we choose $h \sim (n/\log n)^{-1/7}$, which balances between $n^{-1/2}h^{-3/2}\sqrt{\log n}$ and $h^{2}$, then the  rate reduces to $O_{P}((n/\log n)^{-2/7})$.

\subsection{Bootstrap inference}\label{sec:HI}

We consider simultaneous inference for the conditional mode at several design points $\vx_{1},\dots,\vx_{L} \in \cX_{0}$, where $L$ is allowed to depend on $n$, i.e., $L = L_{n} \to \infty$. 
Indeed, we aim at developing a general inference framework to construct confidence sets for linear combinations of the vector $(m(\vx_{\ell}))_{\ell=1}^{L}$. Specifically, we consider making inference on $D(m(\vx_{\ell}))_{\ell=1}^{L}$ where $D$ is a deterministic $M \times L$ matrix and the number of rows $M$ is also allowed to increase with $n$, i.e., $M = M_{n} \to \infty$. The following are a few examples of the matrix $D$. 
See also Examples \ref{ex: simultaneous2} and \ref{ex: testing2} ahead for more details. 

%\begin{example}[Pointwise inference] 
%\label{ex: prointwise}
%Suppose that we are interested in making inference on the conditional mode at one specific design point $\vx$; then $D=1$.
%\end{example}

\begin{example}[Simultaneous confidence intervals] 
	\label{ex: simultaneous}
	Suppose that we are interested in constructing simultaneous confidence intervals for the conditional mode at  design points $\vx_1,\dots,\vx_{L}$. Construction of such simultaneous confidence intervals requires approximating the distribution of the vector $(\hat{m}(\vx_{\ell}) - m(\vx_{\ell}))_{\ell=1}^{L}$, and thus  $D=I_L$ ($L \times L$ identity matrix). 
	
	Another application is constructing simultaneous confidence intervals for partial effects of certain covariates on the conditional mode, i.e., the change of the conditional mode due to the change of one particular covariate while the rest of the covariates are controlled. Inference on partial effects is an important topic in econometrics and social science \citep{williams2012using}. For example, suppose that we have covariate $\vX=(X_1,X_{-1})$ where $X_{-1}$ contains covariates other than $X_1$. Consider to construct simultaneous confidence intervals for  partial effects of $X_1$ at $M$ different design points $x_1^{(1)}, \dots, x_1^{(M)}$: $m(x^{(k)}_1+\delta,x_{-1})-m(x_1^{(k)},x_{-1})$ $(1\le k \le M)$ for some small user-chosen $\delta$ and fixed $x_{-1}$. To this end, we need to approximate the distribution of $(\hat{m}(x^{(k)}_1+\delta,x_{-1})-\hat{m}(x_1^{(k)},x_{-1}))_{k=1}^M$. If we take $\vx_{2k-1}=(x_1^{(k)}+\delta,x_{-1})$ and $\vx_{2k}=(x_1^{(k)},x_{-1})$ for $k=1,\dots M$, then the corresponding $D$ matrix is $D_c$ in (\ref{eq: test design}).
%	with $L=2M$. 
\end{example}

\begin{example}[Testing significance of covariates]
	\label{ex: testing}
	Suppose first that we are interested in testing whether the conditional mode is constant over designs points $\vx_{1},\dots,\vx_{L}$, i.e., $m(\vx_{1}) = \cdots = m(\vx_{L})$, 
	%which can be equivalently stated as
	which is equivalent to test $m(\vx_{\ell+1}) - m(\vx_{\ell}) = 0$ simultaneously for all $1 \le \ell \le L-1$ (this corresponds to testing lack of significance of all covariates).  Calibrating critical values for such tests  reduces to approximating the null distribution of the vector $(\hat{m}(\vx_{\ell+1}) - \hat{m}(\vx_{\ell}))_{\ell=1}^{L-1}$, and thus the matrix $D$ is $D_t$ in (\ref{eq: test design}).
	%Our framework can also cover 
	
	We can also consider testing significance of certain covariates on the conditional mode. For instance, suppose that we have three covariates (including $1$): $\vX=(1,X_1,X_2)^{T}$ with binary $X_2$ (i.e., $X_{2} \in \{ 0,1 \}$), and we are interested in testing lack of significance of the covariate $X_2$, i.e., $m(X_1,0) = m(X_1,1)$ (the constant $1$ is omitted from the expression of $m(\vX)$). This can be carried out by picking designs points $x_1^{(1)},\dots,x_1^{(M)}$ from the support of $X_1$, and testing the simultaneous hypothesis that $m(x_1^{(k)},0) = m(x_1^{(k)},1)$ (or equivalently $m(x_1^{(k)},0) - m(x_1^{(k)},1) = 0$) for all $k=1,\dots,M$. 
	Calibrating critical values for such tests requires us to approximate the distribution of $(\hat{m}(x_1^{(k)},0) - \hat{m}(x_1^{(k)},1))_{k=1}^{M}$. If we define $\vx_{2k-1} = (x_1^{(k)},0)$ and $\vx_{2k} = (x_1^{(k)},1)$ for $k=1,\dots,M$, then the corresponding $D$ matrix is the same as $D_c$ in (\ref{eq: test design}).	
\end{example}

\begin{equation}
D_c = \underbrace{
	\begin{pmatrix}
	1 & -1 & 0 & 0 & \cdots  & 0 & 0 \\
	0 & 0 & 1 & -1 &  \cdots & 0  & 0 \\
	\vdots &\vdots & \vdots &\ddots & \cdots & \vdots & \vdots \\
	0 & 0 & 0 & 0 & \cdots & 1 & -1
	\end{pmatrix}
}_{M \times 2M};
\qquad
D_t = \underbrace{
	\begin{pmatrix}
	1 & -1 & 0 & \cdots  & 0 & 0 \\
	0 & 1 & -1 & \cdots & 0  & 0 \\
	\vdots & & \ddots & \cdots & \vdots & \vdots \\
	0 & 0 & 0 & \cdots & 1 & -1
	\end{pmatrix}
}_{(L-1) \times L}
.
\label{eq: test design}
\end{equation}

To cover above applications in a unified way, we consider to approximate the distribution of $D(\hat{m}(\vx_{\ell})-m(\vx_{\ell}))_{\ell=1}^L$. We will first show that, under regularity conditions,  $\sqrt{nh^{3}}D(\hat{m}(\vx_{\ell})-m(\vx_{\ell}))_{\ell=1}^L$ can be approximated by an $L$-dimensional Gaussian vector uniformly over the hyperrectangles in $\mathbb{R}^L$, even when $L$ and $M$ are possibly much larger then $n$. 
This approximating Gaussian distribution is infeasible in practice since its covariance matrix is unknown. To deal with this difficulty, we propose to further approximate the sampling distribution by a novel pivotal bootstrap or the conventional nonparametric bootstrap.

\subsubsection{Gaussian approximation}
\label{sec: Gaussian approx}
Define $\Psi_i:=\left(\psi_{\vx_1}(U_i,\vX_i),\dots,\psi_{\vx_{L}}(U_i,\vX_i) \right)^T$ and $\Sigma := \E[\Psi_i\Psi_i^T]$.
For $k=1,\dots,M$, let $D_k^T$ denote  the $k$-th row of the  matrix $D$.
We may assume without loss of generality that each row $D_{k}$ is nonzero. 
Further, we will assume that the matrix $D$ is sparse in the sense that the number of nonzero elements of each row $D_{k}$ is of constant order, which is satisfied in all the examples discussed above. 
We are primarily interested in inference for the vector $((m(\vx_{\ell}))_{\ell=1}^{L}$, so we normalize the coordinates of the vector by their approximate standard deviations (technically the normalization does not matter for the Gaussian approximation, but we will replace the approximate standard deviations by their estimates in the bootstrap, whose effect has to be taken care of).  
Let $S_k:=\left\{\ell \in \{ 1,\dots,L \}: D_{k,\ell}\ne 0\right\}$ denote the support of $D_{k}$.
% Define $\sigma^2_{k}:= D_k^T\Sigma D_k$ for $k=1,\dots,M$, which corresponds to the variance of $D_k^T\Psi_i$, and $\Gamma:=\diag\{\sigma_{1},\dots,\sigma_{M}\}$. Set $A=(A_1,\dots,A_M)^T:=\Gamma^{-1}D$. 
Define the normalization matrix $\Gamma:=\diag\{\Gamma_{1},\dots,\Gamma_{M}\}$ and set $A=(A_1,\dots,A_M)^T:=\Gamma^{-1}D$.  In particular, if we take $\Gamma_k=\sqrt{D_k^T\Sigma D_k}$ for $k=1,\dots,M$, which corresponds to the standard deviation of $D_k^T\Psi_i$, such choice of $A$ will result in a studentized statistic, while taking $\Gamma=I_M$ gives a non-studentized statistic.

Related to the matrix $D$ and $\Gamma$, we make the following assumption.

\begin{assumption}
	\label{asmp: assumption2} %We assume the following conditions.
		(i) $\max_{1 \le k \le M}|S_k| = O(1)$ and $\max_{1\le k\le M; 1\le \ell \le L}|D_{k,\ell}| = O(1)$; 
		(ii) There exists a fixed constant $c_{3} > 0$ such that $\min_{1 \le k \le M} D_k^T\Sigma D_k \ge c_{3}$; 
		(iii) There exists a fixed constant $c_{4} > 0$ such that $c_4\le \min_{1 \le k \le M} \Gamma_k \le \max_{1 \le k \le M} \Gamma_k = O(1)$.
\end{assumption}

Condition (i) is a sparsity assumption on the matrix $D$ discussed above. The conditions 
Condition (ii) excludes the situation where $ D_k^T\Psi_i$ has vanishing variance. Condition (iii) imposes a mild condition on the normalization matrix $\Gamma$ which is automatically satisfied for both studentized and non-studentized cases under the previous two conditions.

The following theorem derives a Gaussian approximation result.

\begin{theorem}[Gaussian approximation] 
	\label{thm: HDCLT}
	Suppose that Assumptions \ref{asmp: assumption1} and \ref{asmp: assumption2} hold. In addition, assume that
	\begin{equation}
	\frac{\log^7\left(Mn\right)}{nh} \bigvee \ \frac{\log^3 (Mn)}{n^{1-2/q}h} \bigvee \frac{(\log^2 n)\log M}{nh^{5}} \to 0 \quad \text{and} \quad  (nh^{7} \vee h)\log M \to 0. 
	\label{eq: rate condition}
	\end{equation}
	Then, we have 
	\[
	\sup_{b\in \R^M}\left|\P\left(A\sqrt{nh^3}(\hat{m}(\vx_{\ell})-m(\vx_{\ell}))_{\ell=1}^L\le b\right)-\P\left(AG\le b\right)\right| \to 0,
	\]
	where $G$ is an $L$-dimensional Gaussian random vector with mean $0$ and covariance $\Sigma$.
\end{theorem} 

Condition (\ref{eq: rate condition}) allows $M$ to be much larger than $n$, i.e., $M \gg n$. 
The condition that $nh^{7} \log M \to 0$ is an ``undersmoothing" condition that ensures that the deterministic bias is negligible relative to the stochastic error. This condition can be relaxed by assuming additional smoothness conditions on the conditional density and using higher order kernels.  We do not pursue this extension for brevity. Discussion on the bandwidth selection can be found in Section \ref{sec: implementation}.  

The proof of Theorem \ref{thm: HDCLT} can be found in the Appendix. The proof builds on the uniform asymptotic linear representation developed in  Proposition \ref{prop: UAL} coupled with  the high dimensional Gaussian approximation techniques developed in  \cite{chernozhukov2014gaussian,chernozhukov2017central}. From Theorem \ref{thm: HDCLT}, we see that the distribution of $A\sqrt{nh^3}(\hat{m}(\vx_{\ell})-m(\vx_{\ell}))_{\ell=1}^L$ can be  approximated by the distribution of $AG$ uniformly over the rectangles. 
Still, the distribution of $AG$ is unknown since the covariance matrix of $G$ is unknown.
We will use a new bootstrap called the pivotal bootstrap or nonparametric bootstrap to further estimate the distribution of $AG$. 

\begin{remark}[Limit distribution of maximum deviation]
	It is of interest to find a limit distribution of the maximum deviation, $\zeta_{n} := \max_{1 \le \ell \le L} \sqrt{nh^{3}}
	|\hat{m}(\vx_{\ell}) - m(\vx_{\ell})|/\sigma_{\vx_{\ell}}$ with $\sigma_{\vx}^2 = \E[\psi_{\vx}(U,\vX)^2]$, when $L=L_n \to \infty$ after a suitable normalization. Such a limit distribution enables us to find analytical critical values for simultaneous confidence intervals.  Indeed, combining Theorem \ref{thm: HDCLT} with extreme value theory \citep[cf.][]{leadbetter1983}, we can derive a limit distribution for the maximal deviation under additional regularity conditions, cf. Proposition \ref{prop: Gumbel} in Appendix A and discussion there. 
\end{remark}

\begin{remark}[Conditioning on $\vX_i$'s]
Inspection of the proof of Theorem \ref{thm: HDCLT} shows that a version of the conclusion of Theorem \ref{thm: HDCLT} continues to hold conditionally on the covariate vectors $\vX_1,\dots,\vX_n$, with minor modifications to the regularity conditions:
\begin{equation}
\sup_{b\in \R^M}\left|\P\left(A\sqrt{nh^3}(\hat{m}(\vx_{\ell})-m(\vx_{\ell}))_{\ell=1}^L\le b \mid \vX_1,\dots,\vX_n \right)-\P\left(AG\le b\right)\right| \stackrel{P}{\to} 0.
\label{eq: fixed design}
\end{equation}
Thus, combined with the consistency of the pivotal and nonparametric bootstraps,  the size and coverage guarantees of inference methods constructed from those bootstraps continue to hold conditionally on the covariate vectors $\vX_1,\dots,\vX_n$.
The proof of the result (\ref{eq: fixed design}) is indeed similar to the validity of the pivotal bootstrap (see Theorem \ref{thm: boot} below), as the pivotal bootstrap is essentially using the randomness of $U_1,\dots,U_n$ alone. We omit the details for brevity. 
\end{remark}

\subsubsection{Pivotal bootstrap}
\label{sec: pivotal bootstrap}
The proof of Theorem \ref{thm: HDCLT} shows that the distribution of $G$ comes from approximating the distribution of the process
\begin{equation}
\vx \mapsto  \frac{1}{\sqrt{n}} \sum_{i=1}^{n} \psi_{\vx}(U_{i},\vX_{i})
\label{eq: pivotal process}
\end{equation}
at $\vx \in \{ \vx_{1},\dots,\vx_{L} \}$.
Importantly,  the process (\ref{eq: pivotal process}) is ``pivotal" in the sense that its distribution is completely known up to some estimable nuisance parameters given $\vX_1,\dots,\vX_{n}$ since $U_1,\dots,U_n$ are independent $U(0,1)$ random variables. The baseline idea of the pivotal bootstrap is to simulate the pivotal process (\ref{eq: pivotal process}) (given the data) to estimate the distribution of $G$ by generating $U(0,1)$ random variables.

To implement the pivotal bootstrap, we first have to estimate the nuisance parameters. We consider to estimate the matrix $J(\tau) = \E[f(\vX^{T}\beta(\tau) \mid \vX) \vX\vX^{T}]$ by Powell's kernel method \citep{powell1986censored}, i.e., 
$\hat{J}(\tau):=n^{-1}\sum_{i=1}^n \check{K}_{\check{h}_n}(Y_i-\vX_i^T\hat{\beta}(\tau))  \vX_i\vX_i^T$,
where $\check{K}:\R \to \R$ is a kernel function and  $\check{h}_{n}$ is a bandwidth. For simplicity of exposition, we will use $\check{K} = K$ and $\check{h}_{n} = h$. Then, we shall estimate the influence function $\psi_{\vx}$ by 
\[
\hat{\psi}_{\vx}(u,\vx') := - \frac{\hat{s}_{\vx}(\hat{\tau}_{\vx})}{\hat{s}''_{\vx}(\hat{\tau}_{\vx})\sqrt{h}}K'\left(\frac{\hat{\tau}_{\vx}-u}{h}\right)\vx^T\hat{J}(\hat{\tau}_{\vx})^{-1}\vx',
\]
where $\hat{s}_{\vx}''(\tau)$ is the second derivative of $\hat{s}_{\vx}(\tau)$ with respect to $\tau$. 

The pivotal bootstrap reads as follows. 
Generate $U_1,\dots,U_n \sim U(0,1)$ i.i.d.\ 
that are independent of the data $\cD_n:=(Y_i,\vX_i)_{i=1}^n$. 
We denote the conditional probability $\P( \cdot \mid  \cD_n)$ and conditional expectation $\E[  \cdot \mid  \cD_n ]$ by $\P_{|\cD_n}( \cdot )$ and $\E_{|\cD_n}[ \cdot ]$, respectively.
Define
\[
\hat{\Psi}_i:=\left(\hat{\psi}_{\vx_1}\left(U_i,\vX_i\right),\dots,\hat{\psi}_{\vx_L}\left(U_i,\vX_i\right)\right)^T \quad  \text{and} \quad
\hat{\Sigma}:= \frac{1}{n} \sum_{i=1}^n\E_{|\cD_n}[\hat{\Psi}_i \hat{\Psi}_i^{T}]. 
\]
Then, we shall estimate the distribution of $AG$ (or $n^{-1/2} \sum_{i=1}^{n} A\Psi_{i}$) by the conditional distribution of $n^{-1/2} \sum_{i=1}^{n} \hat{A} \hat{\Psi}_{i}$ given the data $\cD_{n}$, where $\hat{A} = \hat{\Gamma}^{-1}D$ and $\hat{\Gamma} = \diag \{ \hat{\Gamma}_1,\dots,\hat{\Gamma}_M \}$ is some estimator of $\Gamma$ (for example, equation (\ref{eq: Gamma})) that achieves sufficiently fast convergence rate (see Theorem \ref{thm: boot} for details).  
%\[
%\hat{\Gamma} = \diag \{ \hat{\sigma}_{1},\dots,\hat{\sigma}_{M} \} := \diag \left\{ \sqrt{D_{1}^{T}\hat{\Sigma}D_{1}},\dots,\sqrt{D_{M}^{T}\hat{\Sigma}D_{M}} \right\}.
%\]
The conditional distribution can be simulated with arbitrary precision. 
The following theorem establishes consistency of the pivotal bootstrap over the rectangles.

\begin{theorem}[Validity of pivotal bootstrap]
	\label{thm: boot}
	Suppose that Assumptions \ref{asmp: assumption1} and \ref{asmp: assumption2} hold  with $q > 4$ in Condition (v) in Assumption \ref{asmp: assumption1}. In addition,  assume that 
	\begin{enumerate}
		\item[(i).]  $\max_{1 \le k \le M}|\hat{\Gamma}_k-\Gamma_{k}|=O_P(n^{-1/2}h^{-5/2}\sqrt{\log n}+h)$.
		\item[(ii).] $\frac{\log^7\left(Mn\right)}{n^{1-2/q}h} \bigvee \ \frac{\log^3 (Mn)}{n^{1-4/q}h} \bigvee \frac{(\log n)\log^4 M}{nh^{5}} \to 0 \quad \text{and} \quad  h \log^2 M \to 0.$
	\end{enumerate}	
	Then, we have
	\[
	\sup_{b\in \R^M}\left|\P_{|\cD_n}\left(n^{-1/2}{\textstyle\sum}_{i=1}^n \hat{A} \hat{\Psi}_i\le b\right)-\P\left(AG\le b\right)\right| \stackrel{P}{\to} 0. 
	\]
\end{theorem}

The proof of Theorem \ref{thm: boot} can be found in Appendix.  The proof of Theorem \ref{thm: boot} is nontrivial and does not follow directly from existing results since the pivotal bootstrap differs from the nonparametric or multiplier bootstraps that have been analyzed in the literature in the high-dimensional setup. The proof consists of two steps. First, noting that $\hat{\Psi}_{1},\dots,\hat{\Psi}_{n}$ are independent with mean zero conditionally on the data $\cD_{n}$ (cf. equation (\ref{eq: influence})),  we apply the high dimensional CLT conditionally on $\cD_{n}$  to approximate the conditional distribution of  $n^{-1/2}\sum_{i=1}^{n}\hat{A}\hat{\Psi}_{i}$ by the conditional Gaussian distribution $N(0,\hat{A}\hat{\Sigma} \hat{A}^{T})$. 
Second, we compare the $N(0,\hat{A}\hat{\Sigma} \hat{A}^{T})$ distribution with $AG \sim N(0,A\Sigma A^{T})$ by a Gaussian comparison technique.

\begin{remark}[Choice of $\hat{A}$]
	For the non-studentized case, i.e., $\Gamma=I_M$, we can simply take $\hat{A}=D$.  For the studentized case, i.e., $\Gamma_k=\sqrt{D_k^T\Sigma D_k}$ ($1\le k\le M$), we can estimate $\Gamma$ by 
	\begin{equation}\label{eq: Gamma}
	\hat{\Gamma} := \diag \left\{ \sqrt{D_{1}^{T}\hat{\Sigma}D_{1}},\dots,\sqrt{D_{M}^{T}\hat{\Sigma}D_{M}} \right\},
	\end{equation}
	and compute $\hat{A}$ accordingly. We can show that this $\hat{\Gamma}$ satisfies Condition (i) of Theorem \ref{thm: boot} (cf. Lemma \ref{lem: covariance D} in Appendix). 
	In practice, $\hat{\Sigma}$ can be approximated by simulating uniform random variables and then
	$\hat{\Gamma}$ can be computed according to (\ref{eq: Gamma}). 
	%The consistency of this approach can be established under proper conditions, cf. \cite{kato2011note}.
\end{remark}

As a byproduct of the proof of Theorem \ref{thm: boot}, we can show that the conclusion of Theorem \ref{thm: HDCLT} continues to hold even if the matrix $A$ acting on $(\hat{m}(\vx_{\ell})-m(\vx_{\ell}))_{\ell=1}^{L}$ is replaced by its estimate $\hat{A}$. 

\begin{proposition}
	\label{prop: HDCLT normalized}
	Suppose that the conditions of Theorem \ref{thm: HDCLT} together with Condition (i) in the statement of Theorem \ref{thm: boot} hold. In addition, assume that 
	\[
	\frac{(\log n)\log^2 M}{nh^{5}} \to 0 \quad \text{and} \quad  h \log M \to 0.
	\]
	Then, we have 
	\[
	\sup_{b\in \R^M}\left|\P\left(\hat{A}\sqrt{nh^3}(\hat{m}(\vx_{\ell})-m(\vx_{\ell}))_{\ell=1}^L\le b\right)-\P\left(AG\le b\right)\right| \to 0.
	\]
\end{proposition}

In what follows, we discuss applications of the pivotal bootstrap to constructions of pointwise and simultaneous confidence intervals and testing using studentized statistics.

%\begin{example}[Pointwise inference] 
%\label{ex: pointwise2}
%Consider to construct a pointwise confidence interval for $m(\vx)$ for a given $\vx \in \cX_{0}$. In this case, $D=1$ and $A =1/\sigma_{\vx}$ with $\sigma_{\vx}^{2} = \E[\psi_{\vx}(U_i,\vX_i)^2]$. 
%Let $\hat{\sigma}_{\vx}^{2} = \E_{|\cD_n}[\hat{\psi}_{\vx}(U_i,\vX_{i})^2]$. Then, Proposition \ref{prop: HDCLT normalized} and Theorem \ref{thm: boot} imply that 
%\[
%\sup_{b \in \R}\left | \P \left(\sqrt{nh^{3}}
%(\hat{m}(\vx) - m(\vx))/\hat{\sigma}_{\vx} \le b \right) - \P_{U}\left(n^{-1/2} {\textstyle \sum}_{i=1}^{n} \hat{\psi}_{\vx}(U_i,\vX_i)/\hat{\sigma}_{\vx} \le b \right) \right| \stackrel{P}{\to} 0.
%\]
%under regularity conditions. Thus, denoting by 
%\[
%\begin{split}
%\hat{q}_{1-\alpha} &= \text{conditional $(1-\alpha)$-quantile of} \ \left |n^{-1/2}{\textstyle \sum}_{i=1}^{n} \hat{\psi}_{\vx}(U_i,\vX_i)/\hat{\sigma}_{\vx} \right| \\
%&=\inf \left \{ b : \P_{U}\left(\left|n^{-1/2}{\textstyle \sum}_{i=1}^{n} \hat{\psi}_{\vx}(U_i,\vX_i)/\hat{\sigma}_{\vx} \right| \le b \right) \ge 1-\alpha \right\}
%\end{split}
%\]
%we see that the data-dependent interval 
%\[
%\left [ \hat{m}(\vx) \pm \frac{\hat{\sigma}_{\vx}}{\sqrt{nh^{3}}} \hat{q}_{1-\alpha} \right ]
%\]
%contains $m(\vx)$ with probability approaching $1-\alpha$.
%\end{example}

\begin{example}[Simultaneous confidence intervals] 
	\label{ex: simultaneous2}
	Consider construction of a simultaneous confidence interval for  $m(\vx_1),\dots, m(\vx_{L})$. 
	In this case, $D = I_{L}$ ($M=L$), $A=\diag \{ 1/\sigma_{\vx_1},\dots,1/\sigma_{\vx_L} \}$, and $\hat{A} = \diag \{1/\hat{\sigma}_{\vx_1},\dots,1/\hat{\sigma}_{\vx_L} \}$, where $\sigma_{\vx}^{2} = \E[\psi_{\vx}(U,\vX)^2]$ 
	and $\hat{\sigma}_{\vx}^{2} = n^{-1} \sum_{i=1}^{n} \E_{|\cD_n}[\hat{\psi}_{\vx}(U_i,\vX_{i})^2]$. Then, Proposition \ref{prop: HDCLT normalized} and  Theorem \ref{thm: boot} imply that, for $G = (g_1,\dots,g_L)^{T} \sim N(0,\Sigma)$,
	\begin{equation}
	\begin{split}
	&\sup_{b \in \R} \left| \P \left ( \max_{1 \le \ell \le L} \left|\sqrt{nh^{3}}
	(\hat{m}(\vx_{\ell}) - m(\vx_{\ell}))/\hat{\sigma}_{\vx_{\ell}}\right| \le b \right ) - \P \left ( \max_{1 \le \ell \le L} |g_{\ell}/\sigma_{\vx_{\ell}}| \le b \right ) \right | \to 0, \ \text{and} \\
	&\sup_{b \in \R} \left| \P_{|\cD_n} \left ( \max_{1 \le \ell \le L} \left|n^{-1/2}{\textstyle \sum}_{i=1}^{n} \hat{\psi}_{\vx_{\ell}}(U_i,\vX_i)/\hat{\sigma}_{\vx_{\ell}}\right| \le b \right ) - \P \left ( \max_{1 \le \ell \le L} |g_{\ell}/\sigma_{\vx_{\ell}}| \le b \right ) \right | \stackrel{P}{\to} 0.
	\end{split}
	\label{eq: convergence}
	\end{equation}
	Denoting by 
	\[
	\hat{q}_{1-\alpha} =\text{conditional $(1-\alpha)$-quantile of} \ \max_{1 \le \ell \le L}\left|n^{-1/2}{\textstyle \sum}_{i=1}^{n} \hat{\psi}_{\vx_{\ell}}(U_i,\vX_i)/\hat{\sigma}_{\vx_{\ell}}\right|,
	\]
	we can show that the data-dependent rectangle (interval when $L=1$)
	\[
	\prod_{\ell=1}^{L} \left [ \hat{m}(\vx_{\ell}) \pm \frac{\hat{\sigma}_{\vx_{\ell}}}{\sqrt{nh^{3}}}\hat{q}_{1-\alpha}  \right ]
	\]
	contains the vector $(m(\vx_{\ell}))_{\ell=1}^{L}$ with probability approaching $1-\alpha$. 
	
	Formally, the coverage guarantee of the preceding confidence rectangle  follows from 
	\begin{equation}
	\P \left ( \max_{1 \le \ell \le L} \left|\sqrt{nh^{3}}
	(\hat{m}(\vx_{\ell}) - m(\vx_{\ell}))/\hat{\sigma}_{\vx_{\ell}}\right| \le \hat{q}_{1-\alpha} \right )  \to 1-\alpha.
	\label{eq: confidence guarantee}
	\end{equation}
	The latter (\ref{eq: confidence guarantee}) follows from the preceding convergence result (\ref{eq: convergence}) coupled with Lemma \ref{lem: quantile} in Appendix (note: since in general $\max_{1 \le \ell \le L}|g_{\ell}/\sigma_{\vx_{\ell}}|$ need not have a limit distribution, it is not immediate that the former (\ref{eq: convergence}) implies the latter (\ref{eq: confidence guarantee}); cf. Lemma 23.3 in \cite{van2000asymptotic}). A similar analysis can be done for constructing simultaneous confidence intervals for partial effects of certain covariates. 
\end{example}

\begin{example}[Testing significance of covariates]
	\label{ex: testing2}
	Consider testing  the hypothesis $H_0: m(\vx_{1}) = \cdots = m(\vx_{L})$ for some $\vx_{1},\dots,\vx_{L} \in \cX_{0}$. 
	In this case, the matrix $D$ is given by $D_t$ in (\ref{eq: test design}) with $M=L-1$, and $A(\hat{m}(\vx_{\ell}) - m(\vx_{\ell}))_{\ell=1}^{L} = \big ( (\hat{m}(\vx_{\ell+1}) - \hat{m}(\vx_{\ell}))/\sigma_{\vx_{\ell+1},\vx_{\ell}} \big )_{\ell=1}^{L-1}$ under $H_0$, where $\sigma_{\vx_{\ell+1},\vx_{\ell}}^2 = \E[(\psi_{\vx_{\ell+1}}-\psi_{\vx_{\ell}})^2 (U,\vX)]$. 
	Let $\hat{\sigma}_{\vx_{\ell+1},\vx_{\ell}}^{2}=n^{-1} \sum_{i=1}^{n}\E_{|\cD_n}[(\hat{\psi}_{\vx_{\ell+1}}-\hat{\psi}_{\vx_{\ell}})^2 (U_i,\vX_i)]$. 
	We shall consider the test of the form 
	\begin{equation}
	\max_{1 \le \ell \le L-1} \frac{\sqrt{nh^3}|\hat{m}(\vx_{\ell+1}) - \hat{m}(\vx_{\ell})|}{\hat{\sigma}_{\vx_{\ell+1},\vx_{\ell}}} > c \ \Rightarrow \ \text{reject $H_{0}$}
	\label{eq: mode test}
	\end{equation}
	for some critical value $c$. To calibrate critical values, we may use the pivotal bootstrap. For a given level $\alpha \in (0,1)$, let 
	\[
	\hat{c}_{1-\alpha} = \text{conditional $(1-\alpha)$-quantile of} \ \max_{1 \le \ell \le L-1} \left | n^{-1/2} {\textstyle \sum}_{i=1}^{n} (\hat{\psi}_{\vx_{\ell+1}}-\hat{\psi}_{\vx_{\ell}})(U_i,\vX_i)/\hat{\sigma}_{\vx_{\ell+1},\vx_{\ell}} \right |.
	\]
	Then, Proposition \ref{prop: HDCLT normalized} and Theorem \ref{thm: boot} guarantee that, under regularity conditions,  the test (\ref{eq: mode test}) with $c=\hat{c}_{1-\alpha}$ has level approaching $\alpha$ if $H_0$ is true (cf. the discussion at the end of the preceding example). 
	The case where the $D$ matrix is given by $D_c$ in (\ref{eq: test design}) is similar; we omit the details for brevity. 
\end{example}

\subsubsection{Nonparametric bootstrap}

In this section, we consider and analyze the  nonparametric (empirical) bootstrap to approximate the sampling distribution of our estimator or the approximating Gaussian distribution $AG$ that appears in Theorem \ref{thm: HDCLT}. The nonparametric bootstrap proceeds as follows. We draw $n$ i.i.d.\ bootstrap samples $(Y_1^*,\vX_1^*),\dots, (Y_n^*,\vX_{n}^*)$ from the empirical distribution of $(Y_1,\vX_1),\dots,(Y_n,\vX_n)$. For a design point $\vx \in \mathcal{X}_0$, we denote the mode estimator computed from the bootstrap samples by $\hat{m}^*(\vx)$. Then, we shall estimate the distribution of $AG$ by the conditional distribution of $\hat{A}\sqrt{nh^3}(\hat{m}^*(\vx_{\ell})-\hat{m}(\vx_{\ell}))_{\ell=1}^L$ given the data $\mathcal{D}_n$, where we define the same $\hat{A}$ as in the pivotal bootstrap. The following theorem establishes consistency of the nonparametric bootstrap over the rectangles.

\begin{theorem}[Validity of nonparametric bootstrap]\label{thm: nonparametric bootstrap}
	Suppose that Assumptions \ref{asmp: assumption1} and \ref{asmp: assumption2} hold  with $q > 4$ in Condition (v) in Assumption \ref{asmp: assumption1}. In addition,  assume that, for arbitrarily small $\gamma>0$, 
	\begin{enumerate}
		\item[(i).]  $\max_{1 \le k \le M}|\hat{\Gamma}_k-\Gamma_{k}|=O_P(n^{-1/2}h^{-5/2}\sqrt{\log n}+h)$.
		\item[(ii).] $\frac{\log^7\left(Mn\right)}{n^{1-2/q}h} \bigvee \ \frac{\log^3 (Mn)}{n^{1-4/q}h} \bigvee \frac{\log^4 M}{n^{1-\gamma}h^{5}} \to 0 \quad \text{and} \quad  (h \log M\vee nh^7)\log M \to 0.$
	\end{enumerate}	
	Then, we have 
	\[
	\sup_{b\in \R^M}\left|\P_{|\cD_n}\left(\hat{A}\sqrt{nh^3}(\hat{m}^*(\vx_{\ell})-\hat{m}(\vx_{\ell}))_{\ell=1}^L\le b\right)-\P\left(AG\le b\right)\right| \stackrel{P}{\to} 0.
	\]
\end{theorem}
The proof of Theorem \ref{thm: nonparametric bootstrap} can be found in Appendix. The proof consists of the following steps. First, we establish a uniform linear representation for $(\hat{m}^*(\vx_{\ell})-\hat{m}(\vx_{\ell}))_{\ell=1}^L$ based on a Bahadur representation for the nonparametric bootstrap quantile regression estimator (see Lemma \ref{lem: bahadur nonparametric bootstrap}); Then we follow a similar proof strategy to Theorem \ref{thm: HDCLT} while conditioning on the data $\mathcal{D}_n$ that  requires a more involved analysis than Theorem \ref{thm: HDCLT}.

\begin{remark}[Comparison with the pivotal bootstrap]
\label{rem: bootstrap comparison}
	The consistency of two bootstrap methods are established under fairly similar conditions. However, the nonparametric bootstrap can be computationally more demanding since it requires computing mode estimates on sufficiently many bootstrap samples. In contrast, the pivotal bootstrap only requires estimating nuisance parameters once and evaluating the influence functions repeatedly by generating uniform random variables, which can be easily parallelized and adapted to the distributed setting. Therefore, the pivotal bootstrap can be computationally more attractive than the nonparametric bootstrap. Our simulation results also demonstrate the computational advantage of the pivotal bootstrap over the nonparametric bootstrap, cf. Appendix \ref{sec : np bootstrap simulation}. 
\end{remark}

\section{Numerical examples} 
\label{sec: numerical}
\subsection{Simulation results}
In this section, we present the numerical performance of the pivotal bootstrap using synthetic data. Due to the space limitation, we defer the simulation results for the nonparametric bootstrap and pivotal bootstrap testing (Example \ref{ex: testing2}) to Appendices E.1 and E.3 in the supplementary material. We start with discussing implementation details, in particular the bandwidth selection. 

\subsubsection{Implementation details}
\label{sec: implementation}
In our simulation study, we use the biweight kernel, $K(t)=\frac{15}{16}(1-t^2)^2I(|t| < 1)$, and use $\epsilon = 0.1$ when computing our modal estimator.  We estimate the matrix $J(\tau)$ by  $\hat{J}(\tau)=(2n\check{h})^{-1}\sum_{i=1}^{n}I(|Y_i-\vX_i^T\hat{\beta}(\tau)|\le \check{h})\vX_i\vX_i^T$,  where $\check{h}$ is set to be the default bandwidth in \emph{quantreg} package in R (the theory does not require the kernel used to estimate $J(\tau)$ to be smooth). For the minimization of the sparsity function, we used the R function \emph{optimize()} with the computed derivative of the smoothed quantile function as the input.  We find that computing $\hat{s}''(\tau_{\vx})$ by differentiating $\hat{Q}_{\vx}(\tau)$ three times tends to be unstable in the finite sample. Instead, we use the alternative expression $s''(\tau_{\vx})=-f^{(2)}(Q_{\vx}(\tau_{\vx})\ |\ \vx)  s_{\vx}(\tau_{\vx})^4$ and estimate the derivative $f^{(2)}(\cdot \mid \vx)$ by a kernel method as in Remark 9 of \cite{ohta2018quantile} (we plug in $\hat{Q}_{\vx}(\hat{\tau}_{\vx})$ and $\hat{s}_{\vx}(\hat{\tau}_{\vx})$ for $Q_{\vx}(\tau_{\vx})$ and $s_{\vx}(\tau_{\vx})$, respectively). We defer more implementation details of nuisance parameter estimation to Appendix G.

Finally, we discuss bandwidth selection. Corollary \ref{lem: pointwise} implies that the approximate MSE of $\hat{m}(\vx)$  is 
\[
\left [\frac{s_{\vx}(\tau_{\vx})s_{\vx}^{(3)}(\tau_{\vx})}{2s_{\vx}''(\tau_{\vx})} \kappa h^2 \right ]^2 +  \frac{\kappa_1 s_{\vx}(\tau_{\vx})^2\E[(x^{T}J(\tau_{\vx})^{-1}\vX)^2]}{nh^{3}s_{\vx}''(\tau_{\vx})^2}.
\]
The optimal $h$ that minimizes the above approximate MSE is given by 
\[
h_{\mathrm{opt}}(\vx):= \left[\frac{3 \kappa_1 \vx^TJ(\tau_{\vx})^{-1}\E[\vX \vX^T] J(\tau_{\vx})^{-1}\vx }{\kappa^2 s_{\vx}^{(3)}(\tau_{\vx})^2 }\right]^{1/7} n^{-1/7}.
\]
Here we make some remarks on the optimal bandwidth. First, we note that
%$h_{\mathrm{opt}}$ is obtained by balancing the bias and the variance of the proposed estimator. Therefore,
direct use of $h_{\mathrm{opt}}$ will result in an asymptotic bias and a bias-correction will be needed. However, the asymptotic bias contains high order derivatives of the conditional quantile function that
are hard to estimate. Hence, we recommend a smaller bandwidth to be used in the finite sample implementation. In our numerical analysis, we multiply $h_{\mathrm{opt}}$ by $0.8$ to correct for too large bandwidths. We start with an initial bandwidth $h_{\mathrm{ini}}=0.8\times n^{-1/7}$ to get the initial estimate $\hat{\tau}_{\vx}^0$ for $\tau_{\vx}$ and replace $J(\tau_{\vx})$ and $\E[\vX\vX^T]$ in $h_{\mathrm{opt}}$ with $\hat{J}(\hat{\tau}^0_{\vx})$ and $n^{-1}\sum_{i=1}^n \vX_i \vX_i^T$ respectively.
Considering that estimation of the fourth derivative of the conditional quantile function is highly unstable, we adopt a ``rule of thumb"  method by using the fourth derivative of the quantile function of the standard normal distribution, i.e., plugging in $(\Phi^{-1}(\tau))^{(4)}$ for $s_{\vx}^{(3)}(\tau)$ regardless of different design points, where $\Phi$ is the distribution function of $N(0,1)$. This will lead to an estimate of $h_{\mathrm{opt}}$.
 We iterate the process one more time to construct the final computed bandwidth.
For the simultaneous inference on multiple design points, we take the bandwidth to be the median of the pointwise bandwidths at those design points. Our empirical results show that the above bandwidth selection approach works reasonably well. 

\subsubsection{Pointwise confidence intervals} \label{sec: pCI simu}
%\subsubsection{ Linear mode function model} 
We will consider two different models which correspond to linear and nonlinear mode functions respectively.  Suppose that $Y$ and $\vX=(1,X_1)$ are generated from either of the following models,
\begin{itemize}
	\item (Linear modal function) 	$Y=1+3X_1+\sigma(X_1) \xi$,
	\item (Nonlinear modal function) $Y=3U^3-3X_1U^2+3X_1 U $.
\end{itemize}
In the linear modal function case, we take $\sigma(x)=1+2x$. For the distribution of $\xi$, we consider two cases: $\xi \sim N(0,1)$ (\emph{lmNormal} model) and $\log(\xi)\sim N(1,0.64)$ (\emph{lmLognormal} model). These two cases are of interest since the conditional mode coincides with the conditional mean in the first case while they are different in the second. In particular, $m(\vX)=1+3X_1$ for the \emph{lmNormal} model and $m(\vX)=1+3X_1+(1+2X_1) e^{0.36}$ for the \emph{lmLognormal} model, both of which are linear in $\vX$. Similar models are considered in the simulation analyses of \cite{yao2014new} and \cite{ohta2018quantile}. For the nonlinear modal function case (\emph{Nonlinear} model), we take $U \sim U(0,1)$ and thus $m(\vX)=-2X_1^3/9+X_1^2$, which is nonlinear in $\vX$. We generate the covariate $X_1 \sim U(0,1)$ in both linear modal models and $X_1 \sim U(0,3)$ for the nonlinear modal model.

For each model, we construct $95\%$ and $99\%$ confidence intervals for the conditional mode. %when covariate $X_1$ takes value at $0.3$, $0.5$ and $0.7$, i.e. $m(\vx)$ 
For the \emph{lmNormal} and \emph{lmLognormal} models, we consider the following design points $\vx=(1,0.3)$, $(1,0.5)$ and $(1,0.7)$, while for the \emph{Nonlinear} model, we consider $\vx=(1,0.7)$, $(1,0.9)$ and $(1,1.1)$. We consider different sample sizes ranging from 500 to 2000 and repeat computing the confidence intervals  under different sample sizes for $500$ times. %\textcolor{blue}{In the implementation of our pivotal bootstrap, we set $0.9\le\omega\le 1$ for the \emph{lmNormal} model, $0.4\le\omega\le 0.5$ for the \emph{lmLognormal} and \emph{Nonlinear} models.} 
The resulting empirical coverage probabilities and interval length statistics are reported in Tables \ref{tab:point normal} to \ref{tab:point nonlinear}. In the simulation, we find that some of the computed confidence intervals are extremely large, especially when the sample size is comparatively small ($n=500$) due to the unstable estimation of high order derivatives of the conditional quantile function. Therefore, we report the median length of the confidence intervals to exclude the influence of those extreme results. We also present the interquartile range of the lengths of the computed confidence intervals.
\iftrue
\begin{table}[h!]
	\centering
	\caption{Simulation results for pointwise confidence intervals for \emph{lmNormal} model. For each case, the results are computed based on 500 simulated datasets with 500 bootstrap iterations.}
	\label{tab:point normal}
	{\small
		%\resizebox{\textwidth}{!}{
		\renewcommand{\arraystretch}{0.7}%
		\begin{tabularx}{\textwidth}{ Z Z Z Z Z Z Z Z}
			\toprule
			\multicolumn{1}{c}{\multirow{2}{*}{Design point}} & \multicolumn{1}{c}{\multirow{2}{*}{Sample size}} & \multicolumn{2}{c}{Coverage probability} & \multicolumn{2}{c}{Median length} & \multicolumn{2}{c}{Interquantile range} \\
			\cmidrule(lrr){3-4} \cmidrule(lrr){5-6}  \cmidrule(lrr){7-8} 
			
			\multicolumn{1}{c}{}                              & \multicolumn{1}{c}{}                             & 95\%                & 99\%               & 95\%            & 99\%             & 95\%            & 99\%         \\
			\midrule
			\multirow{3}{*}{$X_1$=0.3}                            & $n=500$                                            & 94.4\%              & \multicolumn{1}{c}{97.6\%}              & 1.17            & 1.55       & 1.23            & 1.62                  \\
			& $n=~1000$                                           & 96\%                & \multicolumn{1}{c}{98.2\%}               & 0.94            & 1.23        & 0.92            & 1.19            \\
			& $n=~2000$                                           & 96\%                & \multicolumn{1}{c}{98.2\%}             & 0.79            &1.03      & 0.71            & 0.93        \\
			\midrule
			\multirow{3}{*}{$X_1$=0.5}                            & $n=~500$                                            & 95.2\%              & \multicolumn{1}{c}{98\%}              & 1.43          & 1.87     & 1.29            & 1.64                     \\
			& $n=~1000$                                           & 95.6\%                & \multicolumn{1}{c}{98.8\%}               & 1.12           & 1.47      & 0.99            & 1.3                   \\
			& $n=~2000$                                           & 96\%                & \multicolumn{1}{c}{98.6\%}             & 0.92           & 1.18      & 0.69            & 0.92      
			\\
			\midrule
			\multirow{3}{*}{$X_1$=0.7}                            & $n=~500$                                            & 90.4\%              & \multicolumn{1}{c}{94.2\%}              & 1.42            & 1.89       & 1.67            & 2.21                  \\
			& $n=~1000$                                           & 93\%                & \multicolumn{1}{c}{96.4\%}               & 1.29            & 1.67   & 1.70            & 2.21                       \\
			& $n=~2000$                                           & 93.4\%                & \multicolumn{1}{c}{96.6\%}             & 1.02            & 1.35         & 1.08           & 1.45    
			\\
			\bottomrule   
		\end{tabularx}	
	}
\end{table}

\begin{table}[h!]
	\centering
	\caption{Simulation results for pointwise confidence intervals for \emph{lmLognormal} model. For each case, the results are computed based on 500 simulated datasets with 500 bootstrap iterations.}
	\label{tab:point lognormal}
	{\small
		%\resizebox{\textwidth}{!}{
		\renewcommand{\arraystretch}{0.7}%
		\begin{tabularx}{\textwidth}{ Z Z Z Z Z Z Z Z}
			\toprule
			\multicolumn{1}{c}{\multirow{2}{*}{Design point}} & \multicolumn{1}{c}{\multirow{2}{*}{Sample size}} & \multicolumn{2}{c}{Coverage probability} & \multicolumn{2}{c}{Median length} & \multicolumn{2}{c}{Interquartile range} \\
			\cmidrule(lrr){3-4} \cmidrule(lrr){5-6}  \cmidrule(lrr){7-8} 
			
			\multicolumn{1}{c}{}                              & \multicolumn{1}{c}{}                             & 95\%                & 99\%               & 95\%            & 99\%             & 95\%            & 99\%         \\
			\midrule
			\multirow{3}{*}{$X_1$=0.3}                            & $n=500$                                            & 92.2\%              & \multicolumn{1}{c}{96.6\%}              & 3.00            & 3.73      & 4.05           & 4.82                 \\
			& $n=~1000$                                           & 94.4\%                & \multicolumn{1}{c}{97.4\%}               & 2.12            & 2.77      & 1.70           & 2.09            \\
			& $n=~2000$                                           & 92.4\%                & \multicolumn{1}{c}{96.4\%}             & 2.03            &2.65   & 1.16           & 1.49         \\
			\midrule
			\multirow{3}{*}{$X_1$=0.5}                            & $n=500$                                            & 93.8\%              & \multicolumn{1}{c}{96.8\%}              & 3.70            & 4.65    & 4.61           & 5.69                    \\
			& $n=~1000$                                           & 90.8\%                & \multicolumn{1}{c}{96.2\%}               & 2.45            & 3.26      & 2.07           & 2.76                \\
			& $n=~2000$                                           & 96.4\%                & \multicolumn{1}{c}{98.8\%}             & 2.00            & 2.62        & 1.17           & 1.46   
			\\
			\midrule
			\multirow{3}{*}{$X_1$=0.7}                            & $n=500$                                            &90.2\%              & \multicolumn{1}{c}{94.6\%}              & 4.58            & 5.63      & 5.60           & 6.39                 \\
			& $n=~1000$                                           & 92\%                & \multicolumn{1}{c}{95.8\%}               & 2.89            & 3.68         & 2.19          & 2.82               \\
			& $n=~2000$                                           & 93.6\%                & \multicolumn{1}{c}{97.4\%}             & 2.33            & 3.03   & 1.64           & 2.01        
			\\
			\bottomrule   
		\end{tabularx}	
	}
\end{table}

\begin{table}[h!]
	\centering
	\caption{Simulation results for pointwise confidence intervals for \emph{Nonlinear} model. For each case, the results are computed based on 500 simulated datasets with 500 bootstrap iterations.}
	\label{tab:point nonlinear}
	{\small
		%\resizebox{\textwidth}{!}{
		\renewcommand{\arraystretch}{0.7}%
		\begin{tabularx}{\textwidth}{ Z Z Z Z Z Z Z Z}
			\toprule
			\multicolumn{1}{c}{\multirow{2}{*}{Design point}} & \multicolumn{1}{c}{\multirow{2}{*}{Sample size}} & \multicolumn{2}{c}{Coverage probability} & \multicolumn{2}{c}{Median length} & \multicolumn{2}{c}{Interquartile range} \\
			\cmidrule(lrr){3-4} \cmidrule(lrr){5-6} \cmidrule(lrr){7-8}
			
			\multicolumn{1}{c}{}                              & \multicolumn{1}{c}{}                             & 95\%                & 99\%               & 95\%            & 99\%               & 95\%            & 99\%       \\
			\midrule
			\multirow{3}{*}{$X_1$=0.7}                            & $n=500$                                            & 95.2\%              & \multicolumn{1}{c}{97.2\%}              & 0.80            & 1.01           & 0.97            & 1.12             \\
			& $n=~1000$                                           & 95\%                & \multicolumn{1}{c}{97.6\%}               & 0.66            & 0.84    & 0.77            & 0.93                 \\
			& $n=~2000$                                           & 94.8\%                & \multicolumn{1}{c}{97\%}             & 0.50            & 0.64      & 0.55            & 0.62         \\
			\midrule
			\multirow{3}{*}{$X_1$=0.9}                            & $n=500$                                            & 91.6\%              & \multicolumn{1}{c}{95.2\%}              & 0.75            & 0.95     & 1.03            & 1.25                    \\
			& $n=~1000$                                           & 91.4\%                & \multicolumn{1}{c}{95.8\%}               & 0.64           & 0.83      & 0.85            & 1.07                  \\
			& $n=~2000$                                           & 95.2\%                & \multicolumn{1}{c}{97.8\%}             & 0.58            & 0.75     & 0.80            & 1.04        
			\\
			\midrule
			\multirow{3}{*}{$X_1$=1.1}                            & $n=500$                                            & 92\%              & \multicolumn{1}{c}{95.6\%}              & 0.87            & 1.14    & 1.18            & 1.47                      \\
			& $n=~1000$                                           & 93\%                & \multicolumn{1}{c}{96.6\%}               & 0.69           & 0.89     & 0.92            & 1.15                   \\
			& $n=~2000$                                           & 94.4\%                & \multicolumn{1}{c}{96.8\%}             & 0.57            & 0.74           & 0.79            & 1.00  
			\\
			\bottomrule   
		\end{tabularx}	
	}
\end{table}
\fi
From Tables \ref{tab:point normal} to \ref{tab:point nonlinear}, the bootstrap confidence intervals achieve satisfying coverage probabilities in all three scenarios. We point out that, in each case, the coverage probabilities at $X_1=0.7$ are slightly lower than the other two design points under the same sample size. This is because large $X_1$ results in a large variance of $Y$ which makes the estimation more difficult. We report the mean squared error of our conditional mode estimator, $\hat{m}_{\vx}$ in  Appendix  \ref{sec: MSE} to verify this. However, the pivotal  bootstrap still achieves approximately nominal coverage probabilities in such situations when the sample size is sufficiently large. Besides, we note that the length of the confidence intervals and its variability decrease with the growing sample size for each design point across all three scenarios, which agrees with our asymptotic theories. We also report oracle pivotal bootstrap confidence intervals in Appendix \ref{sec: oracle}  where we plug in $s''_{\vx}(\tau_{\vx})$ using the underlying true density or conditional quantile function. From the results there, we can see a decrease in the length  and interquartile range (in particular, the latter) for the oracle confidence intervals comparing with the results presented above under the same setting. Therefore, we conclude that the estimation of the nuisance parameters may impact the performance of the confidence intervals significantly.  
% We also note that the resulting confidence intervals tend to be ``conservative" in some cases. This is due to those extremely large confidence intervals. 

\subsubsection{Approximate confidence band}\label{sec: sCI simu}
In this section, we investigate the finite sample performance of the pivotal bootstrap in simultaneous inference problems. In particular, we construct approximate confidence bands for the three different models considered in Section \ref{sec: pCI simu}. To build an approximate confidence band, we compute simultaneous confidence intervals over a equally spaced grid of $X_1$. Specifically, we consider a grid with 21 points over interval $[0.4,0.6]$ for the \emph{lmNormal} and \emph{lmLognormal} models and over interval $[0.6,1.2]$  for the \emph{Nonlinear} model. 
%\textcolor{blue}{In the implementation, we set $\omega= 1$ for the \emph{lmNormal} model, $\omega = 0.6$ for the \emph{lmLognormal} model and $\omega = 0.75$ for the \emph{Nonlinear} model.}
%See Figure \ref{} for a visualization of an approximate confidence band. 

We repeat the simulation 500 times for each model and calculate the empirical coverage probabilities and the median lengths defined by taking the median of the median length of the simultaneous confidence intervals in one simulation. The median is used to reduce the influence of potential extreme results in the simulations. The resulting empirical coverage probabilities and median lengths of the approximate confidence bands for each model are presented in 
Table \ref{tab:band}. 

\iftrue
\begin{table}[h!]
	\centering
	\caption{Simulation results for approximate confidence bands for \emph{lmNormal}, \emph{lmLognormal} and \emph{Nonlinear} models. For each case, the results are computed based on 500 simulated datasets with 500 bootstrap iterations.}
	{\small
		%\resizebox{\textwidth}{!}{
		\renewcommand{\arraystretch}{0.7}%
		\begin{tabularx}{\textwidth}{@{} Z Z A A A A  @{} }
			\toprule
			\multicolumn{1}{c}{\multirow{2}{*}{Models}} & \multicolumn{1}{c}{\multirow{2}{*}{Sample size}} & \multicolumn{2}{c}{Coverage probability} & \multicolumn{2}{c}{Median length}  \\
			\cmidrule(lr){3-4} \cmidrule(lrr){5-6} 
			
			\multicolumn{1}{c}{}                              & \multicolumn{1}{c}{}                             & \multicolumn{1}{c}{95\%}                & \multicolumn{1}{c}{99\%}               & 95\%            & 99\%                     \\
			\midrule
			\multirow{3}{*}{lmNormal}                            & $n=500$                                            & 93.4\%              & \multicolumn{1}{c}{97\%}              & 1.71            & 2.16                 \\
			& $n=1000$                                           & 94.6\%                & \multicolumn{1}{c}{97.8\%}               & 1.35           & 1.70               \\
			& $n=2000$                                           & 94.8\%                & \multicolumn{1}{c}{98.4\%}             & 1.07           & 1.36        \\
			\midrule
			\multirow{3}{*}{lmLognormal}                            & $n=500$                                            & 95.2\%              & \multicolumn{1}{c}{97.8\%}              & 6.30           & 7.77  \\
			& $n=1000$                                           & 94.6\%                & \multicolumn{1}{c}{98\%}               & 4.53           & 5.72  \\
			& $n=2000$                                           & 97.4\%                & \multicolumn{1}{c}{99.2\%}             & 3.48            & 4.35           \\
			\midrule
			\multirow{3}{*}{Nonlinear}                            & $n=500$                                            & 96.6\%              & \multicolumn{1}{c}{98.6\%}              & 1.69          & 2.01           \\
			& $n=1000$                                           & 95.6\%                & \multicolumn{1}{c}{98.2\%}               & 1.13          & 1.36            \\
			& $n=2000$                                           & 96.6\%                & \multicolumn{1}{c}{99.2\%}             & 0.84           & 1.02        \\
			\bottomrule   
		\end{tabularx}
		%}	
	}
	\label{tab:band}
\end{table}
\fi
From Table \ref{tab:band}, the approximate confidence bands are able to capture the modes simultaneously with probability close to the nominal probability for large sample sizes. Additionally, similar to the pointwise confidence interval, the lengths of the confidence bands decrease while the sample size grows.

\subsection{U.S. wage data}
In this section, we apply the pivotal bootstrap inference framework on a real US wage data. The data are extracted from U.S. 1980 1\% metro sample from the
Integrated Public Use Microdata Series (IPUMS) website \citep{ruggles2020ipums} and the dataset used in our analysis is provided in the supplemental material. We defer more details of the extracted dataset to Appendix H. In the following, the response $Y$ is the real log annual wage (wage),  and the regressor $\vX$ consists of the highest grade of schooling (edu), age (age) and marital status (marital\_status).
%Our real data analysis consists of two parts. In the first part, 

%We investigate the marginal effect of race on the conditional modes using pivotal bootstrap testing.
%In the second part, we evaluate the performance of bootstrap confidence intervals on bootstrap subsamples of the whole data. 
%\subsubsection{Testing  marginal effect of race}
%In this section, 
We investigate whether the most common wage given the same education and age is different in single and married people.
%Specifically, we investigate whether the most common wage is different in single and married people given the same education and work experience. This can be formulated as a testing of covariate significance problem considered in the simulation study. 
Specifically, we take the two other covariates, education and age, to be the full-sample mode of each covariate and estimate the resulting conditional mode of these two groups.
%and test the equality of the resulting conditional mode between two races. 
The estimation results are presented in Table \ref{tab:realtest}.

\begin{table}[h!]
	\centering
	\caption{Mode estimates and the mode difference confidence intervals of the two groups.}
	{\small
		\renewcommand{\arraystretch}{0.7}
		\begin{tabular}{@{} A A A A @{} }
			\toprule
			\multicolumn{2}{c}{Estimated mode of wage} & \multicolumn{2}{c}{Difference confidence intervals}  \\
			\cmidrule(lr){1-2} \cmidrule(lr){3-4} 
			
			\multicolumn{1}{c}{Single}                & \multicolumn{1}{c}{Married}               & 95\%           & 99\%                     \\
			\midrule
			\multicolumn{1}{c}{9.23}             & \multicolumn{1}{c}{9.68}              & \multicolumn{1}{c}{$(0.13, 0.79)$}           &  \multicolumn{1}{c}{$(0.07,0.84)$}              \\
			\bottomrule   
		\end{tabular}	
	}
	\label{tab:realtest}
\end{table}
To provide an intuitive evaluation of the estimation, in Figure \ref{fig:density}, we collect the people with mode values of education and age from the two groups and plot KDE-based density estimates superimposed on histograms of their log annual wage, respectively. The estimated modes (based on our estimator) and sample means are also highlighted in Figure \ref{fig:density}.
\begin{figure}[h!]
	\begin{subfigure}{0.5\textwidth}
		\centering
		\includegraphics[scale=0.28]{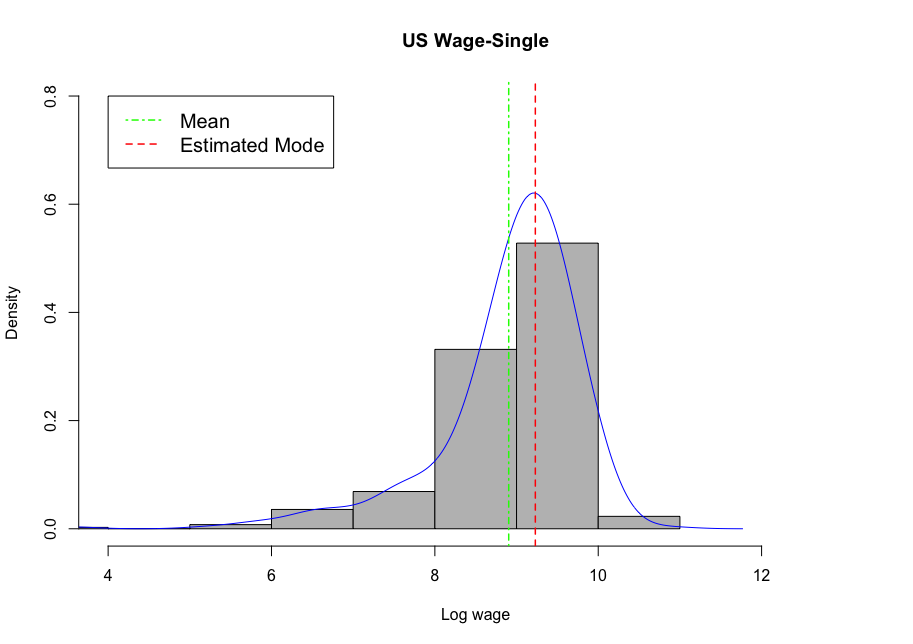}
		\caption{Single people}
		\label{rl_1}
	\end{subfigure}%
	\begin{subfigure}{0.5\textwidth}
		\centering
		\includegraphics[scale=0.28]{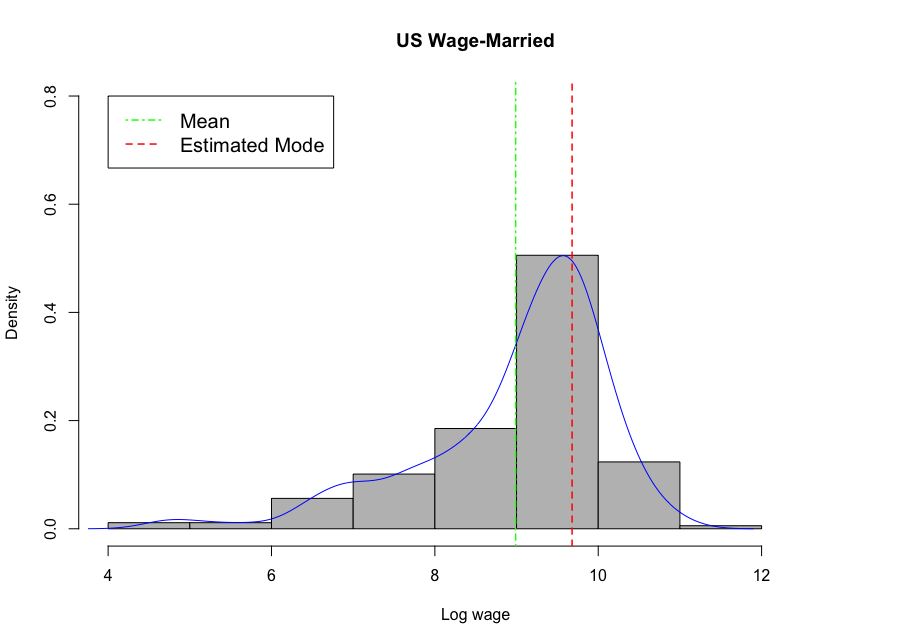}
		\caption{Married people}
		\label{rl_2}
	\end{subfigure}
	\caption{Histograms of log annual wage for single and married people with mode values of education and age based on U.S. 1980 1\% metro sample data.}\label{fig:density}
\end{figure}

From Figure \ref{fig:density}, we have several observations. First, both conditional distributions are skewed, and as argued in \cite{kemp2012regression}, the mode would be a more intuitive measure of central tendency for such skewed data. Second,  our modal estimator provides accurate estimations of conditional modes for both groups. We also present the confidence intervals for the difference of the modes of these two groups (the mode wage of single people minus the mode wage of married people) in Table \ref{tab:realtest}. Since $0$ is not contained in both 95\% and 99\% confidence intervals, we conclude that the difference of the conditional modes between the two groups is statistically significant under those two nominal levels. Therefore, the marital status can possibly be a significant factor contributing to the mode of people's wage which can be of social interest worth further research.

\section{Extension to the increasing dimension case}\label{sec: extension}

In this section, we extend the theoretical analysis to the case where the dimension $d$ of the covariate vector is allowed to increase with the sample size $n$, i.e., $d = d_n \to \infty$.
Such situations arise when we approximate conditional quantile function $Q_{\vx}(\tau)$ by a linear combination of basis functions and the approximation error is negligible (in fact, the theory of this section holds as long as the approximation error is at most of the order as the remainder term in the Bahadur representation; see Lemma \ref{prop: bahadur} in Appendix). In this case, $\vX$ is generated as basis functions of a fixed dimensional genuine covariate $\bm{Z}$, i.e., $\vX = W(\bm{Z})$, where vector $W(\bm{Z})$ includes transformations of $\bm{Z}$ that have good approximation properties such as Fourier series, splines, and wavelets;
cf. \cite{belloni2015joe,belloni2019conditional}. It is then of interest to draw simultaneous confidence intervals for the conditional mode along with values of $\bm{Z}$ which has fixed dimension although the dimension of $\vX$ increases with $n$.  

We first modify Assumption \ref{asmp: assumption1} to accommodate the case where $d=d_n \to \infty$. In what follows, constants refer to nonrandom numbers independent of $n$. 
\begin{assumption}  \label{Maup2}
		(i) There exists a constant $C_2 \ge 1$  such that $C_2^{-1} \sqrt{d} \le \| \vx \| \le C_2 \sqrt{d}$ for all $\vx \in \cX_0$; 
		(ii) There exists $\epsilon_1 \in (\epsilon,1/2)$ such that $\tau_{\vx}\in [\epsilon_1,1-\epsilon_1]$ for all $\vx \in \cX_{0}$; 
		(iii) There exists a positive constant $C_3$ such that $\P(\| \vX \| \le C_3 \sqrt{d}) = 1$.
		The Gram matrix $\E[\vX\vX^T]$ is positive definite with smallest eigenvalue $\lambda_{\min}\ge c_{\min}>0$ and largest eigenvalue $\lambda_{\max}\le c_{\max}<\infty$ for some constants $c_{\min}$ and $c_{\max}$;
		(iv) Conditions (iv)--(vii) in Assumption \ref{asmp: assumption1} hold; 	
	  (v) For any $\delta>0$, there exists a positive constant $c_4$ (that may depend on $\delta$) such that $
		\inf_{\vx\in  \cX_0}\inf_{\tau \in[\epsilon,1-\epsilon]; |\tau-\tau_{\vx}|\ge \delta} \{ s_{\vx}(\tau)-s_{\vx}(\tau_{\vx}) \} \ge c_4$;
		(vi) $d^{4} = o(n^{1-c_5})$ for some $c_5 \in (0,1)$.
\end{assumption}

Condition (i)  requires the design points of interest to be of the same order $\sqrt{d}$. We assume condition (i) to state the results in a concise way, but the $\sqrt{d}$ order can be relaxed as long as $\inf_{\vx \in \cX_0}\|\vx\|$ and $\sup_{\vx \in \cX_0} \| \vx\|$ are of the same order. 
The modified condition (ii) is assumed to avoid boundary problems of $\tau_{\vx}$ when the dimension increases.
We also assume that $\|\vX\|$ is bounded by $C_3 \sqrt{d}$ to avoid some technicalities.
In particular, under series approximation framework, this assumption is satisfied when $\vX$ is generated from basis functions such as Fourier series, B-splines and wavelet series; cf. \cite{belloni2015joe}.  The condition on the Gram matrix is satisfied under mild conditions on the distribution of the genuine covariate $\bm{Z}$ and basis functions; cf. \cite{belloni2019conditional}. 
Condition (v) is a global identification condition on $\tau_{\vx}$ that is needed to verify the uniform consistency of $\hat{\tau}_{\vx}$. If $d$ is fixed, then Condition (v) follows automatically as $\vx \mapsto \tau_{\vx}$ is continuous under Assumption \ref{asmp: assumption1} (see the proof of Lemma \ref{lem: consistency}), but if $d = d_n \to \infty$, then $s_{\vx}$ and $\tau_{\vx}$ depend on $n$, so that
we require Condition (v). Condition (vi) is used to guarantee the Bahadur representation of $\hat{\beta}(\tau)$; cf. Theorem 2 in \cite{belloni2019conditional}.

Redefine $\Psi_i$ as $\Psi_i:=(\xi_{\vx_1}(U_i,\vX_i),\dots,\xi_{\vx_{L}}\left(U_i,\vX_i\right))^T$ with 
\[
\xi_{\vx} (u,\vx') :=\frac{s_{\vx}(\tau_{\vx})}{s_{\vx}''(\tau_{\vx})\sqrt{dh}}\int \vx^TJ(t)^{-1}\vx' \{t-I\left(U\le t\right)\}K''\left(\frac{\tau_{\vx}-t}{h}\right)dt.
\]
Further, redefine the matrices $\Sigma$, $\Gamma$, and  $A$ as in Section \ref{sec: Gaussian approx} corresponding to the new definition of $\Psi_i$. For simplicity, we focus here on the studentized case where $\Gamma_k = \sqrt{D_k^T \Sigma D_k}$ for $k=1,\dots,M$. 
The reason to work with $\xi_{\vx}$ instead of $\psi_{\vx}$ is to better control the residual term in the proof of high dimensional Gaussian approximation result. Normalization by $\sqrt{d}$ ensures that the norm of $\vx/\sqrt{d}$ is bounded on $\cX_0$. 
The Gaussian approximation with $d=d_n \to \infty$ reads as follows.

\begin{theorem}[Gaussian approximation when $d = d_n \to \infty$]\label{HDCLT'}
	Suppose that Assumptions \ref{asmp: assumption2} and \ref{Maup2} hold and we also assume that	
	\begin{equation}
	\frac{d\log^7\left(Mn\right)}{nh}
	\bigvee \frac{d^4 (\log^2 n) \log^2 M}{nh^2} 
	\bigvee \frac{d^{3}(\log^2 n)\log M}{nh^{5}} \to 0 \quad \text{and} \quad  \frac{nh^{7} \log M}{d} \to 0.
	\label{eq: order}
	\end{equation}	
	Then, we have 
	\[
	\sup_{b\in \R^M}\left|\P\left(A\sqrt{nh^3d^{-1}}(\hat{m}(\vx_{\ell})-m(\vx_{\ell}))_{\ell=1}^L\le b\right)-\P\left(AG\le b\right)\right|\to 0, \quad \text{with} \ G \sim N(0,\Sigma).
	\]
\end{theorem}

Suppose that $\log M = O(\log n)$; then Condition (\ref{eq: order}) reduces to 
\[
\frac{d^4 \log^4 n}{nh^2}  \bigvee \frac{d^{3}\log^3 n}{nh^{5}} \to 0 \quad \text{and} \quad  \frac{nh^{7} \log n}{d} \to 0.
\]
If we take $h = (n/d)^{-1/7} (\log n)^{-2}$, then the condition on $d$ reduces to $d^{8} \cdot \text{polylog} (n) = o(n)$. As before, this condition can be relaxed by assuming additional smoothness conditions on the conditional density and using higher order kernels. 
Similar conditions on $d$ appear in the analysis of resampling methods for quantile regression under increasing dimensions; see, e.g., Theorem 5 in \cite{belloni2019conditional}, where $d=o(n^{1/10})$. 

We now establish the validity of the pivotal bootstrap. The theory for the nonparametric bootstrap can be shown similarly but we omit the details due to the space limit.
Redefine $\hat{\Psi}_i = (\hat{\psi}_{\vx_1}(U_i,\vX_i),\dots,\hat{\psi}_{\vx_{L}}(U_i,\vX_i))^T$ with
\[
\hat{\psi}_{\vx}(u,\vx'):=-\frac{\hat{s}_{\vx}(\hat{\tau}_{\vx})}{\hat{s}''_{\vx}(\hat{\tau}_{\vx}) \sqrt{dh}} K'\left(\frac{\hat{\tau}_{\vx}-U_i}{h}\right)\vx^T\hat{J}(\hat{\tau}_{\vx})^{-1}\vX_i
\]
Let $\hat{\Gamma}$ be as in (\ref{eq: Gamma})  corresponding to the new definition of $\hat{\psi}_{\vx}$, and let $\hat{A} = \hat{\Gamma}^{-1}D$. 

\begin{theorem}[Validity of pivotal bootstrap when $d = d_n \to \infty$]\label{boot'}
	Suppose that Assumptions \ref{asmp: assumption2}  and \ref{Maup2} hold and we also assume that 
	\[
	\frac{d\log^7\left(Mn\right)}{nh} 
	%\bigvee \ \frac{d\log^3 (Mn)}{n^{1-4/q}h} 
	\bigvee \frac{d^2(d\vee h^{-2})(\log n)\log^4 M}{nh^{3}} \to 0 \quad \text{and} \quad  h \log^2 M \to 0.
	\]
	Then, we have
	\[
	\sup_{b\in \R^M}\left|\P_{|\cD_n}\left(n^{-1/2}{\textstyle\sum}_{i=1}^n\hat{A}\hat{\Psi}_i\le b \right)-\P\left(AG\le b\right)\right|\mathop{\to}^P 0.
	\]
\end{theorem}

\begin{remark}
	The pivotal bootstrap above is the same as the one under the fixed dimension case as the extra normalization by $\sqrt{d}$ is canceled by the multiplication by $\hat{A}$ (we introduced normalization by $\sqrt{d}$ to facilitate the proof). 
\end{remark}

\section{Summary}
\label{sec: summary}
%In this paper, we propose a novel pivotal bootstrap for uniform inference on conditional modes based on a kernel-smoothed Koenker-Bassett quantile estimator.
In this paper, we study a novel pivotal bootstrap and the nonparametric bootstrap for simultaneous inference on conditional modes based on a kernel-smoothed Koenker-Bassett quantile estimator.
%new quantile based modal estimator. 
%The new estimator is defined by minimizing a kernel-smoothed Koenker-Bassett estimator. 
%We prove a uniform asymptotic linear representation for our new estimator and the pivotal leading term suggests our novel pivotal bootstrap method. Specifically, our pivotal bootstrap method resamples the uniform random variables in the leading term conditioning on the data. 
Our bootstrap inference framework allows for simultaneous inference on multiple linear functions of different conditional modes. We establish the validity of the bootstrap inference in both fixed dimension and increasing dimension settings. 
%which is general to deal with numerous practical inference problems. 
%Building on recent high dimensional probabilistic tools, we prove a high dimensional Gaussian approximation result and the validity of the pivotal and nonparametric bootstrap methods under the fixed dimension setting. We also extend the Gaussian approximation result and the validity of the pivotal bootstrap to the increasing dimension setting. 
%We consider several different inference tasks motivated from real applications. 
The numerical results provide strong support of our theoretical results.
%The numerical results not only provide strong support of our theoretical results, but also demonstrate that the new bootstrap inference framework is a flexible and powerful tool for modal regression.
Several interesting extensions remain, including  the extension to time series or longitudinal data. In such settings, we need to modify the bootstraps and develop new technical tools to deal with dependent data. These are beyond the scope of the current paper and left for future research.

%Our work has focused on the linear quantile model. It will be interesting to extend the current inference framework to more general nonlinear quantile models. Another interesting topic will be designing new modal estimators that are robust to model misspecification.  These extensions are beyond the scope of the current paper and are topics worthy of future research.

\section*{Supplemental materials}
The supplemental materials contain  the wage dataset, R codes and an appendix containing all the proofs, additional simulation results, discussion of  model misspecification and quantile crossing, and additional details of the numerical implementation.

%\section{Acknowledgments}
%The research of Tao Zhang is partially supported by NSF grant DMS-1952306. The research of Kengo Kato is partially supported by NSF grants DMS-1952306 and DMS-2014636. The authors thank the editor, one anonymous associate editor,
%and three anonymous referees for their careful review that helped improve
%upon the quality of the article. 

\clearpage
\def\theequation{A\arabic{section}.\arabic{equation}}
\renewcommand{\thesection}{A\arabic{section}}   
\renewcommand{\thetable}{A\arabic{table}}   
\renewcommand{\thefigure}{A\arabic{figure}}
\setcounter{page}{1}
\setcounter{section}{0}
\setcounter{figure}{0}
\setcounter{table}{0}

%\appendix
\begin{appendices}
	
	%\def\theequation{A\arabic{section}.\arabic{equation}}
	%\def\thesection{S\arabic{section}}
	%\renewcommand{\thesection}{A\arabic{section}}   
	%\renewcommand{\thetable}{A\arabic{table}}   
	%\renewcommand{\thefigure}{A\arabic{figure}}
	
	%	\vspace{.25cm}
	
	%	{\large Tao Zhang and Kengo Kato}
	%	\vspace{.4cm}
	
	%	Department of Statistics and Data Science, Cornell University
	%\end{center}
	
	%This document contains all the proof of the theoretical results in paper `` Pivotal Bootstrap for Quantile-based Modal Regression". 

The Appendix is organized as follows. We present the proofs of the theoretical results in the main text in Appendices A--D. We provide additional simulation results in Appendix E. We discuss the model misspecification and quantile crossing issues in Appendix F. We provide more implementation details in Appendix G. We present more details of the wage dataset in Appendix H.
\section{Auxiliary results for Section 3.3}
\begin{proposition}[Limit distribution of maximal deviation]
	\label{prop: Gumbel}
	Suppose that Assumption \ref{asmp: assumption1} and Condition (\ref{eq: rate condition}) with $M=L$ hold. Let $\zeta_{n} := \max_{1 \le \ell \le L} \sqrt{nh^{3}}
	|\hat{m}(\vx_{\ell}) - m(\vx_{\ell})|/\sigma_{\vx_{\ell}}$ with $\sigma_{\vx}^2 = \E[\psi_{\vx}(U,\vX)^2]$. Assume $L=L_n \to \infty$, and define 
	\[
	a_n = (2\log L_n)^{1/2} \quad \text{and} \quad b_n = (2\log L_n)^{1/2} - \frac{1}{2}(2\log L_n)^{-1/2}(\log \log L_n + \log \pi).
	\]
	If, in addition, $\tau_{\vx_{1}},\dots,\tau_{\vx_{L}}$ are all distinct and $\min_{k \ne \ell}|\tau_{\vx_{k}} - \tau_{\vx_{\ell}}| > 2h$ for sufficiently large $n$, then $a_n (\zeta_n - b_n)$ converges in distribution to the Gumbel distribution, i.e., 
	\[
	\lim_{n \to \infty} \P(a_n (\zeta_n - b_n) \le t) = e^{-e^{-t}}, \ t \in \R.
	\]
\end{proposition}
Proposition \ref{prop: Gumbel} suggests that we can use the Gumbel approximation  to construct simultaneous confidence intervals. The proof shows that if $\min_{k \ne \ell}|\tau_{\vx_{k}} - \tau_{\vx_{\ell}}| > 2h$, then $\Sigma$ is diagonal so that $\zeta_n$ can be approximated by the maximum in absolute value of $L$ independent $N(0,1)$ random variables, which can be further approximated (after normalization) by the Gumbel distribution by  extreme value theory. Compared with the pivotal bootstrap discussed in Section 3.3.2, the Gumbel approximation leads to analytical critical values, so from a computational perspective, using the Gumbel limit seems more attractive. However, the justification of the Gumbel approximation relies on a nontrivial spacing assumption on $\tau_{\vx_k}$'s (which the pivotal bootstrap does not). More importantly, convergence of normal suprema is known to be extremely slow \citep{hall1991}, so simultaneous confidence intervals constructed from the Gumbel approximation may not have desirable coverage accuracy. 

The following lemma is useful to establish the coverage guarantee of our confidence intervals (see Example \ref{ex: simultaneous2} for more discussion).  
\begin{lemma}
	\label{lem: quantile}
	Let $Y_n,W_n,Z_n$ be sequences of random variables defined on a probability space $(\Omega,\mathcal{A},\P)$ such that (i) $Y_n$ is measurable relative to a sub-$\sigma$-field $\mathcal{C}_{n}$ (that may depend on $n$); (ii) $\sup_{t \in \R} |\P(Y_{n} \le t) - \P(Z_n \le t)| \to 0$ and $\sup_{t \in \R}|\P(W_{n} \le t \mid \mathcal{C}_{n}) - \P(Z_n \le t)| \stackrel{P}{\to} 0$; (iii) the distribution function of $Z_n$ is continuous for each $n$ ($Z_n$ need not have a limit distribution). Let $\hat{q}_{n}(\alpha)$ denote the conditional $\alpha$-quantile of $W_n$ given $\mathcal{C}_n$. Then $\P(Y_n \le \hat{q}_{n}(\alpha)) \to \alpha$. 
\end{lemma}
The proofs of the above two auxiliary results can be found in Appendix C.3.4.

\section{Technical tools}
In this section, we collect technical tools that will be used in the subsequent proofs. For a probability measure $Q$ on a measurable space $(S,\cS)$ and a  class of measurable functions $\cF$ on $S$ such that $\cF \subset L^{2}(Q)$, let $N(\cF,\| \cdot \|_{Q,2},\delta)$ denote the $\delta$-covering number for $\cF$ with respect to the $L^{2}(Q)$-seminorm $\| \cdot \|_{Q,2}$.
The class $\cF$ is said to be pointwise measurable if there exists a countable subclass $\cG \subset \cF$ such that for every $f \in \cF$ there exists a sequence $g_{m} \in \cG$ with $g_{m} \to f$ pointwise. 
A function $F: S \to [0,\infty)$ is said to be an envelope for $\cF$ if $F(x) \geq \sup_{f \in \cF} |f(x)|$ for all $x \in S$. 
See Section 2.1 in \cite{vaart1996weak} for details.  For a vector-valued function $g$ defined over a set $T$, we define $\|g\|_T:=\sup_{x \in T}\|g(x)\|$.  The $L_{p,1}$ norm for a random
variable $X$ is defined as $\|X\|_{p,1}:=\int_{0}^{\infty}\P(|X|>t)^{1/p}dt$. 

\begin{lemma}[Local maximal inequality]
	\label{lem: maximal inequality}
	Let $X,X_{1},\dots,X_{n}$ be i.i.d.\ random variables taking values in a measurable space $(S,\cS)$, and let $\cF$ be a pointwise measurable class of (measurable) real-valued functions on $S$ with measurable envelope $F$. Suppose that $\cF$ is VC type, i.e., there exist constants $A \geq e$ and $V \geq 1$ such that
	\[
	\sup_{Q} N(\cF,\| \cdot \|_{Q,2},\epsilon \| F \|_{Q,2}) \leq (A/\epsilon)^{V}, \ 0 < \forall \epsilon \leq 1,
	\]
	where $\sup_{Q}$ is taken over all finitely discrete distributions on $S$.
	Furthermore, suppose that $0 < \E[F^{2}(X)] < \infty$, and let $\sigma^{2} > 0$ be any positive constant such that $\sup_{f \in \cF} \E[f^{2}(X)] \leq \sigma^{2} \leq \E[F^{2}(X)]$. 
	Define $B= \sqrt{\E[ \max_{1 \leq i \leq n} F^{2}(X_{i}) ]}$.  Then
	\[
	\begin{split}
	&\E \left [ \left \|  \frac{1}{\sqrt{n}} \sum_{j=1}^{n} \{ f(X_{j}) - \E[ f(X) ] \} \right \|_{\cF} \right ] \\
	&\quad \leq C \left [  \sqrt{V\sigma^{2} \log \left ( \frac{A \sqrt{\E[F^{2}(X)]}}{\sigma} \right ) } + \frac{VB}{\sqrt{n}} \log \left ( \frac{A \sqrt{\E[F^{2}(X)]}}{\sigma} \right ) \right ],
	\end{split}
	\]
	where $C > 0$ is a universal constant. 
\end{lemma}
\begin{proof}
	See Corollary 5.1 in \cite{chernozhukov2014gaussian}.
\end{proof}

The following anti-concentration inequality for Gaussian measures (called Nazarov's inequality in \cite{chernozhukov2017central}), together with the Gaussian comparison inequality,   will play  crucial roles in proving the validity of the pivotal bootstrap.

\begin{lemma}[Nazarov's inequality]
	\label{lem: Nazarov}
	Let $\bm{Y}=(Y_1,\dots,Y_d)^T$ be a centered Gaussian vector in $\R^d$ such that $\E[Y^2_j]\ge \underline{\sigma}^{2}$ for all $j=1,\dots,d$ and some constant $\underline{\sigma}>0$. Then for every $\bm{y} \in \R^d$ and $\delta >0$,
	\[
	\P(\bm{Y}\le \bm{y}+\delta )-\P(\bm{Y}\le \bm{y})\le \frac{\delta}{\underline{\sigma}} (\sqrt{2\log d} + 2). 
	\]
\end{lemma}
\begin{proof}
	See Lemma A.1 in \cite{chernozhukov2017central}; see also \cite{chernozhukov2017note}. 
\end{proof}

\begin{lemma}[Gaussian comparison]
	\label{lem: Gaussian comparison}
	Let $\bm{Y}$ and $\bm{W}$ be centered Gaussian random vectors in $\R^{d}$ with covariance matrices $\Sigma^{Y} = (\Sigma_{j,k}^{Y})_{1 \le j,k \le d}$ and $\Sigma^{W} = (\Sigma_{j,k}^{W})_{1 \le j,k \le d}$, respectively, and let $\Delta = \| \Sigma^{Y} - \Sigma^{W} \|_{\infty} := \max_{1 \le j,k \le d} |\Sigma_{j,k}^{Y} - \Sigma_{j,k}^{W}|$. 
	Suppose that $\min_{1 \le j \le d} \Sigma_{j,j}^{Y} \bigvee \min_{1 \le j \le d} \Sigma_{j,j}^{W} \ge \underline{\sigma}^{2}$ for some constant $\underline{\sigma} > 0$. 
	Then
	\[
	\sup_{b \in \R^{d}} | \P (\bm{Y} \le b) - \P (\bm{W} \le b) | \le C \Delta^{1/3} \log^{2/3} d,
	\]
	where $C$ is a constant that depends only on $\underline{\sigma}$.
\end{lemma}

\begin{proof}
	Implicit in the proof Theorem 4.1 in \cite{chernozhukov2017central}.
\end{proof}

\section{Proofs for Section 3}

\subsection{Uniform Convergence Rates}
We first establish uniform convergence rates of $\hat{Q}_{\vx}^{(r)}(\tau)$. 
The following Bahadur representation  of the linear quantile regression estimator $\hat{\beta} (\tau)$ will be used in the subsequent proofs.
\begin{lemma}[Bahadur representation of $\hat{\beta}(\tau)$] 
	\label{lem: bahadur}
	Under Assumption \ref{asmp: assumption1}, we have
	\[
	\hat{\beta}(\tau)-\beta(\tau)=J(\tau)^{-1}\left[\frac{1}{n}\sum_{i=1}^n\{\tau-I(U_i \le \tau)\}\vX_i\right]+o_{P}(n^{-3/4} \log n),
	\]
	uniformly in $\tau \in [\epsilon/2,1-\epsilon/2]$, where $U_1,\dots,U_n \sim U(0,1)$ i.i.d.\ that are independent of $\vX_1,\dots,\vX_n$.   In addition, we have
	\[
	\sup_{\tau \in [\epsilon/2,1-\epsilon/2]} \left \|\frac{1}{n}\sum_{i=1}^n\{\tau-I(U_i \le \tau)\}\vX_i\right \|=O_P(n^{-1/2}).
	\]
\end{lemma}
\begin{proof}
	See Lemma 3 in \cite{ohta2018quantile}. See also \cite{ruppert1980trimmed,gutenbrunner1992regression,he1996general}.
\end{proof}
\begin{remark}\label{rem: bahadur}
	Inspection of the proof shows that $U_i=F(Y_i\mid\vX_i)$ where $F(y\mid \vX)$ is the conditional distribution function of $Y$ given $\vX$.
\end{remark}

We first prove the following technical lemma. 
\begin{lemma}\label{UnifQr} 	
	If Assumption \ref{asmp: assumption1} holds,  then for $r=1,2,3$, we have 
	\[
	\sup_{\substack{\vx\in \cX_0 \\ \tau \in [\epsilon,1-\epsilon]}}\left| \frac{1}{nh^{r}}\sum_{i=1}^n\vx^TJ(\tau)^{-1}\vX_i \left \{ K^{(r-1)}\left(\frac{\tau-U_i}{h}\right) -  h I(r=1) \right \} \right|=O_P\left(n^{-1/2}h^{-r+1/2}\sqrt{\log n } \right).
	\]
\end{lemma}

\begin{proof}
	Since $K$ is supported in $[-1,1]$,  for sufficiently large $n$, 
	\[
	\E\left [K^{(r-1)}\left(\frac{\tau-U}{h}\right)\right] =h \int_{(1-\tau)/h}^{\tau/h} K^{(r-1)} (t) dt = h \int_{\R} K^{(r-1)} (t) dt = h I(r=1).
	\]
	Consider the function class $\cF_{h}:=\{(u,\vx') \mapsto  K^{(r-1)}((\tau-u)/h) \vx^TJ(\tau)^{-1}\vx': \vx\in \cX_0, \tau\in[\epsilon,1-\epsilon] \}$  (which depends on $n$ since $h = h_n$ does). It suffices to show that 
	\[
	\E[\|\mathbb{G}_n\|_{\cF_{h}}]=O(\sqrt{h\log n}) \quad \text{with} \quad \mathbb{G}_{n} f = n^{-1/2} \sum_{i=1}^{n} \{f(U_i,\vX_i) - \E[f(U,\vX)] \}.
	\]
	To this end, we will apply Lemma \ref{lem: maximal inequality}.
	The function class $\cF_{h}$ is a subset of the pointwise product of the following two function classes (that are independent of $n$): $\cF' = \{ (u,\vx') \mapsto \vx^TJ(\tau)^{-1}\vx': \vx\in \cX_0,\ \tau \in [\epsilon,1-\epsilon] \}$ and $\cF'' = \{ (u,\vx') \mapsto K^{(r-1)} (au +b) : a,b \in \R \}$.
	The former function class $\cF'$ has envelope $F_1(u,\vx') = C \| \vx' \|$ and the latter function class $\cF''$ has envelope $F_2 (u,\vx') = C'$ where $C,C'$ are some constants independent of $n$. The function class $\cF'$ is a subset of a vector space of  dimension $d$, so that it is a VC subgraph class with VC index at most $d+2$ (cf. Lemma 2.6.15 in \cite{vaart1996weak}). Next, since $K^{(r-1)}$ is of bounded variation (i.e., it can be written as the difference of two bounded nondecreasing functions) and the function class $\{ u \mapsto au + b : a,b \in \R \}$ is a VC subgraph class (as it is a vector space of dimension $2$), the function class $\cF''$ is VC type in view of Lemma 2.6.18 in \cite{vaart1996weak}. Conclude that, for $F(u,\vx') = C C'\| \vx' \|$, there exist positive constants $A,V$ independent of $n$ such that 
	\[
	\sup_{Q}N(\cF_h,\|\cdot \|_{Q,2},\eta \| F \|_{Q,2}) \le (A/\eta)^{V}, \ 0 < \forall \eta \le 1,
	\]
	where $\sup_{Q}$ is taken over all finitely discrete distributions on $(0,1) \times \R^{d}$.

	It is not difficult to verify that, by independence between $U$ and $\vX$, 
	\[
	\begin{split}
	\sup_{\substack{\vx \in \cX_{0}\\ \tau \in [\epsilon,1-\epsilon]}} \E[\{ K^{(r-1)}((\tau-U)/h)  \vx^{T}J(\tau)^{-1}\vX\}^{2}] \le O(1) \int_{0}^{1} K^{(r-1)}((\tau-u)/h)^2 du  = O(h). 
	\end{split}
	\]
	In addition, $\E[ \max_{1 \le i \le n}F^2 (U_i,\vX_i)] \le O(1) \E[\max_{1 \le i \le n} \| \vX_i \|^2] = O(n^{1/2})$ (as $\E[\|\vX\|^{4}] < \infty$). Conclude from Lemma \ref{lem: maximal inequality} that 
	\[
	\E[\|\mathbb{G}_n\|_{\cF_{h}}] = O( \sqrt{h\log n}+ n^{-1/4}\log n) =O(\sqrt{h\log n}).
	\]
	This completes the proof. 
\end{proof}

The following lemma derives uniform convergence rates of $\hat{Q}_{\vx}^{(r)}(\tau)$. 

\begin{lemma}[Uniform convergence rates $\hat{Q}_{\vx}^{(r)}(\tau)$]
	\label{UnifQ}
	Under Assumption \ref{asmp: assumption1}, we have
	\[
	\sup_{\substack{\vx\in \cX_0 \\ \tau \in [\epsilon,1-\epsilon]}}|\hat{Q}_{\vx}^{(r)}(\tau)-Q_{\vx}^{(r)}(\tau)|=
	\begin{cases}
	O_P\left(n^{-1/2}+h^{2}\right) & \textrm{if $r=0$}\\
	O_P\left(n^{-1/2}h^{-r+1/2}\sqrt{\log n}+h^{2}\right) & \text{if $r=1$ or $2$}\\
	O_P\left(n^{-1/2}h^{-5/2}\sqrt{\log n}+h^{}\right) & \text{if $r=3$}
	\end{cases}
	\]
\end{lemma}
\begin{proof}
	Consider first the case where $r=0$.  By definition,
	\begin{align*}
	\sup_{\substack{\vx\in \cX_0 \\ \tau \in [\epsilon,1-\epsilon]}}|\hat{Q}_{\vx}(\tau)-Q_{\vx}^{}(\tau)|&= \sup_{\substack{\vx\in \cX_0 \\ \tau \in [\epsilon,1-\epsilon]}}\left|\int\check{Q}_{\vx}(t)K_h^{}(\tau-t)dt - Q_{\vx}^{}(\tau)\right|\\
	&\le\sup_{\substack{\vx\in \cX_0 \\ \tau \in [\epsilon,1-\epsilon]}}\left|\int[\check{Q}_{\vx}(t)-Q_{\vx}(t)]K_h(\tau-t)dt \right|\\
	&\quad+\sup_{\substack{\vx\in \cX_0 \\ \tau \in [\epsilon,1-\epsilon]}}\left|\int Q_{\vx}(t)K_h^{}(\tau-t)dt - Q_{\vx}^{}(\tau) \right| \\
	&=: I + II. 
	\end{align*}
	We have $I = O_{P}(n^{-1/2})$ by Lemma \ref{lem: bahadur} and $II=O(h^2)$ by Taylor expansion. 
	
	Next, consider $1\le r\le 3$. We note that
	\begin{align*}
	\sup_{\substack{\vx\in \cX_0 \\ \tau \in [\epsilon,1-\epsilon]}}|\hat{Q}_{\vx}^{(r)}(\tau)-Q_{\vx}^{(r)}(\tau)|
	&\le \sup_{\substack{\vx\in \cX_0 \\ \tau \in [\epsilon,1-\epsilon]}}\left|\int[\check{Q}_{\vx}(t)-Q_{\vx}(t)]K_h^{(r)}(\tau-t)dt \right|\\
	&\quad+\sup_{\substack{\vx\in \cX_0 \\ \tau \in [\epsilon,1-\epsilon]}}\left|\int Q_{\vx}(t)K_h^{(r)}(\tau-t)dt - Q_{\vx}^{(r)}(\tau) \right| \\
	&=: III + IV. 
	\end{align*}
	We have $IV=O(h^2)$ for $r=1,2$ and $= O(h)$ for $r=3$ by Taylor expansion (recall that $Q_{\vx}(\tau)$ is four-times continuously differentiable). 
	Observe that, by Lemma \ref{lem: bahadur} and change of variables, 
	\[
	\begin{split}
	III
	&\le \sup_{\substack{\vx\in \cX_0 \\ \tau \in [\epsilon,1-\epsilon]}}\left| \frac{1}{nh^{r}}\sum_{i=1}^n\int \vx^TJ(\tau - th)^{-1}\vX_i\left\{\tau - th-I\left(U_i\le \tau - th\right)\right\}K^{(r)} (t) dt\right| \\
	&\quad + \underbrace{o_{P}(n^{-3/4}h^{-r} \log n)}_{=o_{P}(n^{-1/2}h^{-r+1/2}\sqrt{\log n})}.
	\end{split}
	\]
	Replacing $J(\tau - th)$ by $J(\tau)$ in the first term on the right hand side results in an error of order $O_{P}(n^{-1/2}h^{-r+1})$; this can be verified by a similar argument to the proof of the preceding lemma. Thus, it remains to bound
	\begin{align*}
	&\sup_{\substack{\vx\in \cX_0 \\ \tau \in [\epsilon,1-\epsilon]}}\left| \frac{1}{nh^{r}}\sum_{i=1}^n\int \vx^TJ(\tau)^{-1}\vX_i\left\{\tau-th-I\left(U_i\le \tau-th\right)\right\}K^{(r)}\left(t\right)dt\right| \\
	&=\sup_{\substack{\vx\in \cX_0 \\ \tau \in [\epsilon,1-\epsilon]}}\left| \frac{1}{nh^{r}}\sum_{i=1}^n\vx^TJ(\tau)^{-1}\vX_i  \left\{ K^{(r-1)}\left(\frac{\tau-U_i}{h}\right) + h\int tK^{(r)}\left(t\right)dt\right\}\right|,
	\end{align*}
	where we have used the fact that $K^{(r)}$ integrates to $0$. 
	Here, by integration by parts, 
	\[
	\int t K^{(r)}(t) dt = - \int K^{(r-1)} (t) dt = - I(r=1). 
	\]
	Thus, from Lemma \ref{UnifQr}, we have $III = O(n^{-1/2}h^{-r+1/2} \sqrt{\log n})$. This completes the proof. 
\end{proof}

\begin{remark}[Bias of $\hat{Q}_{\vx}(\tau)$ at $\tau = \tau_{\vx}$]
	\label{rem: bias}
	The bias of $\hat{Q}_{\vx}(\tau)$ can be improved to $O(h^{4})$ at $\tau = \tau_{\vx}$ by $Q_{\vx}''(\tau_{\vx}) = 0$ and symmetry of $K$. 
\end{remark}

\begin{remark}[Expansion of $\hat{Q}_{\vx}''(\tau)$]
	Inspection of the proof shows that 
	\begin{equation}
	\begin{split}
	&\hat{Q}_{\vx}''(\tau) - Q_{\vx}''(\tau) - \frac{Q_{\vx}^{(4)}(\tau)}{2} \kappa h^2 + o(h^{2}) = \frac{1}{nh^{2}}\sum_{i=1}^n K'\left(\frac{\tau-U_i}{h}\right) \vx^TJ(\tau)^{-1}\vX_i   \\
	&\quad + O_{P}(n^{-1/2}h^{-1}) + \underbrace{o_{P}(n^{-3/4}h^{-2} \log n)}_{o_{P}(n^{-1/2}h^{-1})}
	\end{split}
	\label{eq: expansion}
	\end{equation}
	uniformly in $(\tau,\vx) \in [\epsilon,1-\epsilon] \times \cX_{0}$. Recall that $\kappa = \int t^2 K(t) dt$. 
\end{remark}

\subsection{Proofs for Section \ref{sec:UAL}}
We first prove the uniform consistency of $\hat{\tau}_{\vx}$.
\begin{lemma}[Uniform consistency of $\hat{\tau}_{\vx}$]\label{lem: consistency}
	Under Assumption \ref{asmp: assumption1}, we have $\sup_{\vx\in\cX_0}|\hat{\tau}_{\vx} - \tau_{\vx}| \stackrel{P}{\to} 0$.  
\end{lemma}
\begin{proof}
	We divide the proof into two steps. 
	
	\textbf{Step 1}. We will verify that for any $\delta >0$,
	\[
	\eta_{\delta} :=
	\inf_{\vx\in  \cX_0}\inf_{\substack{\tau \in[\epsilon,1-\epsilon] \\ |\tau-\tau_{\vx}|\ge \delta}} \{ s_{\vx}(\tau)-s_{\vx}(\tau_{\vx}) \} >0.
	\]
	This follows from the following two claims: (i) 
	$s_{\vx}(\tau)-s_{\vx}(\tau_{\vx})$ is jointly continuous in $(\tau,\vx)$, (ii)
	$S_\delta:=\{ (\tau,\vx): \vx\in \cX_{0}, \tau \in [\epsilon,1-\epsilon], |\tau-\tau_{\vx}|\ge\delta \}$ is compact in $(0,1) \times \R^{d}$ and  
	the observation that $\tau_{\vx}$ is the unique minimizer of $s_{\vx}(\tau)$, i.e., $\tau_{\vx}=\argmin_{\tau \in [\epsilon,1-\epsilon]}s_{\vx}(\tau)$. Since $s_{\vx}(\tau)= \partial Q_{\vx}(\tau)/\partial \tau$ is continuous in $\tau$ for any fixed $\vx$ under Assumption \ref{asmp: assumption1} and also linear (thus convex) in $\vx$ by the linear quantile assumption, Theorem 10.7 in \cite{rockafellar1970convex} implies that $s_{\vx}(\tau)$ is jointly continuous in $(\tau,\vx$). Now, by Berge's maximum theorem (cf. Theorem 17.31 in \cite{aliprantis06}: see also their Lemma 17.6),  we see that $\tau_{\vx}$ is continuous in $\vx$. The preceding discussion also implies that $s_{\vx}(\tau)-s_{\vx}(\tau_{\vx})$ is jointly continuous in $(\tau,\vx)$. 
	Combining the continuity of $\tau_{\vx}$ and the definition of $S_{\delta}$, we can verify $S_\delta$ is closed and bounded and therefore compact.
	Thus, we have verified claims (i) and (ii) and the conclusion of this step follows. 
	
	\textbf{Step 2}. We will prove the uniform consistency of $\hat{\tau}_{\vx}$. 
	Consider the event $\mathcal{A}_{\delta} :=\{\sup_{\vx \in \cX_0} \{ s_{\vx}(\hat{\tau}_{\vx})-s_{\vx}(\tau_{\vx}) \}\ge \eta_{\delta}\}$. Observe that
	\[
	\begin{split}
	\sup_{\vx \in \cX_0}\{ &s_{\vx}(\hat{\tau}_{\vx})-s_{\vx}(\tau_{\vx}) \} \\
	&\quad \le \sup_{\vx \in \cX_0}\{ s_{\vx}(\hat{\tau}_{\vx})-\hat{s}_{\vx}(\hat{\tau}_{\vx})\}+\sup_{\vx \in \cX_0}\{ \hat{s}_{\vx}(\hat{\tau}_{\vx})-\hat{s}_{\vx}(\tau_{\vx})\}+\sup_{\vx \in \cX_0}\{\hat{s}_{\vx}(\tau_{\vx})-s_{\vx}(\tau_{\vx})\}.
	\end{split}
	\]
	The first and third terms on the right hand side are $o_{P}(1)$ by Lemma \ref{UnifQ}, while the second term is nonpositive by the definition of $\hat{\tau}_{\vx}$. This implies that $\P(\mathcal{A}_{\delta}) \le \P(\eta_{\delta} \le o_{P}(1)) = o(1)$. The uniform consistency of $\hat{\tau}_{\vx}$ follows from the fact that the event $\{ \sup_{\vx \in \cX_{0}} |\hat{\tau}_{\vx} - \tau_{\vx}| \ge \delta \}$ is included in $\mathcal{A}_{\delta}$.  
\end{proof}

The uniform consistency guarantees that the first order condition for $\hat{\tau}_{\vx}$ holds for all $\vx\in \cX_{0}$ with probability approaching one, i.e., 
\begin{equation}
\P \left ( \hat{s}_{\vx}'(\hat{\tau}_{\vx}) = 0, \forall \vx \in \cX_{0} \right ) \to 1.
\label{eq: FOC}
\end{equation}
Recall that $\hat{s}_{\vx}'(\tau) = \hat{Q}_{\vx}''(\tau)$. 
Now, we derive an asymptotic linear representation for $\hat{\tau}_{\vx}$. 
\begin{lemma}[Asymptotic linear representation of $\hat{\tau}_{\vx}$]
	\label{lem: UALtau}
	Under Assumption \ref{asmp: assumption1}, the following expansion holds uniformly in $\vx\in \cX_0$:
	\[
	\begin{split}
	&\hat{\tau}_{\vx} -\tau_{\vx} + \frac{s_{\vx}^{(3)}(\tau_{\vx})}{2s''_{\vx}(\tau_{\vx})
	} \kappa h^2 + o_P(h^2) \\
	&\quad =-\frac{1}{nh^{2}s''_{\vx}(\tau_{\vx})}\sum_{i=1}^n K'\left(\frac{\tau_{\vx}-U_i}{h}\right) \vx^TJ(\tau_{\vx})^{-1}\vX_{i}+ O_P(n^{-1/2}h^{-1} + n^{-1} h^{-4} \log n).
	\end{split}
	\]
	In addition, the first term on the right hand side is $O_{P}(n^{-1/2}h^{3/2}\sqrt{\log n})$ uniformly in $\vx \in \cX_{0}$. 
\end{lemma}
\begin{proof}
	From the first order condition (\ref{eq: FOC}) coupled with the Taylor expansion, we have 
	\[
	0=\hat{Q}_{\vx}''(\hat{\tau}_{\vx})=\hat{Q}_{\vx}''(\tau_{\vx})+\hat{Q}_{\vx}^{(3)}(\check{\tau}_{\vx})(\hat{\tau}_{\vx}-\tau_{\vx}).
	\]
	where $\check{\tau}_{\vx}$ lies between $\hat{\tau}_{\vx}$ and $\tau_{\vx}$. 
	This yields that
	\[
	\hat{\tau}_{\vx}-\tau_{\vx}=-\hat{Q}_{\vx}^{(3)}(\check{\tau}_{\vx})^{-1}\cdot \hat{Q}_{\vx}''(\tau_{\vx}).
	\]
	The rest of the proof is divided into two steps. 
	
	\textbf{Step 1}. We will show that $\sup_{\vx \in \cX_{0}} |\hat{\tau}_{\vx} - \tau_{\vx}| = O_{P}(n^{-1/2}h^{-3/2}\sqrt{\log n} + h^{2})$. 
	Observe that $\hat{Q}_{\vx}^{(3)}(\check{\tau}_{\vx}) = Q^{(3)}_{\vx}(\check{\tau}_{\vx}) + O_{P}(n^{-1/2}h^{-5/2}\sqrt{\log n} + h) = Q^{(3)}_{\vx}(\check{\tau}_{\vx}) + o_{P}(1)$ uniformly in $\vx \in \cX_{0}$ by Lemma \ref{UnifQr},  $Q^{(3)}_{\vx}(\check{\tau}_{\vx}) = Q_{\vx}^{(3)}(\tau_{\vx}) + o_{P}(1)$ uniformly in $\vx \in \cX_0$ by the uniform consistency of $\hat{\tau}_{\vx}$, and the map $\vx \mapsto Q_{\vx}^{(3)}(\tau_{\vx})$ is bounded away from zero on $\cX_{0}$. 
	Thus, we have 
	\[
	\sup_{\vx \in \cX_{0}} |\hat{\tau}_{\vx} - \tau_{\vx}| = O_{P} \left ( \sup_{\vx \in \cX_{0}} | \hat{Q}_{\vx}''(\tau_{\vx}) |  \right ).
	\]
	However, since $Q_{\vx}''(\tau_{\vx}) = 0$, the right hand side on the above equation is $O_{P}(n^{-1/2}h^{-3/2}\sqrt{\log n} + h^{2})$ by Lemma \ref{UnifQr}. 
	
	\textbf{Step 2}. We wish to derive the conclusion of the lemma. From the preceding discussion, we see that $\hat{Q}_{\vx}^{(3)} (\check{\tau}_{\vx}) = Q_{\vx}^{(3)}(\tau_{\vx}) + O_{P}(n^{-1/2}h^{-5/2}\sqrt{\log n}+h)$ uniformly in $\vx\in\cX_{0}$, so that 
	\[
	\hat{\tau}_{\vx} - \tau_{\vx} = - Q_{\vx}^{(3)}(\tau_{\vx})^{-1} \hat{Q}_{\vx}''(\tau_{\vx}) + O_{P}(n^{-1}h^{-4}\log n + n^{-1/2}h^{-1/2}\sqrt{\log n} + h^{3})
	\]
	uniformly in $x \in \cX_{0}$. 
	The conclusion of the lemma follows from combining the expansion (\ref{eq: expansion}). 
\end{proof}

We are now in position to prove Proposition \ref{prop: UAL}. 

\begin{proof}[Proof of Proposition \ref{prop: UAL}]
	We note that, uniformly in $\vx \in \cX_{0}$, 
	\[
	\begin{split}
	&\hat{m}(\vx)- m(\vx) \\
	&\quad =\{ \hat{Q}_{\vx}(\hat{\tau}_{\vx})-Q_{\vx}(\hat{\tau}_{\vx}) \} +\{ Q_{\vx}(\hat{\tau}_{\vx})-Q_{\vx}(\tau_{\vx}) \} \\
	&\quad =Q_{\vx}(\hat{\tau}_{\vx})-Q_{\vx}(\tau_{\vx})+O_P(n^{-1/2}) + o_{P}(h^{2}) \quad \text{(by Lemma \ref{UnifQr} and  Remark \ref{rem: bias})} \\
	&\quad =Q_{\vx}'(\tau_{\vx})(\hat{\tau}_{\vx}-\tau_{\vx})+O_P\left(\sup_{\vx' \in \cX_0} |\hat{\tau}_{\vx'}-\tau_{\vx'}|^3\right)+O_P(n^{-1/2}) + o_{P}(h^{2}) \quad \text{(by $Q_{\vx}''(\tau_{\vx}) = 0$)} \\
	&\quad =\underbrace{\frac{1}{nh^{3/2}}\sum_{i=1}^n\psi_{\vx}(U_i,\vX_i)}_{=O_{P}(n^{-1/2}h^{-3/2}\sqrt{\log n})} - \frac{s_{\vx}(\tau_{\vx})s_{\vx}^{(3)}(\tau_{\vx})}{2s''(\tau_{\vx})} \kappa h^2 \\
	&\qquad +  O_P(n^{-1/2}h^{-1} + n^{-1} h^{-4} \log n) + o_{P}(h^{2}). \quad \text{(by Lemma \ref{lem: UALtau})}
	\end{split}
	\]
	This completes the proof.
\end{proof}

\begin{proof}[Proof of Corollary \ref{lem: pointwise}]
	Proposition \ref{prop: UAL} implies that,  for any fixed $\vx\in \cX_0$,
	\[
	\hat{m}(\vx)-m(\vx)=\frac{1}{nh^{3/2}}\sum_{i=1}^n\psi_{\vx}(U_i,\vX_i)+ o_P(n^{-1/2}h^{-3/2}). 
	\]
	Thus, it suffices to show that $n^{-1/2} \sum_{i=1}^{n} \psi_{\vx}(U_i,\vX_{i}) \stackrel{d}{\to} N(0,V_{\vx})$. 
	Recall that $\psi_{\vx}(U_i,\vX_{i})$ has mean zero. The above result follows from verifying the Lyapunov condition, together with the fact that $\E[\psi_{\vx}(U,\vX)^{2}] = V_{\vx}$. We omit the details for brevity. 
\end{proof}

\subsection{Proofs for Section \ref{sec:HI}}

\subsubsection{Proof of Theorem \ref{thm: HDCLT}}
%\begin{proof}[Proof of Theorem \ref{thm: HDCLT}]
	We divide the proof into two steps. 
	
	\textbf{Step 1}. 
	We will show that 
	\begin{equation}\label{term1}
	\sup_{b\in \R^M}\left|\P\left(n^{-1/2}{\textstyle \sum}_{i=1}^nA\Psi_i\le b\right)-\P\left(AG\le b\right)\right|\to 0.
	\end{equation}
	To this end, we verify Conditions (M.1), (M.2), and (E.2) in Proposition 2.1 of \cite{chernozhukov2017central}.  
	
	\textbf{Condition (M.1)}: For $k=1,\dots,M$, by definition and Assumption \ref{asmp: assumption2}, 
	\[
	\E[ (A_{k}^T\Psi_i)^2] = A_{k}^{T} \E[\Psi_{i}\Psi_{i}^{T}]A_{k} = D_{k}^{T}\Sigma D_{k}/\Gamma_{k}^2,
	\]
	which is bounded away from zero uniformly over $1 \le k \le M$. 
	
	\textbf{Condition (M.2)}: 
	For $k=1,\dots,M$, 
	\[
	\begin{split}
	&\E\left[\left|A^T_{k} \Psi_i \right|^3\right] \le \max_{\ell \in S_k}|A_{k,\ell}|^3 \E\left[ \left \| (\psi_{\vx_{\ell}}(U,\vX))_{\ell \in S_{k}} \right \|_{1}^{3} \right ] \\
	&\quad \le \max_{\ell \in S_{k}}|A_{k,\ell}|^3 \cdot  |S_{k}|^{2} \sum_{\ell \in S_{k}}\E[|\psi_{\vx_{\ell}}(U,\vX)|^{3}
	]
	\le  \max_{\ell \in S_{k}}|A_{k,\ell}|^3 \cdot  |S_{k}|^{3} \max_{1 \le \ell \le L} \E[|\psi_{\vx_{\ell}}(U,\vX)|^{3}
	]
	\end{split}
	\]
	Under our assumption, $\max_{1 \le k \le M;\ell \in S_k}|A_{k,\ell}| = O(1)$ and $\max_{1 \le k \le M} |S_{k}| = O(1)$. In addition, 
	\[
	\begin{split}
	&\max_{1 \le \ell \le L} \E[|\psi_{\vx_{\ell}}(U,\vX)|^{3}] \\
	&\quad \le O(h^{-3/2}) \E[\| \vX \|^{3}] \max_{1 \le \ell \le L} \int_{0}^{1}\left | K'\left (\frac{\tau_{\vx_{\ell}}-u}{h} \right) \right|^{3} du = O(h^{-1/2}). 
	\end{split}
	\]
	Likewise, we have $\max_{1 \le k \le M} \E[|A^T_{k} \Psi_i|^4]= O(h^{-1})$.

	\textbf{Condition (E.2)}:  Similarly to the previous case (but bounding $h^{-1/2} K'((\tau_{\vx_{\ell}}-U)/h)$ by $h^{-1/2} \| K' \|_{\infty}$), we can show that
	\[
	\E\left[\max_{1\le k \le M}\left|A^T_{k}\Psi_i\right|^{q}\right] \le O(h^{-q/2})  \E[\| \vX \|^{q}] =  O(h^{-q/2}). 
	\]
	
	Thus, we can apply Proposition 2.1 in  \cite{chernozhukov2017central}, and the conclusion of this step follows as soon as
	\[
	\frac{\log^7\left(Mn\right)}{nh} \bigvee \frac{\log^3\left(Mn\right)}{n^{1-2/q}h} \to 0,
	\]
	but this is satisfied under our assumption. 
	
	\textbf{Step 2}. Define $\delta_{n} = h^{1/2} + n^{-1/2} h^{-5/2} \log n +  n^{1/2}h^{7/2}$ and $R_n = (\hat{m}(\vx_{\ell}) - m(\vx_{\ell}))_{\ell=1}^{L} - (nh^{3/2})^{-1} \sum_{i=1}^{n} \Psi_{i}$. 
	By Proposition \ref{prop: UAL}, we know that $\sqrt{nh^{3}}\| R_n \|_{\infty} = O_{P}(\delta_{n})$, so that 
	\[
	\begin{split}
	\sqrt{nh^{3}} \| A R_{n} \|_{\infty} &\le \max_{1 \le k \le M} \sum_{\ell \in S_{k}} |A_{k,\ell}| | \sqrt{nh^{3}} R_{n,\ell} | \le \max_{1 \le k \le M; 1 \le \ell \le L} |A_{k,\ell}| \max_{1 \le \ell \le L}\sum_{\ell \in S_{k}} |\sqrt{nh^{3}} R_{n,\ell}| \\
	&\le \max_{1 \le k \le M; 1 \le \ell \le L} |A_{k,\ell}| \max_{1 \le k \le M} |S_{k}| \| \sqrt{nh^{3}} R_{n}\|_{\infty} = O_{P}(\delta_{n}). 
	\end{split}
	\]
	Thus, for any $B_n \to \infty$, we have $\P (\sqrt{nh^{3}} \| A R_{n} \|_{\infty} > B_{n}\delta_{n}) = o(1)$. 
	Now, for any $b \in \R^{M}$,
	\[
	\begin{split}
	&\P\left(A\sqrt{nh^3}(\hat{m}(\vx_{\ell})-m(\vx_{\ell}))_{\ell=1}^L\le b\right) \\
	&\le \P\left(n^{-1/2}{\textstyle \sum}_{i=1}^nA\Psi_i\le b + B_{n}\delta_{n} \right) + o(1) \\ 
	&\le \P( AG \le b + B_{n}\delta_{n}) + o(1) \quad \text{(by Step 1)} \\
	&\le \P( AG \le b ) + O(B_{n}\delta_{n}\sqrt{\log M}) + o(1), \quad \text{(by  Nazarov's inequality (Lemma \ref{lem: Nazarov}))}
	\end{split}
	\]
	where the $o$ and $O$ terms are independent of $b$.  Likewise, we have 
	\[
	\P\left(A\sqrt{nh^3}(\hat{m}(\vx_{\ell})-m(\vx_{\ell}))_{\ell=1}^L\le b\right) \ge \P(AG \le b) - O(B_{n}\delta_{n}\sqrt{\log M}) - o(1).
	\]
	Since $B_n \delta_n \sqrt{\log M} \to 0$ for sufficiently slow $B_n \to \infty$ under our assumption, we obtain the conclusion of the theorem.
%\end{proof}
\qed

\subsubsection{Proofs of Theorem \ref{thm: boot} and Proposition \ref{prop: HDCLT normalized}}

We start with proving some technical lemmas. 
We use $\|\cdot\|_{\text{op}}$ to denote the operator norm of a matrix.
\begin{lemma}\label{lem: Jacobian} 
	Under Assumption \ref{asmp: assumption1}, we have
	\[
	\sup_{\vx\in\cX_0} \| \hat{J}(\hat{\tau}_{\vx})-J(\tau_{\vx}) \|_{\op}=O_P(n^{-1/2}h^{-3/2}\sqrt{\log n}+h^2).
	\]
\end{lemma}
\begin{proof}
	It suffices to show that $\sup_{\vx \in \cX_{0}} |\hat{J}_{j,k}(\hat{\tau}_{\vx})-J_{j,k}(\tau_{\vx}) |=O_P(n^{-1/2}h^{-3/2}\sqrt{\log n}+h^2)$ for any $1 \le j,k \le d$ (as the dimension $d$ is fixed). Observe that 
	\begin{align*}
	&\sup_{\vx \in \cX_{0}} |\hat{J}_{j,k}(\hat{\tau}_{\vx})-J_{j,k}(\tau_{\vx})| \\
	&\le  \sup_{\vx \in \cX_{0}} \left |  \frac{1}{n}\sum_{i=1}^{n}K_h(Y_{i}-\vX_{i}^T\hat{\beta}(\hat{\tau}_{\vx}))X_{ij}X_{ik}- \E\left[K_h(Y-\vX^T\beta)X_{j}X_{k} \right]\big|_{\beta=\hat{\beta}(\hat{\tau}_{\vx})} \right |  \\
	&\quad + \sup_{\vx \in \cX_{0}} \left | \E\left[K_h(Y-\vX^T\beta \mid \vX_{i})X_{j}X_{k} \right]\big|_{\beta=\hat{\beta}(\hat{\tau}_{\vx})}-\E\left[f(\vX^T\beta \mid  \vX)X_{j}X_{k} \right]\big|_{\beta=\hat{\beta}(\hat{\tau}_{\vx})} \right |  \\
	&\quad +\sup_{\vx \in \cX_{0}} 
	\left | \E\left[f(\vX^T\beta \mid  \vX)X_{j}X_{k}\right]\big|_{\beta=\hat{\beta}(\hat{\tau}_{\vx})}-\E\left[f(\vX^T\beta \mid  \vX )X_{j}
	X_{k} \right]\big|_{\beta=\beta(\tau_{\vx})}\right |. 
	\end{align*}
	It is routine to show that the first and second terms on the right hand side are $O_{P}(n^{-1/2}h^{-1})$ and $O(h^{2})$, respectively; cf. the proof of Lemma \ref{UnifQ}. By Taylor expansion, the last term can be bounded by $O_{P}(\| \hat{\beta} (\hat{\tau}_{\vx}) - \beta (\tau_{\vx}) \|_{\cX_{0}})$. Observe that 
	\[
	\begin{split}
	&\| \hat{\beta} (\hat{\tau}_{\vx}) - \beta (\tau_{\vx}) \|_{\cX_{0}} \le \| \hat{\beta} - \beta \|_{[\epsilon,1-\epsilon]} + \| \beta (\hat{\tau}_{\vx}) - \beta (\tau_{\vx}) \|_{\cX_{0}} \\
	&\quad \le O_{P}(n^{-1/2} + \| \hat{\tau}_{\vx} - \tau_{\vx} \|_{\cX_{0}}) =O_{P}(n^{-1/2}h^{-3/2}\sqrt{\log n} + h^{2}). 
	\end{split}
	\]
	This completes the proof. 
\end{proof}

\begin{lemma}
	\label{lem: covariance}
	Under Assumption \ref{asmp: assumption1}, we have 
	\[
	\| \hat{\Sigma} - \Sigma \|_{\infty} = O_{P}(n^{-1/2}h^{-5/2}\sqrt{\log n} + h).
	\]
\end{lemma}
\begin{proof}
	For simplicity of notation, let $\hat{J}_{\vx_k} = \hat{J}(\hat{\tau}_{\vx_k})$ and $J_{\vx_k}=J(\tau_{\vx_k})$.  The difference $\hat{\Sigma}_{\ell,k} - \Sigma_{\ell,k}$ can be decomposed as 
	\begin{align*}
	&\left[\frac{\hat{s}_{\vx_k}(\hat{\tau}_{\vx_k})\hat{s}_{\vx_{\ell}}(\hat{\tau}_{\vx_{\ell}})}{\hat{s}''_{\vx_k}(\hat{\tau}_{\vx_k})\hat{s}''_{\vx_{\ell}}(\hat{\tau}_{\vx_{\ell}})}\right]\E_{|\cD_n}\left[ \frac{1}{h} K'\left(\frac{\hat{\tau}_{\vx_k}-U}{h}\right)K'\left(\frac{\hat{\tau}_{\vx_{\ell}}-U}{h}\right)\right] \vx_k^T\hat{J}_{\vx_k}^{-1} \left [ \frac{1}{n}\sum_{i=1}^n\vX_i \vX_{i}^{T} \right ] \hat{J}_{\vx_{\ell}}^{-1}\vx_{\ell} \\
	&-\left [ \frac{s_{\vx_k}(\tau_{\vx_k})s_{\vx_{\ell}}(\tau_{\vx_{\ell}})}{s''_{\vx_k}(\tau_{\vx_k})s''_{\vx_{\ell}}(\tau_{\vx_{\ell}})} \right ] \E\left[\frac{1}{h}K'\left(\frac{\tau_{\vx_k}-U}{h}\right)K'\left(\frac{\tau_{\vx_{\ell}}-U}{h}\right)\right] \vx_k^TJ_{\vx_k}^{-1}\E[\vX \vX^{T}] J_{\vx_{\ell}}^{-1}\vx_{\ell}.
	\end{align*}
	Observe that 
	\[
	\begin{split}
	&\max_{1 \le k,\ell \le L} \left | \frac{\hat{s}_{\vx_k}(\hat{\tau}_{\vx_k})\hat{s}_{\vx_{\ell}}(\hat{\tau}_{\vx_{\ell}})}{\hat{s}''_{\vx_k}(\hat{\tau}_{\vx_k})\hat{s}''_{\vx_{\ell}}(\hat{\tau}_{\vx_{\ell}})}-\frac{s_{\vx_k}(\tau_{\vx_k})s_{\vx_{\ell}}(\tau_{\vx_{\ell}})}{s''_{\vx_k}(\tau_{\vx_k})s''_{\vx_{\ell}}(\tau_{\vx_{\ell}})} \right | \\
	&\quad \le O_{P} \left ( \| \hat{s}_{\vx}(\hat{\tau}_{\vx}) - s_{\vx}(\tau_{\vx}) \|_{\cX_{0}} \bigvee \| \hat{s}''_{\vx}(\hat{\tau}_{\vx}) - s''_{\vx}(\tau_{\vx}) \|_{\cX_{0}} \right ), \quad \text{and} \\
	&\| \hat{s}^{(r)}_{\vx}(\hat{\tau}_{\vx}) - s^{(r)}_{\vx}(\tau_{\vx}) \|_{\cX_{0}} \le \| \hat{s}^{(r)}_{\vx}(\tau) - s_{\vx}^{(r)}(\tau) \|_{[\epsilon,1-\epsilon] \times \cX_{0}} + 
	\underbrace{\| s_{\vx}^{(r)}(\hat{\tau}_{\vx}) - s_{\vx}^{(r)}(\tau_{\vx}) \|_{\cX_{0}}}_{=O(\|\hat{\tau}_{\vx} - \tau_{\vx} \|_{\cX_{0}})} \\
	&\qquad \qquad = O_{P}(n^{-1/2}h^{-5/2}\sqrt{\log n} + h) \quad \text{for $r=0,2$},
	\end{split}
	\]
	where we have used Lemma \ref{UnifQ} in the last line. 
	
	Next, we note that 
	\begin{align*}
	&\max_{1 \le k,\ell \le L} \left | \vx_k^T\hat{J}_{\vx_k}^{-1}\left [ \frac{1}{n}\sum_{i=1}^n\vX_i\vX_i^T \right ] \hat{J}_{\vx_{\ell}}^{-1} \vx_{\ell} -\vx_k^TJ_{\vx_k}^{-1}\E\left[\vX_i\vX_i^T\right]J_{\vx_{\ell}}^{-1} \vx_{\ell} \right | \\
	&\le O_P\left(\max_{1 \le k \le L} \|\hat{J}_{\vx_k}^{-1}-J_{\vx_k}^{-1}\|_{\op}  \bigvee \left\|\frac{1}{n}\sum_{i=1}^n\vX_i\vX_i^T-\E\left[\vX \vX^T\right]\right\|_{\op} \right) \\
	&=O_P(n^{-1/2}h^{-3/2}\sqrt{\log n}+h^{2}),
	\end{align*}
	where we have used Lemma \ref{lem: Jacobian} in the last line. 
	
	Finally, observe that 
	\[
	\begin{split}
	&\left | \E_{|\cD_n}\left[ \frac{1}{h} K'\left(\frac{\hat{\tau}_{\vx_k}-U}{h}\right)K'\left(\frac{\hat{\tau}_{\vx_{\ell}}-U}{h}\right)\right] - \E_{|\cD_n}\left[ \frac{1}{h} K'\left(\frac{\tau_{\vx_k}-U}{h}\right)K'\left(\frac{\hat{\tau}_{\vx_{\ell}}-U}{h}\right)\right] \right | \\
	&\le \| K'' \|_{\infty} h^{-1} | \hat{\tau}_{\vx_{k}} - \tau_{\vx_{k}}| \E_{|\cD_n}\left [ \left | \frac{1}{h}K'\left(\frac{\hat{\tau}_{\vx_{\ell}}-U}{h}\right) \right | \right] = O_{P}(n^{-1/2}h^{-5/2}\sqrt{\log n} + h)
	\end{split}
	\]
	uniformly in $1 \le k,\ell \le L$. Likewise, we have
	\[
	\begin{split}
	&\left | \E_{|\cD_n}\left[ \frac{1}{h} K'\left(\frac{\tau_{\vx_k}-U}{h}\right)K'\left(\frac{\hat{\tau}_{\vx_{\ell}}-U}{h}\right)\right] - \E_{|\cD_n}\left[ \frac{1}{h} K'\left(\frac{\tau_{\vx_k}-U}{h}\right)K'\left(\frac{\tau_{\vx_{\ell}}-U}{h}\right)\right] \right |\\
	&\quad = O_{P}(n^{-1/2}h^{-5/2}\sqrt{\log n} + h)
	\end{split}
	\]
	uniformly in $1 \le k,\ell \le L$. Combining these estimates, we obtain the conclusion of the lemma. 
\end{proof}

\begin{lemma}
	\label{lem: covariance D}
	Under Assumptions \ref{asmp: assumption1} and \ref{asmp: assumption2}, we have
	\[
	\max_{1 \le k,\ell \le M} |D_{k}^{T}(\hat{\Sigma} - \Sigma)D_{\ell}| = O_{P}(n^{-1/2}h^{-5/2}\sqrt{\log n}+h). 
	\]
\end{lemma}

\begin{proof}
	This follows from the observation that 
	\[
	\begin{split}
	&\max_{1 \le k,\ell \le M} |D_{k}^{T}(\hat{\Sigma} - \Sigma)D_{\ell}| = \max_{1 \le k,\ell \le M} \left | \sum_{k' \in S_{k}}\sum_{\ell' \in S_{\ell}} D_{k,k'}(\hat{\Sigma}_{k',\ell'} - \Sigma_{k',\ell'})D_{\ell,\ell'} \right | \\
	&\quad \le \max_{1 \le k \le M} |S_{k}|^{2} \| D \|_{\infty} \| \hat{\Sigma} - \Sigma \|_{\infty} = O_{P}(n^{-1/2}h^{-5/2}\sqrt{\log n}+h). 
	\end{split}
	\]
\end{proof}

We are now in position to prove Theorem \ref{thm: boot}. 

\begin{proof}[Proof of Theorem \ref{thm: boot}]
	Let $\hat{G}$ be an $L$-dimensional random vector such that conditionally on $\cD_{n}$, $\hat{G} \sim N(0,\hat{\Sigma})$. 
	We begin with noting that 
	\begin{align*}
	&\sup_{b\in \R^M}\left|\P_{|\cD_n}\left(n^{-1/2} {\textstyle \sum}_{i=1}^n\hat{A}\hat{\Psi}_i\le b \right)-\P\left(AG\le b\right)\right| \le  \sup_{b\in \R^M}\Big|\P_{|\cD_n}\left(n^{-1/2}{\textstyle \sum}_{i=1}^n\hat{A}\hat{\Psi}_i\le b \right)-\\
	&\P_{|\cD_n}\left(\hat{A}\hat{G} \le b \right)\Big| + \sup_{b\in \R^M}\left|\P_{|\cD_n}(\hat{A}\hat{G}\le b)-\P_{|\cD_n}(A\hat{G}\le b) \right| + \sup_{b\in \R^M}\left|\P_{|\cD_n}(A\hat{G}\le b )-\P(AG\le b)\right|\\
	&:= I + II +III.  
	\end{align*}
	
	We first analyze $II$ and $III$. 
	In view of the Gaussian comparison inequality (cf. Lemma \ref{lem: Gaussian comparison}), to show that $II \vee III = o_{P}(1)$, it suffices to verify that 
	\begin{equation}
	\left [ \| \hat{A} \hat{\Sigma} \hat{A}^{T} -A\hat{\Sigma} A^{T} \|_{\infty} \vee \| A \hat{\Sigma}A^{T} - A\Sigma A^{T} \|_{\infty} \right ] \log^2 M = o_{P}(1).
	\label{eq: Gaussian comparison}
	\end{equation}
	Indeed, by Lemma \ref{lem: covariance D} and the assumption (i) of the theorem, we deduce that the bracket on the left is $O_{P}(n^{-1/2}h^{-5/2}\sqrt{\log n} + h)$. Thus, (\ref{eq: Gaussian comparison}) holds under our assumption. 
	
	To show that $I =o_{P}(1)$, we apply Proposition 2.1 in \cite{chernozhukov2017central} conditionally on $\cD_{n}$ (recall that conditionally on $\cD_{n}$, the vectors $\hat{\Psi}_{1},\dots,\hat{\Psi}_{n}$ are independent with mean zero). By construction, $n^{-1}\sum_{i=1}^{n}\E_{|\cD_n}[(\hat{A}_{k}^{T}\hat{\Psi}_{i})^{2}] = \hat{A}_{k}^{T} \hat{\Sigma} \hat{A}_{k} = D_{k}^{T}\hat{\Sigma} D_{k}/\hat{\Gamma}_{k}^{2}$ is bounded way from $0$ uniformly in $k$ with probability approaching one. 
	Similarly to the proof of Theorem \ref{thm: HDCLT}, we can verify that $ \max_{1 \le k \le M} n^{-1}\sum_{i=1}^{n} \E_{|\cD_n}[|\hat{A}_{k}^{T}\hat{\Psi}_{i}|^{2+r}] = O_{P}(h^{-r/2})$ for $r=1,2$. Finally,
	\[
	\max_{1 \le i \le n}  \E_{|\cD_n}\left [\max_{1 \le k \le M}|\hat{A}_{k}^{T}\hat{\Psi}_{i}|^{q} \right] \le O_{P}(h^{-q/2}) \max_{1 \le i \le n} \| \vX_{i} \|^{q} = O_{P}(n h^{-q/2}). 
	\]
	Hence, applying Proposition 2.1 in \cite{chernozhukov2017central}, we see that $I = o_{P}(1)$ as soon as 
	\[
	\frac{\log^7 (Mn)}{n^{1-2/q}h} \bigvee \frac{\log^3 (Mn)}{n^{1-4/q}h} \to 0,
	\]
	but this is satisfied under our assumption. This completes the proof.
\end{proof}

\begin{proof}[Proof of Proposition \ref{prop: HDCLT normalized}]
	Theorem \ref{thm: HDCLT} implies that 
	\[
	\sup_{b \in \R^{M}} \left | \P \left ( D \sqrt{nh^{3}} (\hat{m}(\vx_{\ell}) - m(\vx_{\ell}))_{\ell=1}^{L} \le b \right ) - \P (DG \le b) \right | \to 0. 
	\]
	Since the variances of the coordinates of $G$ are bounded, we see that $\E[\| G \|_{\infty}] = O(\sqrt{\log M})$ by Lemma 2.2.2 in \cite{vaart1996weak}. Hence, we have 
	\[
	\left \|  D \sqrt{nh^{3}} (\hat{m}(\vx_{\ell}) - m(\vx_{\ell}))_{\ell=1}^{L} \right\|_{\infty} = O_{P}(\sqrt{\log M}).
	\]
	Combining Condition (i) in the statement of Theorem \ref{thm: boot}, we see that 
	\[
	\left \|  (\hat{\Gamma}^{-1}-\Gamma^{-1})D \sqrt{nh^{3}} (\hat{m}(\vx_{\ell}) - m(\vx_{\ell}))_{\ell=1}^{L} \right\|_{\infty} = o_{P}(1/\sqrt{\log M}).
	\]
The rest of the proof is analogous to the last part of Theorem \ref{thm: HDCLT}. We omit the details for brevity. 
\end{proof}

\subsubsection{Proof of Theorem \ref{thm: nonparametric bootstrap}}
We start with proving the following uniform Bahadur representation for the quantile regression estimator $\hat{\beta}^*$ based on the nonparametric bootstrap samples $(Y_i^*,\vX_i^*)_{i=1}^n$. We define $U_i^*=F(Y_i^*\mid \vX_i^*)$ where $F(y\mid \vx)$ is the conditional distribution function of $Y$ given $\vX$ (see Remark \ref{rem: bahadur}).
\begin{lemma}\label{lem: bahadur nonparametric bootstrap}
	Suppose Assumption \ref{asmp: assumption1} holds. Then we have, for arbitrarily small $\gamma>0$,
	\begin{equation}\label{eq: bahadur nonparametric bootstrap}
	\hat{\beta}^*(\tau)-\beta(\tau)=J(\tau)^{-1}\left[\frac{1}{n}\sum_{i=1}^n\{\tau-I(U^*_i \le \tau)\}\vX_i^*\right]+O_{P}(n^{-3/4+\gamma} ),
	\end{equation}	
	uniformly in $\tau \in [\epsilon/2,1-\epsilon/2]$.
\end{lemma}
\begin{proof}
	We will prove the following equivalent form of  (\ref{eq: bahadur nonparametric bootstrap}),
	\begin{equation}
	\hat{\beta}^*(\tau)-\beta(\tau)=J(\tau)^{-1}\left[\frac{1}{n}\sum_{i=1}^n \pi_i \{\tau-I(U_i \le \tau)\} \vX_i\right]+O_{P}(n^{-3/4+\gamma} ),
	\end{equation}
	where $(\pi_1,\dots,\pi_n)$ is a multinomial random vector with parameters $n$ and (probabilities) $1/n,\dots,1/n$. We will divide the proof into two steps. In the following proof, $C$ is a generic constant independent of $n$ whose value may vary from line to line. 
	
	\textbf{Step 1}. In this step, we will show that $\sup_{\tau \in [\epsilon/2,1-\epsilon/2]}\|\hat{\beta}^*(\tau)-\beta(\tau)\|=O_P(n^{-1/2})$. 
	To this end, we introduce the following quantities
	\begin{align*}
%	\phi_i(\beta,\tau)&:=\{ \tau-I(Y_i\le\vX_i^T\beta) \} \vX_i, \\
	R_{n}&:=\Big\{  (\tau,\beta) \in \mathcal{U} \times \R^d: \|\beta-\beta(\tau)\|\le r_n   \Big\}, \\
	\Upsilon_0&:=\sup_{\tau\in \mathcal{U}}\left \|\frac{1}{\sqrt{n}} \sum_{i=1}^n\pi_i  \{ \tau - I(Y_i \le \vX_i^T \beta (\tau))\} \vX_i\right\|, \\
	\Upsilon_1&:=\sup_{(\tau,\beta) \in R_n}  \Bigg\| \frac{1}{\sqrt{n}}\sum_{i=1}^n \big \{\pi_i \{ I (Y_i\le\vX_i^T\beta)-I(Y_i\le\vX_i^T\beta(\tau)) \}\vX_i \\
	&\qquad \qquad \qquad  -\E [ \{I(Y \le \vX^T \beta ) - \tau \} \vX] \big \} \Bigg\|,\\
	\Upsilon_2&:= \sup_{(\tau,\beta) \in R_n} n^{1/2}\Big\|  \E [ \{ \tau - I(Y \le \vX^T \beta) \} \vX]-J(\tau)(\beta-\beta(\tau))   \Big\|. 
	\end{align*}
	where we define $\mathcal{U}:=[\epsilon/2,1-\epsilon/2]$ and $(r_n)_{n=1}^{\infty}$ is a sequence of constants to be specified. 
	
	%According to Lemma 3 in \cite{belloni2019conditional}, it suffices to show $\Upsilon_0=O_P(1)$, $\Upsilon_1=o_P(1)$ and $\Upsilon_2=o_P(1)$ when $r_n=O(n^{-1/2})$. We bound the three terms as follows.
	
	In view of the proof of Lemma 3 in \cite{belloni2019conditional}, it suffices to show that $\Upsilon_0=O_P(1)$, $\Upsilon_1=o_P(1)$ and $\Upsilon_2=o_P(1)$ when $r_n=O(n^{-1/2})$. We shall bound the three terms in what follows. 
	
	The term $\Upsilon_0$ is the supremum of a multiplier empirical process, and we will apply a multiplier inequality developed in \cite{han2019convergence}. Define the function class
	\[
	\F_1:=\Big\{ (y,\vx) \mapsto (\tau-I\{y\le \vx^T\beta(\tau)\}) \alpha^T\vx:   \tau\in \mathcal{U}, \alpha\in \bS^{d-1}   \Big\},
	\]
	where $\bS^{d-1} = \{ \vx \in \R^d : \| \vx \| = 1 \}$. 
	Then $\Upsilon_0=\|n^{-1/2}\sum_{i=1}^n\pi_if(\vX_i,Y_i)\|_{\F_1}$. We first verify that
	\begin{equation}\label{eq: epsilon 0}
	\E \big[\|\G_n\|_{\F_1}\big]\le Cd^{1/2}.
	\end{equation}
	This can be proved as follows
	\begin{align*}
	\E\big[\|\G_n\|_{\F_1}\big]&\le \E\Bigg[\sup_{\substack{ \tau \in \mathcal{U}}}\Big[\sum_{j=1}^d\left(\mathbb{G}_n\big [  \{ \tau-I(Y_i \le \vX_i^T\beta(\tau)) \} X_{ij} \big ] \right)^2\Big]^{1/2}\Bigg]\\
	&\le \left[\sum_{j=1}^d\E \Big[\sup_{\substack{ \tau \in \mathcal{U}}}\left(\mathbb{G}_n\big [ \{ \tau-I(U_i \le \tau) \} X_{ij} \big ] \right)^2\Big]\right]^{1/2}.
	\end{align*}
	For any $1\le j \le d$, by Theorem 2.14.1 in \cite{vaart1996weak},
	\[
	\E \Big[\sup_{\substack{ \tau \in \mathcal{U}}}\left(\mathbb{G}_n\big [ \{ \tau-I(U_i \le \tau) \} X_{ij} \big ] \right)^2\Big] \le C \mathbb{E}\left[|X_{1j}|^2\right]=O(1). 
	\]
	This leads to (\ref{eq: epsilon 0}).
	
	Now, by Lemma 2.3.6 in \cite{vaart1996weak}, (\ref{eq: epsilon 0}) implies the following bound for the symmetrized empirical process 
	\[
	\E \Big[\Big\|\sum_{i=1}^n \xi_i f(\vX_i,Y_i)\Big\|_{\F_1}\Big]\le  Cd^{1/2}n^{1/2},
 	\]
	where $\xi_i$ ($1\le i\le n$) are i.i.d.\ Rademacher random variables independent of the data.
	Given the above bound, we can apply Corollary 1 in \cite{han2019convergence} (recall $d$ is fixed)
	%(with $\Psi_n(k)=cd^{1/2}k^{1/2}$ in the proof there) 
	to conclude that
	\[
	\E [\Upsilon_0] \le C d^{1/2} \|\pi_1\|_{2,1}=O(d^{1/2}).
	\]
	This implies that $\Upsilon_0=O_P(1)$.
	
	We can bound $\Upsilon_1$ similarly to $\Upsilon_0$. Define the function class
	\[
	\begin{split}
	\F_2=\Bigl\{  (y,\vx) &\mapsto \{ I(y\le \vx^T\beta)-I(y\le \vx^T\beta(\tau)) \} \alpha^T\vx - \E\big[\{ I(Y\le \vX^T\beta) - \tau \} \alpha^T\vX \big] \\
	&\qquad : (\tau,\beta) \in R_n, \alpha\in \bS^{d-1}     \Bigr\}, 
	\end{split}
	\]
	so that  $\Upsilon_1=\|n^{-1/2}\sum_{i=1}^n\pi_if(\vX_i,Y_i)\|_{\F_2}$. Applying Lemma \ref{lem: maximal inequality}, we have
	\[
	\E \Big[ \Big\|\sum_{i=1}^nf(\vX_i,Y_i)\Big\|_{\F_2}\Big]=o(n^{1/4}\log n).
	\]
	Now we apply Corollary 1 in \cite{han2019convergence} (with $\Psi_n(t)=Ct^{1/4+\gamma}$ for some arbitrary small $\gamma>0$ in the proof there) to conclude that
	\[
	\E[\Upsilon_1]\le Cn^{-1/4+\gamma},
	\]
	which implies that $\Upsilon_1=O_P(n^{-1/4+\gamma})=o_P(1)$.
	
	For $\Upsilon_2$, we proceed as in Lemma 33 of \cite{belloni2019conditional} to see that
	\[
	\Upsilon_2\le \sup_{\substack{(\tau,\beta) \in R_n\\ \alpha \in \bS^{d-1}}} \sqrt{n}\Big|\E\big[\big(f(\vX^T\tilde{\beta}(\tau) \mid \vX)-f(\vX\beta(\tau) \mid \vX)\big)\cdot (\alpha^T\vX) \cdot \big(\vX^T(\beta-\beta(\tau))\big) \big] \Big|,
	\]
	where $\tilde{\beta}(\tau)$ lies on the line segment between $\beta(\tau)$ and $\beta$.
	We further bound the right-hand side as 
	\begin{align*}
	\Upsilon_2&\le C \sqrt{n} \sup_{\substack{(\tau,\beta) \in R_n\\ \alpha \in \bS^{d-1}}}\E \Big[ |\alpha^T\vX| \big| (\tilde{\beta}(\tau)-\beta(\tau))^T\vX\vX^T(\beta-\beta(\tau))  \big|\Big]\\
	&\le C \sqrt{n} \sup_{\substack{(\tau,\beta) \in R_n\\ \alpha \in \bS^{d-1}}}  \sqrt{\E\left[\alpha^T\vX\vX^T\alpha\right]}\cdot \sqrt{\E\Big[ \big| (\tilde{\beta}(\tau)-\beta(\tau))^T\vX\vX^T(\beta-\beta(\tau))  \big|^2\Big] }\\
	&\le C \sqrt{n} \cdot O(1) \cdot r_n^2\sqrt{\E\|\vX\|^4}=O(n^{-1/2})=o(1).
	\end{align*}
	%By Lemma 3 in \cite{belloni2019conditional}, we proved the uniform rate of convergence of $\hat{\beta}^*(u)$.
	
	These bounds imply the conclusion of this step, in view of the proof of Lemma 3 in \cite{belloni2019conditional}. 
	
	\textbf{Step 2.} We finish the proof of the lemma. Define
	\[
	r(\tau)=J(\tau)\big(\hat{\beta}^*(\tau)-\beta(\tau)\big)-\frac{1}{n} \sum_{i=1}^n \pi_i \{ \tau-I(U_i \le \tau)\} \vX_i.
	\]
	It suffices to show that $\sup_{\tau \in \mathcal{U}}\|r(\tau)\|=O_P(n^{-3/4+\gamma})$ for arbitrarily small $\gamma>0$.
	We note that, by Step 1, for any $B_n \to \infty$ arbitrarily slowly, $(\tau,\beta^*(\tau))\in R_n$ with $r_n=B_nn^{-1/2}$ holds with probability approaching one. Hence,  taking such $R_n$, with probability $1-o(1)$,
	\[
	\sup_{\tau \in \mathcal{U}}\|r(\tau)\|\le n^{-1/2}(\Upsilon_1+\Upsilon_2+\Upsilon_3),
	\]
	where $\Upsilon_3:= \sup_{\tau \in \mathcal{U}} \| n^{-1/2}\sum_{i=1}^n\pi_i  \{ \tau - I(Y_i \le \vX^T_i \hat{\beta}^* (\tau)) \} \vX_i\|$.
	By Step 1,  taking $B_n \to \infty$ sufficiently slowly, we have 
	\[
	\Upsilon_1=O_P(n^{-1/4+\gamma}) \quad \text{and} \quad \Upsilon_2=O_P(n^{-1/2+\gamma}).
	\]
	To bound $\Upsilon_3$, from the proof of Lemma 34 in \cite{belloni2019conditional},we can deduce that
	\begin{align*}
	\Upsilon_3\le \frac{d}{\sqrt{n}} \max_{1\le i \le n}\|\pi_i\vX_i\|.
	\end{align*}
	We further bound the right-hand side as 
	\begin{align*}
	\Upsilon_3&\le \frac{d}{\sqrt{n}} \cdot \max_{1\le i \le n}|\pi_i|\cdot \max_{1\le i \le n}\|\vX_i\|\\
	&= \frac{d}{\sqrt{n}} \cdot o_P\Big(\frac{\log n}{\log\log n}\Big)\cdot o_P(n^{1/q}) \qquad (\text{Section 4 of \cite{raab1998balls}})\\
	&=o_P( d n^{1/q-1/2}\log n/\log\log n).
	\end{align*}
	Hence we have shown that
	\begin{align*}
	\sup_{\tau \in \mathcal{U}}\|r(\tau)\|&=O_P(n^{-3/4+\gamma}+n^{-1+\gamma}+d n^{1/q-1}\log n/\log\log n) \\
	&=O_P(n^{-3/4+\gamma}),
	\end{align*}
	which finishes the proof.
\end{proof}
Define $\bar{\Psi}:=n^{-1}\sum_{i=1}^{n}\Psi_i$ and $\tilde{\Sigma}:=n^{-1}\sum_{i=1}^n(\Psi_i-\bar{\Psi})(\Psi_i-\bar{\Psi})^T$.
\begin{lemma}
	\label{lem: nb covariance}
	Under Assumption \ref{asmp: assumption1}, we have 
	\[
	\| \tilde{\Sigma} - \Sigma \|_{\infty} = O_{P}(n^{-1/2}h^{-1/2}\sqrt{\log n} + n^{2/q-1}h^{-1}\log n).
	\]
\end{lemma}
\begin{proof}
Recall that $J_{\vx}=J(\tau_{\vx})$. 
	The difference $\tilde{\Sigma}_{j,k}-\Sigma_{j,k}$ can be decomposed as
	\begin{align*}
	&\Bigg [ \frac{s_{\vx_k}(\tau_{\vx_k})s_{\vx_{j}}(\tau_{\vx_{j}})}{s''_{\vx_k}(\tau_{\vx_k})s''_{\vx_{j}}(\tau_{\vx_{j}})} \Bigg ]
	\Bigg\{ \frac{1}{n}\sum_{i=1}^n\frac{1}{h}K'\left(\frac{\tau_{\vx_k}-U_i}{h}\right)K'\left(\frac{\tau_{\vx_{j}}-U_i}{h}\right)\vx_k^TJ_{\vx_k}^{-1}\vX_i \vX_i^{T} J_{\vx_{j}}^{-1}\vx_{j} \\ 
&\qquad	 -\E\left[\frac{1}{h}K'\left(\frac{\tau_{\vx_k}-U}{h}\right)K'\left(\frac{\tau_{\vx_{j}}-U}{h}\right)\right] \vx_k^TJ_{\vx_k}^{-1}\E[\vX \vX^{T}] J_{\vx_{j}}^{-1}\vx_{j}\\
&\qquad	 -\frac{1}{h}\Bigg[\sum_{i=1}^nK'\left(\frac{\tau_{\vx_k}-U_i}{h}\right)\vx_k^TJ_{\vx_k}^{-1}\vX_i \Bigg] 
	\Bigg[\sum_{i=1}^nK'\left(\frac{\tau_{\vx_j}-U_i}{h}\right)\vx_j^TJ_{\vx_j}^{-1}\vX_i \Bigg]
	\Bigg \}.
	\end{align*}
	We define 
	\begin{align*}
	\mathllap{I_{jk}}&:=\frac{1}{n}\sum_{i=1}^n\frac{1}{h}K'\left(\frac{\tau_{\vx_k}-U_i}{h}\right)K'\left(\frac{\tau_{\vx_{j}}-U_i}{h}\right)\vx_k^TJ_{\vx_k}^{-1}\vX_i \vX_i^{T} J_{\vx_{j}}^{-1}\vx_{j} \\
	&\qquad -\E\left[\frac{1}{h}K'\left(\frac{\tau_{\vx_k}-U}{h}\right)K'\left(\frac{\tau_{\vx_{j}}-U}{h}\right)\right] \vx_k^TJ_{\vx_k}^{-1}\E[\vX \vX^{T}] J_{\vx_{j}}^{-1}\vx_{j}; \\
	\mathllap{II_{jk}}&:=\frac{1}{h}\Bigg[\sum_{i=1}^nK'\left(\frac{\tau_{\vx_k}-U_i}{h}\right)\vx_k^TJ_{\vx_k}^{-1}\vX_i \Bigg] 
	\Bigg[\sum_{i=1}^nK'\left(\frac{\tau_{\vx_j}-U_i}{h}\right)\vx_j^TJ_{\vx_j}^{-1}\vX_i \Bigg].
	\end{align*}
	We first bound $I_{jk}$. Define the function class
	\begin{align*}
	\F_1 :=\Bigl\{ (u,\vx) &\mapsto h^{-1}K'((\tau_1-u)/h)K'((\tau_{2}-u)/h)(\vx_1^TJ_{\vx_1}^{-1}\vx)( \vx^{T} J_{\vx_2}^{-1}\vx_2):\\
	&\qquad  \tau_1, \tau_2\in [\epsilon,1-\epsilon], \  \vx_1, \vx_2\in \mathcal{X}_0\Bigr\}.
	\end{align*}
	Then we have $\max_{1 \le j,k \le L}|I_{jk}|\le \sup_{f\in \F_1}|n^{-1}\sum_{i=1}^n\{f(U_i,\vX_{i})-\E[f(U,\vX)]\}|$. Applying Lemma \ref{lem: maximal inequality}, we have
	\[
	\sup_{f\in \F_1}\left | n^{-1}\sum_{i=1}^n\{f(U_i,\vX_{i})-\E[f(U,\vX)]\} \right | =O_P\big(n^{-1/2}h^{-1/2}\sqrt{\log n}+n^{2/q-1}h^{-1}\log n\big).
	\]
	Hence we have shown that
	\[
	\max_{1 \le j,k \le L}|I_{jk}|=O_P\big(n^{-1/2}h^{-1/2}\sqrt{\log n}+n^{2/q-1}h^{-1}\log n\big).
	\]
	
	For $II_{jk}$, we define the following function class
	\begin{align*}
	\F_2:=\Bigl\{   (u,\vx) \mapsto h^{-1/2} K' ((\tau-u)/h)\vx_0^TJ_{\vx_0}^{-1}\vx :\tau\in [\epsilon,1-\epsilon],  \vx_0 \in \mathcal{X}_0\Bigr\}.
	\end{align*}
	Then we have 
	\[
	\max_{1 \le j,k \le L}|II_{jk}|\le  \sup_{f\in \F_2}\left\{n^{-1}\sum_{i=1}^n\{f(U_i,\vX_{i})-\E[f(U,\vX)]\}\right\}^2.
	\]
	Similarly to the previous case, by using Lemma \ref{lem: maximal inequality}, we can show that
	\[
	\sup_{f\in \F_2}\left | n^{-1}\sum_{i=1}^n\{f(U_i,\vX_{i})-\E[f(U,\vX)]\} \right | =O_P(n^{-1/2}\sqrt{\log n}),
	\]
	which implies that $\max_{1 \le j,k \le L}|II_{jk}|=O_P(n^{-1}\log n)$. Combining the above bounds, we obtain the desired result. 
\end{proof}
Using Lemma \ref{lem: nb covariance}, we have the following result; cf. the proof of Lemma \ref{lem: covariance D}.
\begin{lemma}
	\label{lem: covariance D nb}
	Under Assumptions \ref{asmp: assumption1} and \ref{asmp: assumption2}, we have
	\[
	\max_{1 \le k,\ell \le M} |D_{k}^{T}(\tilde{\Sigma} - \Sigma)D_{\ell}| = O_{P}(n^{-1/2}h^{-1/2}\sqrt{\log n} + n^{2/q-1}h^{-1}\log n).
	\]
\end{lemma}

We are now in position to prove Theorem \ref{thm: nonparametric bootstrap}. 
\begin{proof}[Proof of Theorem \ref{thm: nonparametric bootstrap}]
	We begin with noting that,  using the Bahadur representation in Lemma \ref{lem: bahadur nonparametric bootstrap},we can establish  the following asymptotic linear representation for the bootstrap mode estimator by a similar analysis to the proof of Theorem \ref{thm: HDCLT} coupled with the multiplier inequality techniques as in the proof of Lemma \ref{lem: bahadur nonparametric bootstrap}:
	\begin{equation}
	\begin{aligned}
	(\hat{m}^*(\vx_{\ell}) - m(\vx_{\ell}))_{\ell=1}^{L} &= (nh^{3/2})^{-1} \textstyle\sum_{i=1}^{n} \Psi^*_{i}+R^*_n,\\ \sqrt{nh^3}\|R_n^*\|_{\infty}&=O_P(h^{1/2}+n^{-1/2+\gamma}h^{-5/2}+n^{1/2}h^{7/2}),
	\end{aligned} 
	\end{equation}
	where $\Psi_{i}^*=(\psi_{\vx_{1}}(U_i^*,\vX_i^*),\cdots,\psi_{\vx_{L}}(U_i^*,\vX_i^*))^T$. Hence we have
	\begin{align*}
	\sqrt{nh^3}\hat{A}(\hat{m}^*(\vx_{\ell}) - \hat{m}(\vx_{\ell}))_{\ell=1}^{L} &= n^{-1/2} \textstyle\sum_{i=1}^{n}\hat{A} (\Psi^*_{i}-\Psi_i)+\sqrt{nh^3}\hat{A}(R^*_n-R_n)\\
	&:=n^{-1/2} \textstyle\sum_{i=1}^{n} (\hat{A}\Psi^*_{i}-\hat{A}\bar{\Psi})+\sqrt{nh^3}\hat{A}\tilde{R}_n.
	\end{align*}
	Now, we divide the rest of the proof into two steps.
	
	\textbf{Step 1.} 
	We will show that 
	\begin{equation}\label{eq: thm3 step1}
	\sup_{b\in \R^M}\left|\P_{|\cD_n}\left(n^{-1/2}{\textstyle \sum}_{i=1}^n(\hat{A}\Psi^*_{i}-\hat{A}\bar{\Psi})\le b\right)-\P\left(AG\le b\right)\right|\stackrel{P}{\to} 0.
	\end{equation}
	We note that
	\begin{align*}
	&\sup_{b\in \R^M}\left|\P_{|\cD_n}\left(n^{-1/2}{\textstyle \sum}_{i=1}^n(\hat{A}\Psi^*_{i}-\hat{A}\bar{\Psi})\le b\right)-\P\left(AG\le b\right)\right|\\
	&\le \sup_{b\in \R^M}\left|\P_{|\cD_n}\left(n^{-1/2}{\textstyle \sum}_{i=1}^n(\hat{A}\Psi^*_{i}-\hat{A}\bar{\Psi})\le b\right)-\P_{|\cD_n}\left(\hat{A}\tilde{G}\le b\right)\right|\\
	&+ \sup_{b\in \R^M}\left|\P_{|\cD_n}\left(\hat{A}\tilde{G}\le b\right)-\P_{|\cD_n}\left(A\tilde{G}\le b\right)\right|+ \sup_{b\in \R^M}\left|\P_{|\cD_n}\left(A\tilde{G}\le b\right)-\P\left(AG\le b\right)\right|\\&:= I+II+III,
	\end{align*}
	where $\tilde{G}\sim N(0,\tilde{\Sigma})$ and recall $\tilde{\Sigma}:=n^{-1}\sum_{i=1}^n(\Psi_i-\bar{\Psi})(\Psi_i-\bar{\Psi})^T$.
	
	We first analyze $II$ and $III$. 
	In view of the Gaussian comparison inequality (cf. Lemma \ref{lem: Gaussian comparison}), to show that $II \vee III = o_{P}(1)$, it suffices to verify that 
	\begin{equation}
	\left [ \| \hat{A} \tilde{\Sigma} \hat{A}^{T} -A\tilde{\Sigma} A^{T} \|_{\infty} \vee \| A \tilde{\Sigma}A^{T} - A\Sigma A^{T} \|_{\infty} \right ] \log^2 M = o_{P}(1).
	\label{eq: Gaussian comparison nb}
	\end{equation}
	Indeed, by Lemma \ref{lem: covariance D nb} and Condition (i) of the theorem, we can deduce that the bracket on the left hand side is $O_{P}(n^{-1/2}h^{-5/2}\sqrt{\log n} + h)$. Thus, (\ref{eq: Gaussian comparison nb}) holds under our assumption. 
	
	To show that $I =o_{P}(1)$, we apply Proposition 2.1 in \cite{chernozhukov2017central} conditionally on $\cD_{n}$ (recall that conditionally on $\cD_{n}$, the vectors $\Psi^*_{1}-\bar{\Psi},\dots,\Psi^*_{n}-\bar{\Psi}$ are independent with mean zero). By construction, $n^{-1}\sum_{i=1}^{n}\E_{|\cD_n}[(\hat{A}_{k}^{T}\Psi^*_{i}-\hat{A}_{k}^{T}\bar{\Psi})^{2}] = \hat{A}_{k}^{T} \tilde{\Sigma} \hat{A}_{k} = D_{k}^{T}\tilde{\Sigma} D_{k}/\hat{\Gamma}_{k}^{2}$ is bounded away from zero uniformly over $1 \le k \le M$ with probability approaching one. 
	Similarly to the proof of Theorem \ref{thm: HDCLT}, we can verify that $ \max_{1 \le k \le M} n^{-1}\sum_{i=1}^{n} \E_{|\cD_n}[|\hat{A}_{k}^{T}\Psi^*_{i}-\hat{A}_{k}^{T}\bar{\Psi}|^{2+r}] = O_{P}(h^{-r/2})$ for $r=1,2$.
	Finally,
	\begin{align*}
	\max_{1 \le i \le n}  \E_{|\cD_n}\left [\max_{1 \le k \le M}|\hat{A}_{k}^{T}\Psi^*_{i}-\hat{A}_{k}^{T}\bar{\Psi}|^{q} \right] 
	&\le O(1)  \max_{1 \le i \le n}\max_{1 \le k \le M}|\hat{A}_{k}^{T}\Psi_{i}|^{q}  \\
	&\le O_{P}(h^{-q/2}) \max_{1 \le i \le n} \| \vX_{i} \|^{q} = O_{P}(n h^{-q/2}). 
	\end{align*}
	Hence, applying Proposition 2.1 in \cite{chernozhukov2017central}, we see that $I = o_{P}(1)$ as soon as 
	\[
	\frac{\log^7 (Mn)}{n^{1-2/q}h} \bigvee \frac{\log^3 (Mn)}{n^{1-4/q}h} \to 0,
	\]
	but this is satisfied under our assumption. This completes  Step 1.
	
	\textbf{Step 2.} We finish the proof by a similar analysis as Step 2 in the proof of Theorem \ref{thm: HDCLT}.
	Define $\tilde{\delta}_{n} =h^{1/2}+n^{-1/2+\gamma}h^{-5/2}+n^{1/2}h^{7/2}$.
	Combining the analysis before Step 1 and the fact that $\sqrt{nh^{3}}\| R_n \|_{\infty} = O_{P}(h^{1/2}+n^{-1/2}h^{-5/2}\log n+n^{1/2}h^{7/2})$,  we have $\sqrt{nh^{3}}\| \tilde{R}_n \|_{\infty} = O_{P}(\tilde{\delta}_{n})$. Similarly to Step 2 in the proof of Theorem \ref{thm: HDCLT}, we can show that
	$\sqrt{nh^{3}} \| \hat{A} \tilde{R}_{n} \|_{\infty} = O_{P}(\tilde{\delta}_{n}). $
	The rest of the proof is analogous to the last part of Theorem \ref{thm: HDCLT}. We omit the details for brevity. 
\end{proof}

\subsubsection{Proofs for Appendix A}
\begin{proof}[Proof of Proposition \ref{prop: Gumbel}]
	Since $K$ is supported in $[-1,1]$, if $| \tau_{\vx_{k}} - \tau_{\vx_{\ell}} | > 2h$, then 
	\[
	\E \left [ K'\left ( \frac{\tau_{\vx_{k}}-U}{h} \right ) K'\left ( \frac{\tau_{\vx_{\ell}}-U}{h} \right ) \right ] = 0.
	\]
	Thus, $\Sigma = \diag \{ \sigma_{\vx_{1}}^{2},\dots,\sigma_{\vx_{L}}^{2} \}$, so that Theorem \ref{thm: HDCLT} implies that 
	\[
	\sup_{b \in \R}\left | \P (\zeta_n \le b) - \P\big(\max_{1 \le \ell \le L} |W_{\ell}| \le b\big) \right | \to 0,
	\]
	where $W_1,\dots,W_L \sim N(0,1)$ i.i.d.\ 
	The rest of the proof follows from standard extreme value theory; cf. Theorem 1.5.3 in \cite{leadbetter1983}. 
\end{proof}

\begin{proof}[Proof of Lemma \ref{lem: quantile}]
	Let $q_{n}(\alpha)$ denote the $\alpha$-quantile of $Z_n$. By assumption, we may choose a sequence $\delta_n \to 0$ such that 
	\[
	\begin{split}
	&\sup_{t \in \R}|\P(Y_n \le t) - \P(Z_n \le t)| \le \delta_n \quad \text{and} \\
	&\P \left ( \sup_{t \in \R} |\P(W_n \le t \mid \mathcal{C}_n) - \P(Z_n \le t)| > \delta_n \right )\le \delta_n. 
	\end{split}
	\]
	The latter follows from the fact that the Ky Fan metric metrizes convergence in probability. Define the event $E_n = \{ \sup_{t \in \R} |\P(W_n \le t \mid \mathcal{C}_n) - \P(Z_n \le t)| \le  \delta_n \}$.
	On this event, 
	\[
	\P(W_n \le q_{n}(\alpha+\delta_n) \mid \mathcal{C}_n) \ge \underbrace{\P(Z_n \le q_{n}(\alpha+\delta_n))}_{=\alpha+\delta_n} - \delta_n =  \alpha,
	\]
	so that $\hat{q}_{n}(\alpha) \le q_n (\alpha+\delta_n)$. 
	Thus, 
	\[
	\P(Y_n \le \hat{q}_{n}(\alpha)) \le \P(Y_n \le q_n(\alpha + \delta_n)) + \delta_n \le \P(Z_n \le q_n(\alpha+\delta_n)) + 2\delta_n = \alpha + 3\delta_n. 
	\]
	Likewise, on the event $E_n$,
	\[
	\P(Z_n \le t)|_{t = \hat{q}_{n}(\alpha)} \ge \underbrace{\P(W_n \le t \mid \mathcal{C}_{n})|_{t = \hat{q}_{n}(\alpha)}}_{\ge \alpha} - \delta_n  \ge \alpha - \delta_n, 
	\]
	so that $\hat{q}_{n}(\alpha) \ge q_n (\alpha-\delta_n)$. Arguing as in the previous case, we see that $\P(Y_n \le \hat{q}_{n}(\alpha)) \ge \alpha-3\delta_n$. This completes the proof. 
	
\end{proof}

\section{Proofs for Section \ref{sec: extension}}
 Recall that $\bS^{d-1}$ is the unit sphere in $\R^{d}$, i.e., $\bS^{d-1} := \{ \vx \in \R^{d} : \| \vx \|=1 \}$. Also recall that in Section 5, we allow $d=d_n \to \infty$.

\subsection{Proof of Theorem \ref{HDCLT'}}
Overall, the proof is analogous to that of Theorem \ref{thm: HDCLT}. 
The following Banadur representation is taken from \cite{belloni2019conditional}.

\begin{lemma}\label{prop: bahadur}
	Under Assumption \ref{Maup2}, we have 
	\[
	\hat{\beta}(\tau)-\beta(\tau)=J(\tau)^{-1}\left[\frac{1}{n}\sum_{i=1}^n\left\{\tau -I(U_i\le \tau)) \right\}\vX_i\right]+\check{R}_n(\tau)
	\]
	with $\| \check{R}_{n} \|_{[\epsilon/2,1-\epsilon/2]}=O_P(n^{-3/4}d\sqrt{\log n})$ and $\|n^{-1}\sum_{i=1}^n\left\{\tau -I(U_i\le \tau) \right\}\vX_i\|_{[\epsilon/2,1-\epsilon/2]}=O_P(\sqrt{d/n})$
\end{lemma} 

\begin{proof}
	See Theorems 1 and 2 in \cite{belloni2019conditional}. 
\end{proof}

The rates of convergence of $\hat{Q}_{\vx}^{(r)}(\tau_{\vx})$ change as follows. 
\begin{lemma}
	\label{lem: uniform rate2}
	Under the conditions of Theorem \ref{HDCLT'}, we have
	\[
	\sup_{\substack{\vx \in \mathcal{X}_0 \\ \tau \in [\epsilon,1-\epsilon]}}|\hat{Q}_{\vx}^{(r)}(\tau)-Q_{\vx}^{(r)}(\tau)|=
	\begin{cases}
	O_P\left(n^{-1/2}d+h^{2}\right) & \text{if $r=0$}\\
	O_P\left(n^{-1/2}h^{-r+1/2}d\sqrt{\log n}+h^{2}\right) & \text{if $r=1$ or $2$}\\
	O_P\left(n^{-1/2}h^{-5/2}d\sqrt{\log n}+h \right) & \text{if $r=3$}
	\end{cases}
	\]
\end{lemma}

%This can be shown by keeping track of dimension $d$ when we apply the local maximal inequality (Lemma \ref{lem: maximal inequality}) as in the proof of Lemma \ref{UnifQr}.

\begin{proof}
	We divide the proof into two steps.
	
	\textbf{Step 1}. We will show that
	\[
	\sup_{\substack{\vx\in \cX_0 \\ \tau \in [\epsilon,1-\epsilon]}}\left| \frac{1}{nh^{r}}\sum_{i=1}^n\vx^TJ(\tau)^{-1}\vX_i \left \{ K^{(r-1)}\left(\frac{\tau-U_i}{h}\right) -  h I(r=1) \right \} \right|=O_P\left(n^{-1/2}h^{-r+1/2}d\sqrt{\log n } \right),
	\]
	for $r=1,2,3$.
	The proof is analogous to that of Lemma \ref{UnifQr}, so we only point out required modifications. The envelope function $F$ should be modified to $F(u,\vx') = C\sqrt{d} \| \vx' \|$ for some constant $C$, and note that the VC constant $V$ is of order $V = O(d)$. 
	Observe that 
	\[
	\sup_{\substack{\vx \in \cX_{0}\\ \tau \in [\epsilon,1-\epsilon]}} \E[\{ K^{(r-1)}((\tau-U)/h)  \vx^{T}J(\tau)^{-1}\vX\}^{2}] \le O(d) \int_{0}^{1} K^{(r-1)}((\tau-u)/h)^2 du  = O(hd),
	\]
	and $\E[\max_{1 \le i \le n}F^2(U_i,\vX_i)] = O(d^2)$ (as $\| \vX \| \le C_3 \sqrt{d}$). 
	Applying Lemma \ref{lem: maximal inequality} leads to the above rates. 
	
	\textbf{Step 2}. We will show the conclusion of the lemma. This part is analogous to the proof of Lemma \ref{UnifQ}, so we only point out required modifications. The $r=0$ follows from Lemma \ref{prop: bahadur} and Taylor expansion. 
	%Consider first the case where $r=0$.  By definition,
	%\begin{align*}
	%\sup_{\substack{\vx\in \cX_0 \\ \tau \in [\epsilon,1-\epsilon]}}|\hat{Q}_{\vx}(\tau)-Q_{\vx}^{}(\tau)|&= \sup_{\substack{\vx\in \cX_0 \\ \tau \in [\epsilon,1-\epsilon]}}\left|\int\check{Q}_{\vx}(t)K_h^{}(\tau-t)dt - Q_{\vx}^{}(\tau)\right|\\
	%&\le\sup_{\substack{\vx\in \cX_0 \\ \tau \in [\epsilon,1-\epsilon]}}\left|\int[\check{Q}_{\vx}(t)-Q_{\vx}(t)]K_h(\tau-t)dt \right|\\
	%&\quad+\sup_{\substack{\vx\in \cX_0 \\ \tau \in [\epsilon,1-\epsilon]}}\left|\int Q_{\vx}(t)K_h^{}(\tau-t)dt - Q_{\vx}^{}(\tau) \right| \\
	%&=: I + II. 
	%\end{align*}
	%We have $I = O_{P}(n^{-1/2}d)$ by Lemma \ref{prop: bahadur} and $II=O(h^2)$ by Taylor expansion. 
	For $1\le r\le 3$, combining Lemma \ref{prop: bahadur}, change of variables, and Taylor expansion, we can bound $\sup_{\vx\in \cX_0; \tau \in [\epsilon,1-\epsilon]}|\hat{Q}_{\vx}^{(r)}(\tau)-Q_{\vx}^{(r)}(\tau)|$ by
	\begin{align*}
	%&\sup_{\substack{\vx\in \cX_0 \\ \tau \in [\epsilon,1-\epsilon]}}|\hat{Q}_{\vx}^{(r)}(\tau)-Q_{\vx}^{(r)}(\tau)| \\
	&\sup_{\substack{\vx\in \cX_0 \\ \tau \in [\epsilon,1-\epsilon]}}\left| \frac{1}{nh^{r}}\sum_{i=1}^n\int \vx^TJ(\tau - th)^{-1}\vX_i\left\{\tau - th-I\left(U_i\le \tau - th\right)\right\}K^{(r)} (t) dt\right| \\
	&\quad + \underbrace{O_{P}(n^{-3/4}d^{3/2}h^{-r} \sqrt{\log n})}_{=o_{P}(n^{-1/2}dh^{-r+1/2}\sqrt{\log n})} + O(h^2 I(r=1,2) + h I(r=3)). 
	%&\le \sup_{\substack{\vx\in \cX_0 \\ \tau \in [\epsilon,1-\epsilon]}}\left| \frac{1}{nh^{r}}\sum_{i=1}^n\int \vx^TJ(\tau - th)^{-1}\vX_i\left\{\tau - th-I\left(U_i\le \tau - th\right)\right\}K^{(r)} (t) dt\right|
	%&\le \sup_{\substack{\vx\in \cX_0 \\ \tau \in [\epsilon,1-\epsilon]}}\left|\int[\check{Q}_{\vx}(t)-Q_{\vx}(t)]K_h^{(r)}(\tau-t)dt \right|\\
	%&\quad+\sup_{\substack{\vx\in \cX_0 \\ \tau \in [\epsilon,1-\epsilon]}}\left|\int Q_{\vx}(t)K_h^{(r)}(\tau-t)dt - Q_{\vx}^{(r)}(\tau) \right| \\
	%&=: III + IV. 
	\end{align*}
	%We have $IV=O(h^2)$ for $r=1,2$ and $= O(h)$ for $r=3$ by Taylor expansion (recall that $Q_{\vx}(\tau)$ is four-times continuously differentiable). 
	%Observe that, by Lemma \ref{prop: bahadur} and change of variables, 
	%\[
	%\begin{split}
	%III
	%=\sup_{\substack{\vx\in \cX_0 \\ \tau \in [\epsilon,1-\epsilon]}}\left|\int \vx^T J^{-1}(t)\left[ \frac{1}{n}\sum_{i=1}^n\left\{t-I\left(U_i\le t\right)\right\}\vX_i+R_n(t)\right]h^{-(r+1)}K^{(r)}\left(\frac{\tau-t}{h}\right)dt\right|\\
	%&\le \sup_{\substack{\vx\in \cX_0 \\ \tau \in [\epsilon,1-\epsilon]}}\left| \frac{1}{nh^{r}}\sum_{i=1}^n\int \vx^TJ(\tau - th)^{-1}\vX_i\left\{\tau - th-I\left(U_i\le \tau - th\right)\right\}K^{(r)} (t) dt\right| \\
	%&\quad + \underbrace{O_{P}(n^{-3/4}d^{3/2}h^{-r} \sqrt{\log n})}_{=o_{P}(n^{-1/2}dh^{-r+1/2}\sqrt{\log n})}. 
	%\end{split}
	%\]
	Replacing $J(\tau - th)$ by $J(\tau)$ in the first term on the right hand side results in an error of order $O_{P}(n^{-1/2}dh^{-r+1})$.
	Given Step 1, the rest of the proof is completely analogous to the last part of the proof of Lemma \ref{UnifQ}.
	%this can be verified by a similar argument to the proof of the preceding lemma. Thus, it remains to bound
	%\begin{align*}
	%&\sup_{\substack{\vx\in \cX_0 \\ \tau \in [\epsilon,1-\epsilon]}}\left| \frac{1}{nh^{r}}\sum_{i=1}^n\int \vx^TJ(\tau)^{-1}\vX_i\left\{\tau-th-I\left(U_i\le \tau-th\right)\right\}K^{(r)}\left(t\right)dt\right| \\
	%&=\sup_{\substack{\vx\in \cX_0 \\ \tau \in [\epsilon,1-\epsilon]}}\left| \frac{1}{nh^{r}}\sum_{i=1}^n\vx^TJ(\tau)^{-1}\vX_i  \left\{ K^{(r-1)}\left(\frac{\tau-U_i}{h}\right) + h\int tK^{(r)}\left(t\right)dt\right\}\right|,
	%\end{align*}
	%where we have used the fact that $K^{(r)}$ integrates to $0$. 
	%Here, by integration by parts, 
	%\[
	%\int t K^{(r)}(t) dt = - \int K^{(r-1)} (t) dt = - I(r=1). 
	%\]
	%Thus, from Lemma \ref{UnifQr}, we have $III = O(n^{-1/2}dh^{-r+1/2} \sqrt{\log n})$. This completes the proof. 
\end{proof}
\begin{remark}[Expansion of $\hat{Q}_{\vx}''(\tau)$]
	Inspection of the proof shows that 
	\[
	\begin{split}
	&\hat{Q}_{\vx}''(\tau) - Q_{\vx}''(\tau) = \frac{1}{nh^{3}}\sum_{i=1}^n\int \vx^TJ(t )^{-1}\vX_i\left\{t-I\left(U_i\le t \right)\right\}K^{''}\left ( \frac{\tau-t}{h} \right ) dt   \\
	&\quad + O_{P}(n^{-3/4}h^{-2}d^{3/2} \sqrt{\log n})+O(h^2)
	\end{split}
	\]
	uniformly in $(\tau,\vx) \in [\epsilon,1-\epsilon] \times \cX_{0}$, and the uniform rate over  $(\tau,\vx) \in [\epsilon,1-\epsilon] \times \cX_{0}$ of the first term on the right hand side is $O_{P}(n^{-1/2}h^{-3/2} d\sqrt{\log n})$. 
\end{remark}

Recall the definition of $\xi_{\vx}$.
In view of the proof of Lemma \ref{lem: consistency}, the following lemma follows relatively directly from Lemma \ref{lem: uniform rate2}. 

\begin{lemma}
	\label{prop: UAL'}
	Under the conditions of Theorem \ref{HDCLT'}, the following asymptotic linear representation holds uniformly in $\vx\in \cX_0$:	
	\[
	\hat{m}(\vx)-m(\vx) = \frac{\sqrt{d}}{nh^{3/2}}\sum_{i=1}^n\xi_{\vx}(U_i,\vX_i)+ O_P(n^{-3/4}h^{-2}d^{3/2}\sqrt{\log n} + n^{-1}h^{-4}d^2\log n + h^{2})
	\]
	where $U_{1},\dots,U_{n} \sim U(0,1)$ i.i.d.\ independent of $\vX_{1},\dots,\vX_{n}$. 
	In addition, we have 
	\[\sup_{\vx\in\cX_0}\left|\frac{1}{nh^{3/2}}\sum_{i=1}^n\psi_{\vx}(U_i,\vX_i)\right|=O_P(n^{-1/2}h^{-3/2}d\sqrt{\log n}).\]
	%The uniform rate over $\vx \in \cX_{0}$ of the term term on the right hand side is $O_P(n^{-1/2}h^{-3/2}d\sqrt{\log n})$.
	%\begin{align*}
	%&\sup_{\vx\in\cX_0}\left|\frac{\sqrt{d}}{nh^{3/2}}\sum_{i=1}^n\xi_{\vx}(U_i,\vX_i)\right|=O_P(n^{-1/2}h^{-3/2}d\sqrt{\log n}).
	%\end{align*}
\end{lemma}

%This can be shown by keeping track of dimension $d$ when we apply the local maximal inequality (Lemma \ref{lem: maximal inequality}) as in the proof of Lemma \ref{UnifQr}.

We are now in position to prove Theorem \ref{HDCLT'}.

\begin{proof}[Proof of Theorem \ref{HDCLT'}]
	As before, we split the proof into two parts. 
	
	\textbf{Step 1}. We will apply Proposition 2.1 in \cite{chernozhukov2017central} to $n^{-1/2}\sum_{i=1}^{n}A\Psi_{i}$. To this end, we will check Conditions (M.1), (M.2), and (E.1) of \cite{chernozhukov2017central}. Condition (M.1) follows automatically, so we will verify Conditions (M.2) and (E.1). 
	
	%\textbf{Condition (M.1)}: This is again automatically satisfied under our set up.
	
	\textbf{Condition (M.2)}. 
	Recall that $\|\vx\|/\sqrt{d} \le C_{2}$ for all $\vx \in \cX_{0}$. Observe that  %$\max_{1 \le k \le M} \E\left[ \left \| (\xi_{\vx_{\ell}}(U_{i},\vX_{i}))_{\ell \in S_{k}} \right \|_{1}^{3} \right]$.
	\begin{align*}
	&\max_{1 \le k \le M} \E\left[ \left \| (\xi_{\vx_{\ell}}(U_{i},\vX_{i}))_{\ell \in S_{k}} \right \|_{1}^{3} \right]\\
	&\le O(h^{-3/2}) \sup_{\alpha \in \bS^{d-1}} \E[|\alpha^{T}\vX|^{3}]  \max_{1\le \ell \le L} \int \left| K''\left(\frac{\tau_{\vx_\ell}-t}{h}\right) \right|^{3} dt=O(h^{-1/2}d^{1/2}),
	\end{align*}
	where we used the fact that 
	\[
	\sup_{\alpha \in \bS^{d-1}}\E [|\alpha^T\vX |^3]\le C_{3} \sqrt{d} \sup_{\alpha \in \bS^{d-1}}\E [(\alpha^T \vX)^{2}] = C_3 \sqrt{d} \| \E[\vX\vX^{T}] \|_{\op}=O(\sqrt{d}). 
	\]
	This implies that $\max_{1 \le k \le M}\E[|A_{k}^{T}\Psi_{i}|^{3}] = O(d^{1/2}h^{-1/2})$. 
	Likewise,  $\max_{1 \le \ell \le M} \E[|A^T_{k} \Psi_i|^4]= O(dh^{-1})$. 
	
	\textbf{Condition (E.2)}. Since $\| \vX \| \le C_{3}\sqrt{d}$, we have $|A^T_{k}\Psi_i| \le \text{const.}~h^{-1/2}d^{1/2}$.
	
	Thus, applying Proposition 2.1 in  \cite{chernozhukov2017central}, we have
	\[
	\sup_{b \in \R^{M}}\left| \P \left ( n^{-1/2}{\textstyle \sum}_{i=1}^{n}A\Psi_{i} \le b \right) - \P(AG \le b) \right | \to 0,
	\]
	provided that
	\[
	\frac{d\log^7\left(Mn\right)}{nh} \to 0,
	%\bigvee \frac{d\log^3\left(Mn\right)}{n^{1-q/2}h}\\
	%= o(1),
	\]
	which is satisfied under our assumption.
	
	\textbf{Step 2}. Observe that 
	\begin{equation}
	\begin{split}
	&\left \| A \sqrt{nh^{3}d^{-1}} (\hat{m}(\vx_{\ell}) - m(\vx_{\ell}))_{\ell=1}^{L} - n^{-1/2}{\textstyle \sum}_{i=1}^{n}A\Psi_{i} \right \|_{\infty} \\
	&=  O_{P}(n^{-1/4}h^{-1/2}d\sqrt{\log n} + n^{-1/2}h^{-5/2}d^{3/2}\log n + n^{1/2}h^{7/2}d^{-1/2}). 
	\end{split}
	\label{eq: remainder}
	\end{equation}
	In view of the proof of Step 2 in Theorem \ref{thm: HDCLT}, the desired conclusion follows if the right hand side on (\ref{eq: remainder}) is $o_{P}(1/\sqrt{\log M})$, which is satisfied under our assumption. 
\end{proof}

\subsection{Proof of Theorem \ref{boot'}}
Define $\check{\Psi}_i:=(\check{\psi}_{\vx_1}(U_i,\vX_i),\dots,\check{\psi}_{\vx_{L}}(U_i,\vX_i))^T$ with 
\[
\check{\psi}_{\vx}(u,\vx'):=\frac{s_{\vx}(\tau_{\vx})}{s''_{\vx}(\tau_{\vx})\sqrt{dh}}K'\left(\frac{\tau_{\vx}-u}{h}\right)\vx^TJ(\tau_{\vx})^{-1}\vx'.
\]
Further, define $\check{\Sigma}=\E\left[\check{\Psi}_i\check{\Psi}_i^T\right]$,  $\check{\Gamma}:=\diag\{\check{\sigma}_1,\dots,\check{\sigma}_M\}$ with $\check{\sigma}^2_{i}:= D_i^T \check{\Sigma} D_i$, and $\check{A}:=\check{\Gamma}^{-1}D$. 

The following operator norm bound is in parallel to Lemma \ref{lem: Jacobian} for the fixed dimensional case.
\begin{lemma}\label{lem: Jacobian'}
	Under the conditions of Theorem \ref{boot'}, we have
	\[
	\sup_{\vx\in\cX_0} \| \hat{J}(\hat{\tau}_{\vx})-J(\tau_{\vx}) \|_{\op}=O_P(n^{-1/2}h^{-3/2}d^{3/2}\sqrt{\log n}+h^2).
	\]
\end{lemma}
\begin{proof}
	Observe that the left hand side can be bounded by
	\[
	\begin{split}
	&\left\| \frac{1}{n}\sum_{i=1}^{n}K_h(Y_i-\vX_i^T\hat{\beta}(\hat{\tau}_{\vx}))(\alpha^T\vX_{i})^2- \E\left[K_h(Y_i-\vX_i^T\beta) (\alpha^T\vX)^2\right]\big|_{\beta=\hat{\beta}(\hat{\tau}_{\vx})} \right\|_{\bS^{d-1} \times \cX_0}\\
	&\quad +  \left\| \E\left[K_h(Y_i-\vX_i^T\beta) (\alpha^T\vX)^2\right]\big|_{\beta=\hat{\beta}(\hat{\tau}_{\vx})}-\E\left[f(\vX^T\beta \mid  \vX)(\alpha^T\vX)^2\right]\big|_{\beta=\hat{\beta}(\hat{\tau}_{\vx})} \right\|_{\bS^{d-1} \times \cX_0}\\
	&\quad +\left\| \E\left[f(\vX^T\beta \mid  \vX)(\alpha^T\vX)^2\right]\big|_{\beta=\hat{\beta}(\hat{\tau}_{\vx})}-\E\left[f(\vX^T\beta(\tau_{\vx}) \mid  \vX)(\alpha^T\vX)^2\right]\right\|_{\bS^{d-1} \times \cX_0}\\
	&=: I+II+III,
	\end{split}
	\]
	where $\| \cdot \|_{\bS^{d-1} \times \cX_{0}} = \sup_{(\alpha,\beta) \in \bS^{d-1} \times \cX_{0}} | \cdot |$. 
	By Taylor expansion and $\sup_{\alpha \in \bS^{d-1}}\E[(\alpha^{T}\vX)^{2}] = \| \E[\vX\vX^{T}] \|_{\op} = O(1)$, we see that $II = O(h^{2})$. 
	Next, applying the local maximal inequality (Lemma \ref{lem: maximal inequality}) combined with the fact that $\sup_{\alpha \in \bS^{d-1}} \E[|\alpha^{T}\vX|^4] = O(d)$, we can show that  $\rom{1}=O_P(\sqrt{n^{-1}h^{-1}d^2\log n})$. Finally, the term $III$ is bounded by 
	\[
	\begin{split}
	&C_{1}\| \hat{\beta}(\hat{\tau}_{\vx}) - \beta (\tau_{\vx}) \|_{\cX_{0}} \underbrace{\sup_{\alpha \in \bS^{d-1}}
		\E[|\alpha^{T}\vX|^{3}]}_{=O(d^{1/2})} \quad \text{and} \\
	&\|\hat{\beta}(\hat{\tau}_{\vx})-\beta(\tau_{\vx})\|_{\cX_0}=O_P\left(\|\hat{\beta}-\beta\|_{[\epsilon,1-\epsilon]} \bigvee \left\|\beta(\hat{\tau}_{\vx})-\beta(\tau_{\vx})\right\|_{\mathcal{X}_0}  \right)\\
	&\quad =O_P\left(n^{-1/2}d^{1/2} \bigvee n^{-1/2}h^{-3/2}d\sqrt{\log n}\right)= O_P(n^{-1/2} h^{-3/2}d\sqrt{\log n}),
	\end{split}
	\]
	where we used the observation that $Q'_{\vx}(\tau)=\vx^T\beta'(\tau)$ is bounded in $(\tau,\vx) \in [\epsilon,1-\epsilon] \times \cX$. Conclude that $III = O_{P}(n^{-1/2}h^{-3/2}d^{3/2}\sqrt{\log n})$.
\end{proof}
Similarly, we have the following lemma in parallel to Lemma \ref{lem: covariance}.
\begin{lemma}
	Under Assumption \ref{Maup2}, we have 
	\[
	\| \hat{\Sigma} - \check{\Sigma} \|_{\infty} = O_{P}(n^{-1/2}h^{-3/2}d(d^{1/2}\vee h^{-1})\sqrt{\log n} + h).
	\]
\end{lemma}
\begin{proof}
	The proof is analogous to the proof of Lemma \ref{lem: covariance}, given that $\vx_{\ell}/\sqrt{d} \le C_2$ and we added normalization by $\sqrt{d}$ in the definition of $\psi_{\vx}$. The only missing part is a bound on
	\[
	\left \| \frac{1}{n} \sum_{i=1}^{n} \vX_i \vX_i^{T} - \E[\vX\vX^T] \right \|_{\op},
	\]
	but Rudelson's inequality yields that the above term is $O_{P}(\sqrt{d(\log d)/n})$; cf. \cite{rudelson1999}. 
\end{proof}

We are now in position to prove Theorem \ref{boot'}.
\begin{proof}[Proof of Theorem \ref{boot'}]
	
	%\textcolor{red}{Edit the following text.}
	%We emphasize that the bootstrap quantities we use is intrinsically the same as the fixed dimension case except for a $\sqrt{d}$ normalization (see the discussion after Theorem \ref{boot'}), though we use a new $\Psi$ in our increasing dimensional Gaussian approximation result. To prove Theorem \ref{boot'}, we first try to ``bridge the gap" between the newly-defined $\Psi$ and ``old" bootstrap quantity $\hat{\Psi}$ as following,
	Observe that
	\begin{equation}
	\begin{split}
	&\sup_{b\in \R^M}\left|\P_{|\cD_n}\left(n^{-1/2}{\textstyle\sum}_{i=1}^n\hat{A}\hat{\Psi}_i\le b\right)-\P(AG\le b)\right|\\
	&\le \sup_{b\in \R^M}\left|\P_{|\cD_n}\left(n^{-1/2}{\textstyle\sum}_{i=1}^n\hat{A}\hat{\Psi}_i\le b\right)-\P(\check{A}\check{G}\le b)\right| \\
	&\quad +\sup_{b\in \R^M}\left|\P(\check{A}\check{G}\le b)-\P(AG\le b) \right|,
	\end{split}
	\label{eq: decomp2}
	\end{equation}
	where $\check{G}\sim N(0,\check{\Sigma})$.
	The first term on the right hand side of (\ref{eq: decomp2}) is bounded by
	\[
	\begin{split}
	&\underbrace{\sup_{b\in \mathbb{R}^M}\left|\mathbb{P}_{U}\left(n^{-1/2}{\textstyle\sum}_{i=1}^n\hat{A}\hat{\Psi}_i\le b \right)-\mathbb{P}_{U}(\hat{A}\hat{G}\le b )\right|}_{\rom{1}}\\
	& + \underbrace{\sup_{b\in \mathbb{R}^M}\left|\mathbb{P}_{U}(\hat{A}\hat{G}\le b)-\P_{|\cD_n}(\check{A}\hat{G}\le b) \right|}_{II} + \underbrace{ \sup_{b\in \mathbb{R}^M}\left|\mathbb{P}_{U}(\check{A}\hat{G}\le b)-\mathbb{P}(\check{A}\check{G}\le b)\right|}_{III},
	\end{split}
	\]
	where $\hat{G} \sim N(0,\hat{\Sigma})$ conditionally on $\cD_{n}$. 
	%we define $\hat{G:=n^{-1/2}\sum_{i=1}^ng_i$ with i.i.d.\ $g_i \mathop{\sim} N\left(0, \mathbb{E}_{U}\left[\hat{\Psi}_i\hat{\Psi}_{i}^{T} \right]\right) $. 
	For $I$, we can apply Proposition 2.1 in \cite{chernozhukov2017central} conditionally on $\cD_n$.
	Similarly to the last part of the proof of Theorem \ref{thm: boot},  we can show that $I = o_{P}(1)$ if 
	\[
	\frac{d\log^7\left(Mn\right)}{nh} \to 0,
	\]
	which is satisfied under our assumption.
	We can analyze $II$ and $III$ as in the proof of Theorem \ref{thm: boot} and show that $II \vee III = o_{P}(1)$ if $n^{-1/2}h^{-3/2}d(d^{1/2}\vee h^{-1})(\sqrt{\log n})\log^2 M=o(1)$ and $ h\log^2 M=o(1)$, which is satisfied under our assumption. 
	%For $\rom{2}$ and $\rom{3}$, we proceed exactly as the fixed dimensional case but replace the convergence rate as the rate under increasing dimension. Finally, we need $n^{-1/2}h^{-3/2}d(d^{1/2}\vee h^{-1})(\sqrt{\log n})\log^2 M=o(1)$ and $ h\log^2 M=o(1)$ to ensure both $\rom{2}$ and $\rom{3}$ converge to $0$ in probability, which are satisfied under our conditions. \\
	
	Finally, in view of the Gaussian comparison inequality (Lemma \ref{lem: Gaussian comparison}), we see that the second term on the right hand side of (\ref{eq: decomp2}) is $o(1)$ if $\| \check{A} \check{\Sigma} \check{A}^{T} - A\Sigma A^{T}\|_{\infty}\log^2 M = o(1)$. It is not difficult to see that 
	\[
	\| \check{A} \check{\Sigma} \check{A}^{T} - A\Sigma A^{T}\|_{\infty}
	=O\left(\sup_{\vx_1,\vx_2 \in \mathcal{X}_0} |\mathbb{E}[\xi_{\vx_1}\xi_{\vx_2}-\check{\psi}_{\vx_1}\check{\psi}_{\vx_2}]|\right)
	=O(h) = o(1/\log^2 M).
	\]
	This completes the proof.
\end{proof}

\section{Additional simulation results}
\subsection{Nonparametric bootstrap pointwise confidence intervals}\label{sec : np bootstrap simulation}
In this section, we present simulation results for the nonparametric bootstrap. Due to the heavy computational burden of the nonparametric bootstrap, we only consider pointwise confidence intervals in the simulation.  We consider the $lmNormal$, $lmLognormal$ and $Nonlinear$ models as in Section \ref{sec: pCI simu}, together with the same subsample sizes, $n=500$, $1000$ and $2000$, and repetition number $s= 500$. The results are presented in Tables \ref{tab:point normal nonparametric}--\ref{tab:point nonlinear nonparametric}.
\begin{table}[h!]
	\centering
	\caption{Nonparametric bootstrap pointwise confidence intervals for \emph{lmNormal} model.}
	\label{tab:point normal nonparametric}
	{\small
		%\resizebox{\textwidth}{!}{
		\begin{tabularx}{\textwidth}{ Z Z Z Z Z Z Z Z}
			\toprule
			\multicolumn{1}{c}{\multirow{2}{*}{Design point}} & \multicolumn{1}{c}{\multirow{2}{*}{Sample size}} & \multicolumn{2}{c}{Coverage probability} & \multicolumn{2}{c}{Median length} & \multicolumn{2}{c}{Interquartile range} \\
			\cmidrule(lrr){3-4} \cmidrule(lrr){5-6} \cmidrule(lrr){7-8} 
			
			\multicolumn{1}{c}{}                              & \multicolumn{1}{c}{}                             & 95\%                & 99\%               & 95\%            & 99\%   & 95\%            & 99\%                      \\
			\midrule
			\multirow{3}{*}{$X_1$=0.3}                            & $n=500$                                            & 92.4\%              & \multicolumn{1}{c}{97.6\%}              & 0.80            & 1.14       & 0.38            & 0.52                  \\
			& $n=~1000$                                           & 93.2\%                & \multicolumn{1}{c}{98.8\%}               & 0.69            & 0.95     & 0.29          & 0.39              \\
			& $n=~2000$                                           & 92\%                & \multicolumn{1}{c}{99.2\%}             & 0.57            & 0.79    & 0.25            & 0.33            \\
			\midrule
			\multirow{3}{*}{$X_1$=0.5}                            & $n=500$                                            & 92.4\%              & \multicolumn{1}{c}{98.4\%}              & 0.91            & 1.25   & 0.40            & 0.57                        \\
			& $n=~1000$                                           & 93.4\%                & \multicolumn{1}{c}{98.2\%}               & 0.76            & 1.05     & 0.30            & 0.44                     \\
			& $n=~2000$                                           & 93.4\%                & \multicolumn{1}{c}{98\%}             & 0.64            & 0.88       & 0.22            & 0.32       
			\\
			\midrule
			\multirow{3}{*}{$X_1$=0.7}                            & $n=500$                                            & 92.4\%              & \multicolumn{1}{c}{98.6\%}              & 1.10            & 1.57   & 0.54            & 0.80                        \\
			& $n=~1000$                                           & 91.4\%                & \multicolumn{1}{c}{97.6\%}               & 0.91            & 1.26  & 0.40            & 0.58                        \\
			& $n=~2000$                                           & 93.8\%                & \multicolumn{1}{c}{97.6\%}             & 0.79            & 1.07      & 0.27           & 0.40         
			\\
			\bottomrule   
		\end{tabularx}	
	}
\end{table}

\begin{table}[h!]
	\centering
	\caption{Nonparametric bootstrap pointwise confidence intervals for \emph{lmLognormal} model.}
	\label{tab:point lognormal nonparametric}
	{\small
		%\resizebox{\textwidth}{!}{
		\begin{tabularx}{\textwidth}{ Z Z Z Z Z Z Z Z}
			\toprule
			\multicolumn{1}{c}{\multirow{2}{*}{Design point}} & \multicolumn{1}{c}{\multirow{2}{*}{Sample size}} & \multicolumn{2}{c}{Coverage probability} & \multicolumn{2}{c}{Median length}  & \multicolumn{2}{c}{Interquartile range}\\
			\cmidrule(lrr){3-4} \cmidrule(lrr){5-6}  \cmidrule(lrr){7-8}
			
			\multicolumn{1}{c}{}                              & \multicolumn{1}{c}{}                             & 95\%                & 99\%               & 95\%            & 99\%         & 95\%            & 99\%              \\
			\midrule
			\multirow{3}{*}{$X_1$=0.3}                            & $n=500$                                            & 87.8\%              & \multicolumn{1}{c}{95.8\%}              & 2.85            & 4.31     & 2.70            & 4.06                   \\
			& $n=~1000$                                           & 86.2\%                & \multicolumn{1}{c}{94.4\%}               & 1.86            & 3.27        & 2.05            & 3.61             \\
			& $n=~2000$                                           & 89.2\%                & \multicolumn{1}{c}{95.4\%}             & 1.47           & 2.69       & 2.07            & 3.36      \\
			\midrule
			\multirow{3}{*}{$X_1$=0.5}                            & $n=500$                                            & 85\%              & \multicolumn{1}{c}{95.2\%}              & 3.28            & 5.05     & 2.86            & 4.14                       \\
			& $n=~1000$                                           & 88\%                & \multicolumn{1}{c}{95.8\%}               & 2.03           & 3.57     & 2.39           & 3.99                      \\
			& $n=~2000$                                           & 84.4\%                & \multicolumn{1}{c}{92.6\%}             & 1.55           & 2.76       & 1.99           & 3.49        
			\\
			\midrule
			\multirow{3}{*}{$X_1$=0.7}                            & $n=500$                                            & 82\%              & \multicolumn{1}{c}{92.8\%}              & 4.51            & 6.64      & 4.27            & 6.69                      \\
			& $n=~1000$                                           & 86.2\%                & \multicolumn{1}{c}{96\%}               & 2.76            & 5.39        & 3.10            & 4.43                   \\
			& $n=~2000$                                           & 83.6\%                & \multicolumn{1}{c}{94.4\%}             & 1.75           & 3.13        & 2.07            & 3.95        
			\\
			\bottomrule   
		\end{tabularx}	
	}
\end{table}

\begin{table}[h!]
	\centering
	\caption{Nonparametric bootstrap pointwise confidence intervals for \emph{Nonlinear} model.}
	\label{tab:point nonlinear nonparametric}
	{\small
		%\resizebox{\textwidth}{!}{
		\begin{tabularx}{\textwidth}{ Z Z Z Z Z Z Z Z}
			\toprule
			\multicolumn{1}{c}{\multirow{2}{*}{Design point}} & \multicolumn{1}{c}{\multirow{2}{*}{Sample size}} & \multicolumn{2}{c}{Coverage probability} & \multicolumn{2}{c}{Median length} & \multicolumn{2}{c}{Interquartile range} \\
			\cmidrule(lrr){3-4} \cmidrule(lrr){5-6}  \cmidrule(lrr){7-8} 
			
			\multicolumn{1}{c}{}                              & \multicolumn{1}{c}{}                             & 95\%                & 99\%               & 95\%            & 99\%        & 95\%            & 99\%               \\
			\midrule
			\multirow{3}{*}{$X_1$=0.7}                            & $n=500$                                            & 94.8\%              & \multicolumn{1}{c}{99\%}              & 1.05            & 1.62      & 1.55           & 2.62                      \\
            & $n=~1000$                                           & 92.2\%                & \multicolumn{1}{c}{98.4\%}               & 0.78            & 1.27   & 1.01          & 1.91                        \\
            & $n=~2000$                                           & 94.8\%                & \multicolumn{1}{c}{98.8\%}             & 0.67           & 1.08     & 0.92           & 1.32           
            \\
			\midrule
			\multirow{3}{*}{$X_1$=0.9}                            & $n=500$                                            & 94\%              & \multicolumn{1}{c}{99.4\%}              & 1.11          & 1.59     & 1.25            & 1.77                    \\
			& $n=~1000$                                           & 94.0\%                & \multicolumn{1}{c}{98\%}               & 0.84          & 1.42   & 1.18            & 1.57                     \\
			& $n=~2000$                                           & 94.8\%                & \multicolumn{1}{c}{98\%}             & 0.60           & 1.04       & 0.86           & 1.25    
			\\
			\midrule
			\multirow{3}{*}{$X_1$=1.1}                            & $n=500$                                            & 94\%              & \multicolumn{1}{c}{99\%}              & 0.80            & 1.49      & 1.05           & 1.17                      \\
			& $n=~1000$                                           & 95.2\%                & \multicolumn{1}{c}{99.2\%}               & 0.56            & 1.08   & 0.80          & 1.14                        \\
			& $n=~2000$                                           & 92.8\%                & \multicolumn{1}{c}{97.8\%}             & 0.36           & 0.60     & 0.39           & 0.74           
			\\
			\bottomrule   
		\end{tabularx}	
	}
\end{table}
From the simulation results, the nonparametric bootstrap confidence intervals achieve close to nominal coverage probabilities under large sample sizes for the \emph{lmNormal} and \emph{Nonlinear} models.  
For the \emph{lmLognormal} model, the nonparametric bootstrap confidence intervals have lower coverage probabilities than the nominal level.
%tend to be ``over-confident" in finite samples, i.e., they result in coverage probabilities lower than the nominal level. 
This may be due to the slow convergence rate of the bootstrap approximation to the sampling distribution under such a data generating process.  Compared with the pivotal bootstrap confidence intervals in the other two models, the nonparametric bootstrap provides shorter and more stable confidence intervals, i.e., less variable interval lengths, in the \emph{lmNormal} model while the pivotal bootstrap performs better in the \emph{Nonlinear} model.

To further demonstrate the computational advantage of the pivotal bootstrap over the nonparametric bootstrap, we report the average running time of these two bootstraps in the \emph{Nonlinear} model with design point $X_1=0.9$ (results in other scenarios are similar). The simulation results are obtained in the R environment with 28 Intel Xeon processors and 240 Gbytes RAM over Red Hat OpenStack Platform. We measure the average running time in seconds and report the results in Table \ref{tab: run time}. From the table, we can see that  the pivotal bootstrap requires substantially less computational time than the nonparametric bootstrap as predicted in Remark \ref{rem: bootstrap comparison} in the main text. 

\begin{table}[h!]
	\centering
	\caption{Running time comparison between the pivotal and  nonparametric bootstraps.}
	\label{tab: run time}
	{\small
		%\resizebox{\textwidth}{!}{
		\begin{tabularx}{\textwidth}{ Z Z Z Z }
			\toprule
			\multicolumn{1}{c}{\multirow{1}{*}{Design point}} & \multicolumn{1}{c}{\multirow{1}{*}{Sample size}} & \multicolumn{1}{c}{Pivotal bootstrap} & \multicolumn{1}{c}{Nonparametric bootstrap}  \\
			\midrule
			\multirow{3}{*}{$X_1$=0.9}                            & $n=500$                                            & 1.02              & \multicolumn{1}{c}{22.23}                               \\
			& $n=~1000$                                           & 1.43                & \multicolumn{1}{c}{33.79}                                    \\
			& $n=~2000$                                           & 2.27               & \multicolumn{1}{c}{58.08}            
			\\
			\bottomrule   
		\end{tabularx}	
	}
\end{table}

\subsection{Mean squared error comparison with existing modal estimators}\label{sec: MSE}
In the following, we present the mean squared error of our modal estimator $\hat{m}_{\vx}$, which is defined by
\[ MSE(\hat{m}_{\vx}):= s^{-1}\sum_{i=1}^s (\hat{m}_{\vx}^{(i)}-m_{\vx})^2,\]
where $m_{\vx}$ is the true conditional mode and $\hat{m}_{\vx}^{(i)}$ is our modal estimator in the $i$-th repetition. We compare our method with existing methods by \cite{ohta2018quantile}, \cite{kemp2012regression} and \cite{yao2014new}. We consider the $lmNormal$, $lmLognormal$ and $Nonlinear$ models as in Section \ref{sec: pCI simu}. The same subsample sizes, $n=500$, $1000$ and $2000$, and repetition number $s= 500$ are considered. The results are presented in Tables \ref{tab: normal MSE}--\ref{tab: nonlinear MSE}. The KS-YL estimator refers to the linear modal estimator studied by \cite{kemp2012regression} and \cite{yao2014new}.

\begin{table}[h!]
	\centering
	\caption{Mean squared error comparison: \emph{lmNormal}.}
	\label{tab: normal MSE}
	\resizebox{\textwidth}{!}{
		\begin{tabularx}{\textwidth}{ Z Z Z Z Z Z Z Z }
			\toprule
			\multicolumn{1}{c}{\multirow{1}{*}{Design point}} & \multicolumn{1}{c}{\multirow{1}{*}{Sample size}} & \multicolumn{1}{c}{Our estimator } & \multicolumn{1}{c}{\cite{ohta2018quantile}} & \multicolumn{1}{c}{KS-YL estimator} \\
			\midrule
			\multirow{3}{*}{$X_1$=0.3}                            & $n=500$                                            & 0.020             & \multicolumn{1}{c}{0.063}              &  0.207                   \\
			& $n=1000$                                           & 0.018               & \multicolumn{1}{c}{0.043}               & 0.208                    \\
			& $n=2000$                                           & 0.010                & \multicolumn{1}{c}{0.032}             & 0.275            \\
			\midrule
			\multirow{3}{*}{$X_1$=0.5}                            & $n=500$                                            & 0.026             & \multicolumn{1}{c}{0.091}              & 0.712                  \\
			& $n=1000$                                           & 0.017                & \multicolumn{1}{c}{0.063}               & 0.713                    \\
			& $n=2000$                                           & 0.012                & \multicolumn{1}{c}{0.047}             & 0.928          
			\\
			\midrule
			\multirow{3}{*}{$X_1$=0.7}                            & $n=500$                                            & 0.109             & \multicolumn{1}{c}{0.167}              & 1.559            \\
			& $n=1000$                                           & 0.026                & \multicolumn{1}{c}{0.128}               & 1.840                  \\
			& $n=2000$                                           & 0.024                & \multicolumn{1}{c}{0.103}             & 1.954           
			\\
			\bottomrule   
		\end{tabularx}	
	}
\end{table}

\begin{table}[h!]
	\centering
	\caption{Mean squared error comparison: \emph{lmLognormal}.}
	\label{tab: lognormal MSE}
	\resizebox{\textwidth}{!}{
		\begin{tabularx}{\textwidth}{ Z Z Z Z Z Z Z Z }
			\toprule
			\multicolumn{1}{c}{\multirow{1}{*}{Design point}} & \multicolumn{1}{c}{\multirow{1}{*}{Sample size}} & \multicolumn{1}{c}{Our estimator } & \multicolumn{1}{c}{\cite{ohta2018quantile}} & \multicolumn{1}{c}{KS-YL estimator} \\
			\midrule
			\multirow{3}{*}{$X_1$=0.3}                            & $n=500$                                            & 0.286             & \multicolumn{1}{c}{0.329}              &  0.146                   \\
			& $n=1000$                                           & 0.162               & \multicolumn{1}{c}{0.296}               & 0.124                    \\
			& $n=2000$                                           & 0.118               & \multicolumn{1}{c}{0.259}             & 0.066            \\
			\midrule
			\multirow{3}{*}{$X_1$=0.5}                            & $n=500$                                            & 0.421             & \multicolumn{1}{c}{0.508}              & 0.204                 \\
			& $n=1000$                                           & 0.219                & \multicolumn{1}{c}{0.466}               & 0.160                   \\
			& $n=2000$                                           & 0.151                & \multicolumn{1}{c}{0.410}             & 0.094         
			\\
			\midrule
			\multirow{3}{*}{$X_1$=0.7}                            & $n=500$                                            & 0.845             & \multicolumn{1}{c}{0.774}              & 0.407           \\
			& $n=1000$                                           & 0.491               & \multicolumn{1}{c}{0.657}               & 0.269                 \\
			& $n=2000$                                           & 0.217               & \multicolumn{1}{c}{0.605}             & 0.224          
			\\
			\bottomrule   
		\end{tabularx}	
	}
\end{table}

\begin{table}[h!]
	\centering
	\caption{Mean squared error comparison: \emph{Nonlinear}.}
	\label{tab: nonlinear MSE}
	\resizebox{\textwidth}{!}{
		\begin{tabularx}{\textwidth}{ Z Z Z Z Z Z Z Z }
			\toprule
			\multicolumn{1}{c}{\multirow{1}{*}{Design point}} & \multicolumn{1}{c}{\multirow{1}{*}{Sample size}} & \multicolumn{1}{c}{Our estimator } & \multicolumn{1}{c}{\cite{ohta2018quantile}} & \multicolumn{1}{c}{KS-YL estimator} \\
			\midrule
			\multirow{3}{*}{$X_1$=0.3}                            & $n=500$                                            & 0.0090             & \multicolumn{1}{c}{0.0023}              &  0.051                   \\
			& $n=1000$                                           & 0.0071              & \multicolumn{1}{c}{0.00082}               & 0.041                   \\
			& $n=2000$                                           & 0.0044                & \multicolumn{1}{c}{0.00045}             & 0.031            \\
			\midrule
			\multirow{3}{*}{$X_1$=0.5}                            & $n=500$                                            & 0.016             & \multicolumn{1}{c}{0.0052}              & 0.075                  \\
			& $n=1000$                                           & 0.010                & \multicolumn{1}{c}{0.0046}               & 0.063                    \\
			& $n=2000$                                           & 0.0086               & \multicolumn{1}{c}{0.0043}             & 0.049          
			\\
			\midrule
			\multirow{3}{*}{$X_1$=0.7}                            & $n=500$                                            & 0.026             & \multicolumn{1}{c}{0.025}              & 0.074            \\
			& $n=1000$                                           & 0.017                & \multicolumn{1}{c}{0.023}               & 0.068                  \\
			& $n=2000$                                           & 0.012                & \multicolumn{1}{c}{0.022}             & 0.052           
			\\
			\bottomrule   
		\end{tabularx}	
	}
\end{table}
From the simulation results, while no method dominates in all the three scenarios, the proposed modal estimator performs reasonably well uniformly in all the settings. In particular, our estimator and \cite{ohta2018quantile}'s estimator outperform the KS-YL linear modal estimator in the \emph{Nonlinear} model as expected.    

\subsection{Simulation results for the pivotal bootstrap testing}\label{sec: addtest}\label{sec: hp simu}
In this section, we consider testing significance of a covariate on the conditional mode. Suppose covariate $\vX=(1,X_1,X_2)$ where $X_1$ is continuous and $X_2$ is binary (0 or 1). We want to test the null hypothesis $H_0:$ $m(X_1,0)=m(X_1,1)$ versus the alternative hypothesis ${H}_{1}:$ $m(X_1,0)\ne m(X_1,1)$, where $m(x_1,x_2)$ is the conditional mode of $Y$ given $X_1=x_1$ and $X_2=x_2$. We will generate $\vX$ according to $X_1\sim \text{Unif}(0,1)$ and $X_2\sim \text{Binomial}(0.5)$. For the outcome $Y$, two generation schemes are considered: (1) $Y=1+3X_1+\xi$ and (2) $Y=1+3X_1+\alpha X_2+\xi$, where we take $\xi \sim N(0,1)$ and $\alpha \ne 0$ in both models. The corresponding mode functions are $m(\vX)=1+3X_1$ and $m(\vX)=1+3X_1+\alpha X_2$, respectively. Therefore, the two generation schemes correspond to $H_0$ being true and false respectively which allows us to evaluate both power and size of our bootstrap testing procedure. We will take $\alpha=0.8$ or $1$ in the simulation. In the current setup, the limiting Gaussian distribution given by Theorem \ref{thm: HDCLT} is one dimensional and the corresponding variance can be calculated explicitly based on the above setup. Therefore, an oracle test procedure can be constructed by using the quantiles of the corresponding limiting Gaussian distribution to define the test rejection region. We will compare the performance of our bootstrap testing with this benchmark oracle testing.

We conduct hypothesis testing of nominal level $0.05$ and $0.01$ for $X_1$ taking value at 0.3, 0.5 and 0.7. For each value of $X_1$, three different sample sizes from 500 to 2000 are considered. We report the empirical size and power of both bootstrap testing and oracle testing based on 500 simulations in Tables \ref{tab:hp} and \ref{tab:addhp0.8}.
\begin{table}[h!]
	\centering
	\caption{Size and power for  bootstrap testing and oracle testing ($\alpha=1$).}
	\label{tab:hp}
	{\small
		%	\resizebox{\textwidth}{!}{
		\begin{tabularx}{\textwidth}{ A A Z Z Z Z Z Z Z Z }
			\toprule
			\multicolumn{1}{c}{} & \multicolumn{1}{c}{} &  \multicolumn{4}{c}{Bootstrap testing} & \multicolumn{4}{c}{Oracle testing}\\ 
			\cmidrule(lr){3-6} \cmidrule(lr){7-10} 
			\multicolumn{1}{c}{Design Point}& \multicolumn{1}{c}{Sample size} & \multicolumn{2}{c}{Size} & \multicolumn{2}{c}{Power} & \multicolumn{2}{c}{Size} & \multicolumn{2}{c}{Power} \\
			\cmidrule(lr){3-4} \cmidrule(lr){5-6} \cmidrule(lr){7-8} \cmidrule(lr){9-10}
			
			\multicolumn{1}{c}{}                              & \multicolumn{1}{c}{}                             & \multicolumn{1}{c}{0.05}                & \multicolumn{1}{c}{0.01}               & \multicolumn{1}{c}{0.05}            & \multicolumn{1}{c}{0.01}  &   \multicolumn{1}{c}{0.05} &\multicolumn{1}{c}{0.01} &\multicolumn{1}{c}{0.05} &\multicolumn{1}{c}{0.01}        \\
			\midrule
			\multirow{3}{*}{$X_1=0.3$}                            & $n=500$                                            & 0            & \multicolumn{1}{c}{0}              & 0.622           & 0.348      & 0.012            & \multicolumn{1}{c}{0}              & 1           & 0.996                \\
			
			& $n=1000$                                           & 0.006               & \multicolumn{1}{c}{0.002}               & 0.814            & 0.628       & 0.022              & \multicolumn{1}{c}{0.004}               & 1            & 1               \\
			
			& $n=2000$                                           & 0                & \multicolumn{1}{c}{0}             & 0.932           & 0.874          & 0.032                & \multicolumn{1}{c}{0.002}             & 1           & 1        \\
			\midrule
			\multirow{3}{*}{$X_1=0.5$}                            & $n=500$                                            & 0.004            & \multicolumn{1}{c}{0}              & 0.55           & 0.322                 & 0.016            & \multicolumn{1}{c}{0.002}              & 0.994            & 0.984               \\
			& $n=1000$                                           & 0.004               & \multicolumn{1}{c}{0}               & 0.794            & 0.608            & 0.048                & \multicolumn{1}{c}{0.002}               & 1            & 0.996        \\
			& $n=2000$                                           & 0.004                & \multicolumn{1}{c}{0}             & 0.928           & 0.84        & 0.022                & \multicolumn{1}{c}{0.006}             & 1            & 1              \\
			\midrule
			\multirow{3}{*}{$X_1=0.7$}                            & $n=500$                                            & 0.002            & \multicolumn{1}{c}{0}              & 0.596           & 0.408               & 0.012            & \multicolumn{1}{c}{0.02}              & 0.998           & 0.994       \\
			& $n=1000$                                           & 0               & \multicolumn{1}{c}{0}               & 0.816            & 0.644           & 0.024               & \multicolumn{1}{c}{0.006}               & 0.998            & 0.996         \\
			& $n=2000$                                           & 0.004                & \multicolumn{1}{c}{0}             & 0.928           & 0.842         & 0.038                & \multicolumn{1}{c}{0.006}             & 1           & 1         \\
			\bottomrule   
		\end{tabularx}
		%	}
	}
\end{table}

\iftrue
\begin{table}[h!]
	\centering
	\caption{Power for  bootstrap testing and oracle testing with $\alpha=0.8$.}
	\label{tab:addhp0.8}
	\resizebox{\textwidth}{!}{
		\begin{tabularx}{\textwidth}{ Z Z Z Z Z Z }
			\toprule
			\multicolumn{1}{c}{\multirow{2}{*}{Design point}} & \multicolumn{1}{c}{\multirow{2}{*}{Sample size}} & \multicolumn{2}{c}{Bootstrap testing} & \multicolumn{2}{c}{Oracle testing}  \\
			\cmidrule(lrr){3-4} \cmidrule(lrr){5-6} 
			
			\multicolumn{1}{c}{}                              & \multicolumn{1}{c}{}                             & 0.05                & 0.01               & 0.05            & 0.01                     \\
			\midrule
			\multirow{3}{*}{$X_1=0.3$}                            & $n=500$                                            & 0.39            & \multicolumn{1}{c}{0.176}              & 0.978           & 0.894                   \\
			
			& $n=1000$                                           & 0.66               & \multicolumn{1}{c}{0.406}               & 0.992            & 0.98                   \\
			
			& $n=2000$                                           & 0.858                & \multicolumn{1}{c}{0.696}             & 1           & 1               \\
			\midrule
			\multirow{3}{*}{$X_1=0.5$}                            & $n=500$                                            & 0.392            & \multicolumn{1}{c}{0.18}              & 0.962          & 0.846                              \\
			& $n=1000$                                           & 0.628               & \multicolumn{1}{c}{0.396}               & 0.992            & 0.968                   \\
			& $n=2000$                                           & 0.874                & \multicolumn{1}{c}{0.706}             & 1          & 0.998                  \\
			\midrule
			\multirow{3}{*}{$X_1=0.7$}                            & $n=500$                                            & 0.416            & \multicolumn{1}{c}{0.21}              & 0.976           & 0.894                    \\
			& $n=1000$                                           & 0.66               & \multicolumn{1}{c}{0.432}               & 0.996            & 0.982                \\
			& $n=2000$                                           & 0.872                & \multicolumn{1}{c}{0.708}             & 1           & 0.998                \\
			\bottomrule   
		\end{tabularx}
	}
\end{table}
\fi

From the tables, the Type \rom{1} errors are well preserved for both tests at three design points while the bootstrap testing committed slightly fewer Type \rom{1} errors. We can see the decrease of power of both bootstrap testing and oracle testing under the same design point ($X_1$) and subsample size ($n$) as $\alpha$ gets smaller. For a fixed design point, the power of both tests approaches $1$ with the increasing subsample size which supports our theory. In particular, the good performance of the oracle testing justifies our normal approximation theory. The performance of the proposed bootstrap testing is inferior to the oracle testing under the same design point and subsample size. This may due to several reasons including the bootstrap approximation error and the bias in the estimation of the nuisance parameters. However, the performance of bootstrap testing is reasonable when the sample size is sufficiently large which agrees with our asymptotic theory.    

We also remark that we here test the significance of the covariates for the conditional mode by testing the change of the conditional mode due to the change of design points, instead of testing the corresponding coefficient in the quantile regression slope vector. This is due to the following observation. Although we assume a linear quantile model, it does not imply a direct modeling on the conditional mode function. That is, the conditional mode depends on the covariates in an implicit way under our quantile-based mode regression, i.e. under our modeling, the mode function is $m(\vx)=\vx^T\beta(\tau_{\vx})$ where the coordinates of $\beta(\tau_{\vx})$ are functions of $\vx$ that may have arbitrary forms (as long as they satisfy some natural restrictions resulted from quantile functions). 
It is possible that one coefficient of the $\beta(\tau_{\vx})$ is nonzero while the corresponding covariate does not contribute to the conditional mode function. 
This is in contrast to the linear modal regression where the mode function is assumed to be $m(\vx)=\vx^T\beta$ where $\beta$ is a constant vector (independent of $\vx$) and therefore testing the significance of the covariates is equivalent to testing the coefficients in $\beta$ being $0$ or not. 

\subsection{Pivotal bootstrap confidence intervals using oracle information}
\label{sec: oracle}
In this section, we provide simulation results of pivotal bootstrap inference using oracle model information, i.e., we estimate the nuisance parameters in the influence function based on the underlying true density or conditional quantile function. We reexamine the setups in Section \ref{sec: pCI simu} and the corresponding results for pointwise confidence intervals are presented in Tables \ref{tab:point normal oracle} to \ref{tab:point nonlinear oracle} and Table \ref{tab:oracle band} for approximate confidence bands. 
%\subsubsection{Pointwise confidence intervals}
\begin{table}[h!]
	\centering
	\caption{Oracle pointwise confidence intervals for \emph{lmNormal} model.}
	\label{tab:point normal oracle}
	{\small
		%\resizebox{\textwidth}{!}{
		\begin{tabularx}{\textwidth}{ Z Z Z Z Z Z Z Z}
			\toprule
			\multicolumn{1}{c}{\multirow{2}{*}{Design point}} & \multicolumn{1}{c}{\multirow{2}{*}{Sample size}} & \multicolumn{2}{c}{Coverage probability} & \multicolumn{2}{c}{Median length} & \multicolumn{2}{c}{Interquartile range}  \\
			\cmidrule(lrr){3-4} \cmidrule(lrr){5-6} \cmidrule(lrr){7-8}
			
			\multicolumn{1}{c}{}                              & \multicolumn{1}{c}{}                             & 95\%                & 99\%               & 95\%            & 99\%        & 95\%            & 99\%               \\
			\midrule
			\multirow{3}{*}{$X_1$=0.3}                            & $n=500$                                            & 95.4\%              & \multicolumn{1}{c}{99.6\%}              & 0.94            & 1.24    & 0.16            & 0.19                   \\
			& $n=~1000$                                           & 95\%                & \multicolumn{1}{c}{99.2\%}               & 0.79            & 1.02     & 0.13            & 0.14            \\
			& $n=~2000$                                           & 96.6\%                & \multicolumn{1}{c}{99\%}             & 0.64            & 0.85    & 0.08           & 0.10       \\
			\midrule
			\multirow{3}{*}{$X_1$=0.5}                            & $n=500$                                            & 95.2\%              & \multicolumn{1}{c}{99.2\%}              & 1.04            & 1.38    & 0.24            & 0.32                    \\
			& $n=~1000$                                           & 94.6\%                & \multicolumn{1}{c}{98.6\%}               & 0.86            & 1.14   & 0.15            & 0.20                    \\
			& $n=~2000$                                           & 95\%                & \multicolumn{1}{c}{99.2\%}             & 0.71            & 0.94      & 0.11            & 0.13     
			\\
			\midrule
			\multirow{3}{*}{$X_1$=0.7}                            & $n=500$                                            & 96.6\%              & \multicolumn{1}{c}{98.8\%}              & 1.29            & 1.70      & 0.37            & 0.41                 \\
			& $n=~1000$                                           & 96\%                & \multicolumn{1}{c}{99.4\%}               & 1.06            & 1.42       & 0.22            & 0.27                \\
			& $n=~2000$                                           & 95.6\%                & \multicolumn{1}{c}{99.6\%}             & 0.88            & 1.16     & 0.16            & 0.21       
			\\
			\bottomrule   
		\end{tabularx}	
	}
\end{table}

\begin{table}[h!]
	\centering
	\caption{Oracle pointwise confidence intervals for \emph{lmLognormal} model.}
	\label{tab:point lognormal oracle}
	{\small
		%\resizebox{\textwidth}{!}{
		\begin{tabularx}{\textwidth}{ Z Z Z Z Z Z Z Z}
			\toprule
			\multicolumn{1}{c}{\multirow{2}{*}{Design point}} & \multicolumn{1}{c}{\multirow{2}{*}{Sample size}} & \multicolumn{2}{c}{Coverage probability} & \multicolumn{2}{c}{Median length}  & \multicolumn{2}{c}{Interquartile range} \\
			\cmidrule(lrr){3-4} \cmidrule(lrr){5-6}  \cmidrule(lrr){7-8} 
			
			\multicolumn{1}{c}{}                              & \multicolumn{1}{c}{}                             & 95\%                & 99\%               & 95\%            & 99\%     & 95\%            & 99\%                 \\
			\midrule
			\multirow{3}{*}{$X_1$=0.3}                            & $n=500$                                            & 77\%              & \multicolumn{1}{c}{80.6\%}              & 1.44            & 1.86     & 0.91         & 1.20                   \\
			& $n=~1000$                                           & 94\%                & \multicolumn{1}{c}{96\%}               & 1.34           & 1.75    & 0.52            & 0.63              \\
			& $n=~2000$                                           & 96.4\%                & \multicolumn{1}{c}{99.8\%}             & 1.01            & 1.33     & 0.28            & 0.35       \\
			\midrule
			\multirow{3}{*}{$X_1$=0.5}                            & $n=500$                                            & 83.4\%              & \multicolumn{1}{c}{84.8\%}              & 2.07            & 2.67    & 1.31            & 1.77                    \\
			& $n=~1000$                                           & 96\%                & \multicolumn{1}{c}{97.8\%}               & 1.85            & 2.43    & 0.79         & 1.00                    \\
			& $n=~2000$                                           & 97.4\%                & \multicolumn{1}{c}{99.8\%}             & 1.52            & 2.02     & 0.54            & 0.69      
			\\
			\midrule
			\multirow{3}{*}{$X_1$=0.7}                            & $n=500$                                            & 76.6\%              & \multicolumn{1}{c}{81\%}              & 2.33           & 3.19       & 1.65           & 2.26                 \\
			& $n=~1000$                                           & 93.4\%                & \multicolumn{1}{c}{94.8\%}               & 2.25            & 3.03      & 1.36            & 1.61                   \\
			& $n=~2000$                                           & 96.2\%                & \multicolumn{1}{c}{98.6\%}             & 1.84           & 2.43    & 0.76            & 0.92       
			\\
			\bottomrule   
		\end{tabularx}	
	}
\end{table}

\begin{table}[h!]
	\centering
	\caption{Oracle pointwise confidence intervals for \emph{Nonlinear} model.}
	\label{tab:point nonlinear oracle}
	{\small
		%\resizebox{\textwidth}{!}{
		\begin{tabularx}{\textwidth}{ Z Z Z Z Z Z Z Z}
			\toprule
			\multicolumn{1}{c}{\multirow{2}{*}{Design point}} & \multicolumn{1}{c}{\multirow{2}{*}{Sample size}} & \multicolumn{2}{c}{Coverage probability} & \multicolumn{2}{c}{Median length}  & \multicolumn{2}{c}{Interquartile range}\\
			\cmidrule(lrr){3-4} \cmidrule(lrr){5-6}  \cmidrule(lrr){7-8} 
			
			\multicolumn{1}{c}{}                              & \multicolumn{1}{c}{}                             & 95\%                & 99\%               & 95\%            & 99\%         & 95\%            & 99\%             \\
			\midrule
			\multirow{3}{*}{$X_1$=0.7}                            & $n=500$                                            & 92.4\%              & \multicolumn{1}{c}{97.4\%}              & 0.58           & 0.73     & 0.48           & 0.47                 \\
	 	   & $n=~1000$                                           & 96\%                & \multicolumn{1}{c}{98.8\%}               & 0.45            & 0.59      & 0.45          & 0.49                 \\
	    	& $n=~2000$                                           & 97\%                & \multicolumn{1}{c}{99\%}             & 0.35           & 0.46        & 0.22           & 0.29   
		    \\
			\midrule
			\multirow{3}{*}{$X_1$=0.9}                            & $n=500$                                            & 94.4\%              & \multicolumn{1}{c}{97.2\%}              & 0.62            & 0.81   & 0.46           & 0.57                    \\
			& $n=~1000$                                           & 94.8\%                & \multicolumn{1}{c}{97.6\%}               & 0.50           & 0.66    & 0.24           & 0.32                  \\
			& $n=~2000$                                           & 96\%                & \multicolumn{1}{c}{98.8\%}             & 0.45            & 0.58        & 0.19           & 0.24 
			\\
			\midrule
			\multirow{3}{*}{$X_1$=1.1}                            & $n=500$                                            & 93.2\%              & \multicolumn{1}{c}{96.8\%}              & 0.53           & 0.70      & 0.18           & 0.26                 \\
			& $n=~1000$                                           & 94.8\%                & \multicolumn{1}{c}{97.6\%}               & 0.46            & 0.59     & 0.17           & 0.21                  \\
			& $n=~2000$                                           & 94\%                & \multicolumn{1}{c}{98\%}             & 0.37           & 0.48        & 0.10           & 0.14   
			\\
			\bottomrule   
		\end{tabularx}	
	}
\end{table}

\begin{table}[h!]
	\centering
	\caption{Oracle approximate confidence bands for \emph{lmNormal}, \emph{lmLognormal} and \emph{Nonlinear} models.}
	{\small
		%\resizebox{\textwidth}{!}{
		\begin{tabularx}{\textwidth}{@{} Z Z A A A A @{} }
			\toprule
			\multicolumn{1}{c}{\multirow{2}{*}{Models}} & \multicolumn{1}{c}{\multirow{2}{*}{Sample size}} & \multicolumn{2}{c}{Coverage probability} & \multicolumn{2}{c}{Median length}  \\
			\cmidrule(lr){3-4} \cmidrule(lrr){5-6} 
			
			\multicolumn{1}{c}{}                              & \multicolumn{1}{c}{}                             & \multicolumn{1}{c}{95\%}                & \multicolumn{1}{c}{99\%}               & 95\%            & 99\%                     \\
			\midrule
			\multirow{3}{*}{lmNormal}                            & $n=500$                                            & 94.8\%              & \multicolumn{1}{c}{98.6\%}              & 1.14            & 1.48                 \\
			& $n=1000$                                           & 94.6\%                & \multicolumn{1}{c}{98.8\%}               & 0.97           & 1.22              \\
			& $n=2000$                                           & 93.4\%                & \multicolumn{1}{c}{98.8\%}             & 0.79           & 1.02        \\
			\midrule
			\multirow{3}{*}{lmLognormal}                            & $n=500$                                            & 82.6\%              & \multicolumn{1}{c}{84.2\%}              & 2.39           & 3.06  \\
			& $n=1000$                                           & 95.2\%                & \multicolumn{1}{c}{97.1\%}               & 2.03           & 2.65  \\
			& $n=2000$                                           & 97.6\%                & \multicolumn{1}{c}{99.4\%}             & 1.73            & 2.18           \\
			\midrule
			\multirow{3}{*}{Nonlinear}                            & $n=500$                                            & 94.6\%              & \multicolumn{1}{c}{97.6\%}              & 1.12           & 1.23           \\
			& $n=1000$                                           & 96\%                & \multicolumn{1}{c}{97.8\%}               & 0.82         & 0.98            \\
			& $n=2000$                                           & 94.2\%                & \multicolumn{1}{c}{98\%}             & 0.49          & 0.59        \\
			\bottomrule   
		\end{tabularx}
		%}	
	}
	\label{tab:oracle band}
\end{table}
From the simulation results, the oracle confidence intervals achieve close to nominal level coverage probabilities when the sample size is sufficiently large, which supports our theoretical results.

\section{Additional discussion: model misspecification and quantile crossing}
In theory, the quantile crossing problem does not happen since we assume that
\[
s_{\vx}(\tau) = Q_{\vx}'(\tau) =1/f(Q_{\vx}(\tau) \mid \vx ) \ge 1/c_1 > 0, \ \tau \in [\epsilon,1-\epsilon].
\]
See  Condition (v) in Assumption 1. This implies the quantile function is strictly increasing in the quantile index $\tau \in [\epsilon,1-\epsilon]$. Further, under our assumption, $\hat{s}_{\vx}(\tau) = \hat{Q}_{\vx}'(\tau)$ is uniformly consistent over $\tau \in [\epsilon,1-\epsilon]$ and $\vx \in \mathcal{X}_0$ (see Lemma 7 in Appendix), so that the estimated conditional quantile function $\tau \mapsto \hat{Q}_{\vx}(\tau)$ is strictly increasing on $[\epsilon,1-\epsilon]$ with probability approaching one.

Of course,  the quantile crossing problem may happen in the finite sample  even if the model is correctly specified (Koenker 2005). Several methods have been proposed to deal with the quantile crossing problem in the literature, cf. He (1997), Chernozhukov et al. (2009) and Bondell et al. (2010). We can apply any of such monotonization methods to the estimated conditional quantile function to prevent quantile crossing; our theoretical results continue to hold for such a monotonized conditional quantile estimate, as the monotonized estimate agrees with the vanilla conditional quantile estimate with probability approaching one. 

Under model misspecification, the fitted linear quantile function can be interpreted as the best linear approximation to the quantile function under some quadratic discrepancy measure cf. section 2.9 in Koenker (2005) and Angrist et al. (2006). Thus, our $\tau_x$ can be interpreted as a minimization point of the derivative of the best linear approximation. 
%Therefore, our modal estimator can be interpreted in such sense accordingly.
Significant number of quantile crossings also implies the misspecification of the model (Koenker (2005)). If this happens, we suggest using the series approximation approach of our method as discussed in Section 5.

\section{More implementation details of Section~4.1}
In this section, we provide more implementation details of estimating $f^{(2)}(m(\vx)~|~\vx)$ in the simulation.  % In general, we use the kernel method as in remark 9 of \cite{ohta2018quantile}  (we plug in $\hat{Q}_{\vx}(\hat{\tau}_{\vx})$ and $\hat{s}_{\vx}(\hat{\tau}_{\vx})$ for $m(\vx)$ and $s_{\vx}(\tau_{\vx})$, respectively).  
We use the following kernel estimator for $f^{(2)}(m(\vx)~|~\vx)$ (recall that $\vX=(1,X_1)$):
\[
\hat{f}^{(2)}(\hat{m}(\vx)~|~\vx)=\frac{(nb_{Y}^3b_{X_1})^{-1}\textstyle\sum_{i=1}^n K_1''((\hat{m}(\vx)-Y_i)/b_Y) K_2((x_1-X_{i1})/b_{X_1})  }{(nb_{X_1})^{-1}\textstyle\sum_{i=1}^n K_2((x_1-X_{i1})/b_{X_1})},
\]
where $X_{i1}$ is the observed value of $X_1$ in the $i$-th data point. We use the Gaussian kernel for $K_1$
and the Epanechnikov kernel for $K_2$ in the simulation.
We use the following bandwidths for covariate $X_1$ and $Y$ in our simulation, respectively:
\begin{equation}
b_{X_1}=\omega \cdot n^{-1/5}\hat{\sigma}_X \quad \text{ and } \quad b_{Y}=\omega \cdot n^{-1/9}\hat{\sigma}_Y,
\end{equation}
where $\hat{\sigma}_{\cdot}$ is the corresponding sample standard deviation and $\omega>0$. To select $\omega$ in a data-driven approach, we propose the following procedure.  The procedure is motivated by the $L_{\infty}$-based bandwidth selector in \cite{bissantz2007non}. 
\begin{enumerate}
	\item[\textbf{Step 1}.] Choose a proper grid $G_1$ with $J$ values for $\omega$. 
	\item[\textbf{Step 2}.] Generate a dataset $\mathcal{D}_n$ of size $n$ and then compute $\hat{f}^{(2)}(\hat{m}(\vx)~|~\vx)$ by taking $\omega=G_1(j)$ ($1\le j \le J$) and denote the resulting estimator by $\hat{f}_j^{(2)}$. Compute $d_{j,j+1}:=\big|\hat{f}_{j+1}^{(2)}-\hat{f}_j^{(2)}\big|$ ($1\le j \le J-1$). Choose 
	$\omega^*:= \max\{ G_1(j): d_{j,j+1}\ge t\cdot d_{J-1,J}, 1\le j\le J-1 \}$ for some $t>1$. Repeat the computation of $\omega^*$ for $N$ times to get $\omega^*_1,\cdots,\omega^*_N$ and denote the mode of $\{\omega^*_1,\cdots,\omega^*_N\}$s by $\omega^*_{opt}$. 
	\item[\textbf{Step 3}.] Choose a subgrid $G_2\subset G_1$ centering at $\omega^*_{opt}$. Then we proceed as Step 2 and output the final selected $\omega$.
\end{enumerate} 
In fact, we may iterate the above procedure for more times by using a further subgrid based on the output of Step 3. However, we find the above three-step algorithm works reasonably well in our numerical experiments. In our simulation of pointwise confidence intervals, we take $G_1$ as a equally spaced grid on $[0.05,1.25]$ with $J=13$; $ 2\le t \le 4$ ($t=2$ for \emph{lmNormal} and $3\le t\le4$ for \emph{lmLogNormal} and \emph{Nonlinear});  $N=500$ and $G_2$ of length $7$.
To relieve the computational burden, for each model we only select $\omega$ once for a particular design point $\vx$ at sample size $n=2000$. The selected values of $\omega$ for pointwise confidence intervals are reported in Tables \ref{tab:omega1} and \ref{tab:omega2}.  For the simulation of confidence bands, we use a common $\omega$ for all the different design points in the considered models to reduce the computational burden and the corresponding values are reported in Table \ref{tab:omega3}. Based on our numerical experience, we recommend a slightly larger $\omega$ than the pointwisely selected $\omega$ for the confidence bands if a common $\omega$ is used. In practice, we may use pointwisely selected $\omega$ for different design points when constructing confidence bands.
 For the practical use, $\mathcal{D}_n$ in Step 2 may be taken as the bootstrap subsamples. 

\begin{table}[h!]
	\centering
	\caption{Values of $\omega$ selected for \emph{lmNormal} and \emph{lmLognormal} models}
	{\small
		%\resizebox{\textwidth}{!}{
		\begin{tabularx}{\textwidth}{@{} Z Z Z Z  @{} }
			\toprule
			\multicolumn{1}{c}{\multirow{2}{*}{Models}} & \multicolumn{1}{c}{\multirow{2}{*}{$X_1=0.3$}} & \multicolumn{1}{c}{\multirow{2}{*}{$X_1=0.5$}} & \multicolumn{1}{c}{\multirow{2}{*}{$X_1=0.7$}} \\

			\multicolumn{1}{c}{}                              & \multicolumn{1}{c}{}                             & \multicolumn{1}{c}{}               & \multicolumn{1}{c}{}                              \\
			\midrule
			\multirow{1}{*}{lmNormal}                                                                     &0.75             & 0.85              & 0.95                  \\
			%\midrule
			\multirow{1}{*}{lmLogormal}                                                                     &0.35             & 0.45              & 0.55                  \\
			\bottomrule   
		\end{tabularx}
		%}	
	}
	\label{tab:omega1}
\end{table}

\begin{table}[h!]
	\centering
	\caption{Values of $\omega$ selected for \emph{Nonlinear} model}
	{\small
		%\resizebox{\textwidth}{!}{
		\begin{tabularx}{\textwidth}{@{} Z Z Z Z  @{} }
			\toprule
			\multicolumn{1}{c}{\multirow{2}{*}{Models}} & \multicolumn{1}{c}{\multirow{2}{*}{$X_1=0.7$}} & \multicolumn{1}{c}{\multirow{2}{*}{$X_1=0.9$}} & \multicolumn{1}{c}{\multirow{2}{*}{$X_1=1.1$}} \\

			\multicolumn{1}{c}{}                              & \multicolumn{1}{c}{}                             & \multicolumn{1}{c}{}               & \multicolumn{1}{c}{}                              \\
			\midrule

			\multirow{1}{*}{Nonlinear}                                                                     &0.55             & 0.65              & 0.75                  \\
			\bottomrule   
		\end{tabularx}
		%}	
	}
	\label{tab:omega2}
\end{table}

\begin{table}[h!]
	\centering
	\caption{Values of $\omega$ selected for approximate confidence bands}
	{\small
		%\resizebox{\textwidth}{!}{
		\begin{tabularx}{\textwidth}{@{} Z Z Z Z  @{} }
			\toprule
			\multicolumn{1}{c}{\multirow{2}{*}{Models}} & \multicolumn{1}{c}{\multirow{2}{*}{lmNormal}} & \multicolumn{1}{c}{\multirow{2}{*}{lmLognormal}} & \multicolumn{1}{c}{\multirow{2}{*}{Nonlinear}} \\

			\multicolumn{1}{c}{}                              & \multicolumn{1}{c}{}                             & \multicolumn{1}{c}{}               & \multicolumn{1}{c}{}                              \\
			\midrule
			\multirow{1}{*}{$\omega$}                                                                     &1.00             & 0.60              & 0.80                  \\
			%\midrule
			\bottomrule   
		\end{tabularx}
		%}	
	}
	\label{tab:omega3}
\end{table}

\section{More details on the U.S. wage dataset}
The data are extracted from U.S. 1980 1\% metro sample from the
Integrated Public Use Microdata Series (IPUMS) website \citep{ruggles2020ipums}. We first collect data of black and white people aged 30 - 60 with at least kindergarten level of education (nursery school is excluded), with positive annual earnings in the year preceding the census. Individuals with missing values for age, education and earnings are also excluded from the sample. Then we randomly sample $10,000$ people from the single and married groups, respectively and combine the resulting $20,000$ data as the final U.S. wage dataset.

The wage variable (wage) is the log annual wage, calculated as the
log of the reported annual income from work in the previous year. 
The education variable (edu) corresponds to the highest grade of school
completed starting from $0$ which indicates the kindergarten level. For the marital status variable (marital\_status), 0 stands for "being single" and 1 stands for "being married". For the race variable (race), 1 corresponds to white people and 2 corresponds to black people. For the sex variable (sex), 1 corresponds to male and 2 correspond to female.
\end{appendices}

\bibliographystyle{apalike}
\bibliography{Bibs} 

\begin{thebibliography}{}

\bibitem[Aliprantis and Border, 2006]{aliprantis06}
Aliprantis, C.~D. and Border, K.~C. (2006).
\newblock {\em Infinite Dimensional Analysis: a Hitchhiker's Guide}.
\newblock Springer, Berlin; London.

\bibitem[Autor et~al., 2008]{autor2008trends}
Autor, D.~H., Katz, L.~F., and Kearney, M.~S. (2008).
\newblock Trends in us wage inequality: Revising the revisionists.
\newblock {\em The Review of Economics and Statistics}, 90(2):300--323.

\bibitem[Bamford et~al., 2008]{bamford2008revealing}
Bamford, S.~P., Rojas, A.~L., Nichol, R.~C., Miller, C.~J., Wasserman, L.,
  Genovese, C.~R., and Freeman, P.~E. (2008).
\newblock Revealing components of the galaxy population through non-parametric
  techniques.
\newblock {\em Monthly Notices of the Royal Astronomical Society},
  391(2):607--616.

\bibitem[Belloni et~al., 2019]{belloni2019conditional}
Belloni, A., Chernozhukov, V., Chetverikov, D., and Fernández-Val, I. (2019).
\newblock Conditional quantile processes based on series or many regressors.
\newblock {\em Journal of Econometrics}, 213(1):4 -- 29.

\bibitem[Belloni et~al., 2015]{belloni2015joe}
Belloni, A., Chernozhukov, V., Chetverikov, D., and Kato, K. (2015).
\newblock Some new asymptotic theory for least squares series: Pointwise and
  uniform results.
\newblock {\em Journal of Econometrics}, 186:345--366.

\bibitem[Bissantz et~al., 2007]{bissantz2007non}
Bissantz, N., D{\"u}mbgen, L., Holzmann, H., and Munk, A. (2007).
\newblock Non-parametric confidence bands in deconvolution density estimation.
\newblock {\em Journal of the Royal Statistical Society: Series B (Statistical
  Methodology)}, 69(3):483--506.

\bibitem[Bound and Krueger, 1991]{bound1991extent}
Bound, J. and Krueger, A.~B. (1991).
\newblock The extent of measurement error in longitudinal earnings data: Do two
  wrongs make a right?
\newblock {\em Journal of Labor Economics}, 9(1):1--24.

\bibitem[Buchinsky, 1994]{buchinsky1994changes}
Buchinsky, M. (1994).
\newblock Changes in the {U.S.} wage structure 1963-1987: Application of
  quantile regression.
\newblock {\em Econometrica}, 62(2):405--458.

\bibitem[Chac\'{o}n, 2018]{chacon2018}
Chac\'{o}n, J. (2018).
\newblock The modal age of statistics.
\newblock arXiv:1807.02789.

\bibitem[Chen and Kato, 2020]{chen2019jackknife}
Chen, X. and Kato, K. (2020).
\newblock Jackknife multiplier bootstrap: finite sample approximations to the
  u-process supremum with applications.
\newblock {\em Probability Theory and Related Fields}, 176(3):1097--1163.

\bibitem[Chen, 2018]{chen2018modal}
Chen, Y.-C. (2018).
\newblock Modal regression using kernel density estimation: A review.
\newblock {\em Wiley Interdisciplinary Reviews: Computational Statistics},
  10(4):e1431.

\bibitem[Chen et~al., 2016]{chen2016nonparametric}
Chen, Y.-C., Genovese, C.~R., Tibshirani, R.~J., and Wasserman, L. (2016).
\newblock Nonparametric modal regression.
\newblock {\em The Annals of Statistics}, 44(2):489--514.

\bibitem[Cheng, 1995]{cheng1995mean}
Cheng, Y. (1995).
\newblock Mean shift, mode seeking, and clustering.
\newblock {\em IEEE Transactions on Pattern Analysis and Machine Intelligence},
  17(8):790--799.

\bibitem[Chernozhukov et~al., 2014]{chernozhukov2014gaussian}
Chernozhukov, V., Chetverikov, D., and Kato, K. (2014).
\newblock Gaussian approximation of suprema of empirical processes.
\newblock {\em The Annals of Statistics}, 42(4):1564--1597.

\bibitem[Chernozhukov et~al., 2016]{chernozhukov2016empirical}
Chernozhukov, V., Chetverikov, D., and Kato, K. (2016).
\newblock Empirical and multiplier bootstraps for suprema of empirical
  processes of increasing complexity, and related gaussian couplings.
\newblock {\em Stochastic Processes and their Applications},
  126(12):3632--3651.

\bibitem[Chernozhukov et~al., 2017a]{chernozhukov2017central}
Chernozhukov, V., Chetverikov, D., and Kato, K. (2017a).
\newblock Central limit theorems and bootstrap in high dimensions.
\newblock {\em The Annals of Probability}, 45(4):2309--2352.

\bibitem[Chernozhukov et~al., 2017b]{chernozhukov2017note}
Chernozhukov, V., Chetverikov, D., and Kato, K. (2017b).
\newblock Detailed proof of nazarov's inequality.
\newblock arXiv:1711.10696.

\bibitem[Chernozhukov et~al., 2009]{chernozhukov2009}
Chernozhukov, V., Hansen, C., and Jansson, M. (2009).
\newblock Finite sample inference for quantile regression models.
\newblock {\em Journal of Econometrics}, 152:93--103.

\bibitem[Deng and Zhang, 2017]{deng2017beyond}
Deng, H. and Zhang, C.-H. (2017).
\newblock Beyond gaussian approximation: Bootstrap for maxima of sums of
  independent random vectors.
\newblock {\em arXiv preprint arXiv:1705.09528}.

\bibitem[Einbeck and Tutz, 2006]{einbeck2006modelling}
Einbeck, J. and Tutz, G. (2006).
\newblock Modelling beyond regression functions: an application of multimodal
  regression to speed--flow data.
\newblock {\em Journal of the Royal Statistical Society: Series C (Applied
  Statistics)}, 55(4):461--475.

\bibitem[Feng et~al., 2020]{feng2020statistical}
Feng, Y., Fan, J., and Suykens, J.~A. (2020).
\newblock A statistical learning approach to modal regression.
\newblock {\em Journal of Machine Learning Research}, 21(2):1--35.

\bibitem[Gutenbrunner and Jureckov{\'a}, 1992]{gutenbrunner1992regression}
Gutenbrunner, C. and Jureckov{\'a}, J. (1992).
\newblock Regression rank scores and regression quantiles.
\newblock {\em The Annals of Statistics}, pages 305--330.

\bibitem[Hall, 1991]{hall1991}
Hall, P. (1991).
\newblock On convergence rates of suprema.
\newblock {\em Probability Theory and Related Fields}, 89(4):447--455.

\bibitem[Han and Wellner, 2019]{han2019convergence}
Han, Q. and Wellner, J.~A. (2019).
\newblock Convergence rates of least squares regression estimators with
  heavy-tailed errors.
\newblock {\em Annals of Statistics}, 47(4):2286--2319.

\bibitem[He, 2017]{he2017resampling}
He, X. (2017).
\newblock Resampling methods.
\newblock In {\em Handbook of quantile regression}, pages 7--19. Chapman and
  Hall/CRC.

\bibitem[He and Shao, 1996]{he1996general}
He, X. and Shao, Q.-M. (1996).
\newblock A general bahadur representation of m-estimators and its application
  to linear regression with nonstochastic designs.
\newblock {\em The Annals of Statistics}, 24(6):2608--2630.

\bibitem[Heckman et~al., 2001]{heckman2001molecular}
Heckman, D.~S., Geiser, D.~M., Eidell, B.~R., Stauffer, R.~L., Kardos, N.~L.,
  and Hedges, S.~B. (2001).
\newblock Molecular evidence for the early colonization of land by fungi and
  plants.
\newblock {\em Science}, 293(5532):1129--1133.

\bibitem[Hedges and Shah, 2003]{hedges2003comparison}
Hedges, S.~B. and Shah, P. (2003).
\newblock Comparison of mode estimation methods and application in molecular
  clock analysis.
\newblock {\em BMC Bioinformatics}, 4(1):31.

\bibitem[Ho et~al., 2017]{ho2017}
Ho, C., Damien, P., and Walker, S. (2017).
\newblock Bayesian mode regression using mixtures of triangular densities.
\newblock {\em Journal of Econometrics}, 197(2):273--283.

\bibitem[Hu and Schennach, 2008]{hu2008instrumental}
Hu, Y. and Schennach, S.~M. (2008).
\newblock Instrumental variable treatment of nonclassical measurement error
  models.
\newblock {\em Econometrica}, 76(1):195--216.

\bibitem[Kemp and Santos-Silva, 2012]{kemp2012regression}
Kemp, G.~C. and Santos-Silva, J. (2012).
\newblock Regression towards the mode.
\newblock {\em Journal of Econometrics}, 170(1):92--101.

\bibitem[Khardani and Yao, 2017]{khardani2017}
Khardani, S. and Yao, A. (2017).
\newblock Non linear parametric mode regression.
\newblock {\em Communications in Statistics-Theory and Methods},
  46(6):3006--3024.

\bibitem[Koenker, 2005]{koenker2005quantile}
Koenker, R. (2005).
\newblock {\em Quantile Regression}.
\newblock Econometric Society Monographs. Cambridge University Press.

\bibitem[Koenker, 2017]{koenker2017quantile}
Koenker, R. (2017).
\newblock Quantile regression: 40 years on.
\newblock {\em Annual Review of Economics}, 9:155--176.

\bibitem[Koenker and Bassett, 1978]{koenker1978regression}
Koenker, R. and Bassett, G. (1978).
\newblock Regression quantiles.
\newblock {\em Econometrica}, 46(1):33--50.

\bibitem[Krief, 2017]{krief2017semi}
Krief, J.~M. (2017).
\newblock Semi-linear mode regression.
\newblock {\em The Econometrics Journal}, 20(2):149--167.

\bibitem[Leadbetter et~al., 1983]{leadbetter1983}
Leadbetter, M., Lindgren, G., and Rootz{\'e}n, H. (1983).
\newblock {\em Extremes and Related Properties of Random Sequences and
  Processes}.
\newblock Springer.

\bibitem[Lee, 1989]{lee1989mode}
Lee, M.-J. (1989).
\newblock Mode regression.
\newblock {\em Journal of Econometrics}, 42(3):337--349.

\bibitem[Lee, 1993]{lee1993}
Lee, M.-J. (1993).
\newblock Quadratic mode regression.
\newblock {\em Journal of Econometrics}, 57(1-3):1--19.

\bibitem[Lee and Kim, 1998]{lee1998semiparametric}
Lee, M.-J. and Kim, H. (1998).
\newblock Semiparametric econometric estimators for a truncated regression
  model: A review with an extension.
\newblock {\em Statistica Neerlandica}, 52(2):200--225.

\bibitem[Manski, 1991]{manski1991regression}
Manski, C.~F. (1991).
\newblock Regression.
\newblock {\em Journal of Economic Literature}, 29(1):34--50.

\bibitem[Ota et~al., 2019]{ohta2018quantile}
Ota, H., Kato, K., and Hara, S. (2019).
\newblock Quantile regression approach to conditional mode estimation.
\newblock {\em Electronic Journal of Statistics}, 13(2):3120--3160.

\bibitem[Parzen et~al., 1994]{parzen1994}
Parzen, M.~I., Wei, L.~J., and Ying, Z. (1994).
\newblock A resampling method based on pivotal estimating equations.
\newblock {\em Biometrika}, 81:341--350.

\bibitem[Politis et~al., 1999]{politis1999}
Politis, D., Romano, J., and Wolf, M. (1999).
\newblock {\em Subsampling}.
\newblock Springer.

\bibitem[Powell, 1986]{powell1986censored}
Powell, J.~L. (1986).
\newblock Censored regression quantiles.
\newblock {\em Journal of Econometrics}, 32(1):143--155.

\bibitem[Raab and Steger, 1998]{raab1998balls}
Raab, M. and Steger, A. (1998).
\newblock "balls into bins" - a simple and tight analysis.
\newblock In {\em Proceedings of the Second International Workshop on
  Randomization and Approximation Techniques in Computer Science}, RANDOM '98,
  page 159–170, Berlin, Heidelberg. Springer-Verlag.

\bibitem[Rockafellar, 1970]{rockafellar1970convex}
Rockafellar, R.~T. (1970).
\newblock {\em Convex Analysis}.
\newblock Princeton Mathematical Series. Princeton University Press, Princeton,
  N. J.

\bibitem[Rudelson, 1999]{rudelson1999}
Rudelson, M. (1999).
\newblock Random vectors in the isotropic position.
\newblock {\em Journal of Functional Analysis}, 164:60--72.

\bibitem[Ruggles et~al., 2020]{ruggles2020ipums}
Ruggles, S., Flood, S., Goeken, R., Grover, J., Meyer, E., Pacas, J., and
  Sobek, M. (2020).
\newblock Ipums usa: Version 10.0 [dataset]. minneapolis, mn: Ipums; 2020.

\bibitem[Ruppert and Carroll, 1980]{ruppert1980trimmed}
Ruppert, D. and Carroll, R.~J. (1980).
\newblock Trimmed least squares estimation in the linear model.
\newblock {\em Journal of the American Statistical Association},
  75(372):828--838.

\bibitem[Sager and Thisted, 1982]{sager1982maximum}
Sager, T.~W. and Thisted, R.~A. (1982).
\newblock Maximum likelihood estimation of isotonic modal regression.
\newblock {\em Ann. Statist.}, 10(3):690--707.

\bibitem[Sasaki et~al., 2016]{sasaki2016}
Sasaki, H., Ono, Y., and Sugiyama, M. (2016).
\newblock Modal regression via direct log-density derivative estimation.
\newblock In {\em International Conference on Neural Information Processing},
  pages 108--116.

\bibitem[van~der Vaart, 2000]{van2000asymptotic}
van~der Vaart, A.~W. (2000).
\newblock {\em Asymptotic Statistics}, volume~3.
\newblock Cambridge University Press.

\bibitem[van~der Vaart and Wellner, 1996]{vaart1996weak}
van~der Vaart, A.~W. and Wellner, J.~A. (1996).
\newblock {\em Weak convergence and empirical processes: with applications to
  statistics}.
\newblock Springer.

\bibitem[Wang et~al., 2017]{wang2017regularized}
Wang, X., Chen, H., Cai, W., Shen, D., and Huang, H. (2017).
\newblock Regularized modal regression with applications in cognitive
  impairment prediction.
\newblock In {\em Advances in Neural Information Processing Systems 30}, pages
  1448--1458.

\bibitem[Western and Rosenfeld, 2011]{western2011unions}
Western, B. and Rosenfeld, J. (2011).
\newblock Unions, norms, and the rise in us wage inequality.
\newblock {\em American Sociological Review}, 76(4):513--537.

\bibitem[Williams, 2012]{williams2012using}
Williams, R. (2012).
\newblock Using the margins command to estimate and interpret adjusted
  predictions and marginal effects.
\newblock {\em The Stata Journal}, 12(2):308--331.

\bibitem[Yao and Li, 2014]{yao2014new}
Yao, W. and Li, L. (2014).
\newblock A new regression model: modal linear regression.
\newblock {\em Scandinavian Journal of Statistics}, 41(3):656--671.

\bibitem[Yao et~al., 2012]{yao2012}
Yao, W., Lindsay, B., and Li, R. (2012).
\newblock Local modal regression.
\newblock {\em Journal of Nonparametric Statistics}, 24(3):647--663.

\bibitem[Yao and Xiang, 2016]{yao2016nonparametric}
Yao, W. and Xiang, S. (2016).
\newblock Nonparametric and varying coefficient modal regression.
\newblock {\em arXiv preprint arXiv:1602.06609}.

\end{thebibliography}

\end{document}